\documentclass[10pt]{amsart}

\usepackage[utf8]{inputenc}
\usepackage{hyperref}
\hypersetup{
     colorlinks = true,
     linkcolor = magenta,
     anchorcolor = magenta,
     citecolor = magenta,
     filecolor = magenta,
     urlcolor = magenta
     }

\usepackage{geometry}
\usepackage{graphicx}
\usepackage{amssymb}
\usepackage{amsthm}
\usepackage[all,cmtip]{xy}
\usepackage{color}

\newcommand{\Z}{ {\mathbb{Z}} } 
\newcommand{\Q}{ {\mathbb{Q}} }
\newcommand{\R}{ {\mathbb{R}} }

\newcommand{\C}{ {\mathbb{C}} }

\newcommand{\A}{ {\mathbb{A}} }

\newcommand{\Ad}{{ \mathrm{Ad}}}

\newcommand{\GL}{\mathrm{GL}}
\newcommand{\PGL}{\mathrm{PGL}}
\newcommand{\SL}{\mathrm{SL}}

\newcommand{\GSp}{\mathrm{GSp}}
\newcommand{\PGSp}{\mathrm{PGSp}}
\newcommand{\SO}{\mathrm{SO}}

\newtheorem{theorem}{Theorem}[section]

\newtheorem{lemma}[theorem]{Lemma}

\newtheorem{remark}[theorem]{Remark}
\newtheorem{corollary}[theorem]{Corollary}

\newtheorem{proposition}[theorem]{Proposition}

\title[Automorphic $\SL_2$-periods and the subconvexity problem for $\GL_2 \times \GL_3$]{Automorphic $\SL_2$-periods and the subconvexity problem for $\GL_2 \times \GL_3$}
\author{Aprameyo Pal}
\address{A. Pal: Universit\"at Duisburg-Essen. Fakult\"at f\"ur Mathematik, Thea-Leymann-Strasse 9, 45127 Essen, Germany.}
\email{aprameyo.pal@uni-due.de}

\author{Carlos de Vera-Piquero}
\address{C. de Vera-Piquero: Universitat Politècnica de Catalunya. Departament de Matemàtiques, C. Jordi Girona 1-3, 08034 Barcelona, Spain.}
\email{cdeverapiquero@gmail.com}

\date{May 2020}

\begin{document}
\begin{abstract}
We prove a new (conditional) result towards the subconvexity problem for certain automorphic $L$-functions for $\GL_2\times \GL_3$. This follows from the computation of new $\SL_2$-period integrals associated with newforms $f$ and $g$ of even weight and odd squarefree level. The same computations also lead to a central value formula for degree $6$ complex $L$-series of the form $L(f\otimes \mathrm{Ad}(g), s)$, extending previous work in \cite{PaldVP}.
\end{abstract}

\maketitle

\section{Introduction}

Let $f \in S_{2k}(N_f)$ and $g \in S_{\ell+1}(N_g)$ be two normalized newforms of weight $2k$ and $\ell+1$, and level $\Gamma_0(N_f)$ and $\Gamma_0(N_g)$, respectively. We assume throughout that $\ell \geq k \geq 1$ are both odd integers, and that the levels $N_f$ and $N_g$ are both squarefree and odd. We set $\ell-k = 2m$, with $m\geq 0$. We emphasize that we consider level structure of $\Gamma_0$-type, hence both $f$ and $g$ have trivial Nebentype character.

Associated with $f$ and $g$, one has a degree $6$ complex $L$-series
\[
L(f \otimes \Ad(g), s),
\]
which is the Artin $L$-series corresponding to the tensor product $V(f) \otimes \Ad(V(g))$ of the (compatible system of $p$-adic) Galois representation(s) attached to $f$ and the adjoint of the one attached to $g$. This $L$-series admits a representation as an Euler product
\[
L(f\otimes \mathrm{Ad}(g), s) = \prod_{p} L_p(f \otimes\mathrm{Ad}(g), s),
\]
where $p$ varies over all rational primes. For example, if $p$ is a rational prime not dividing $N_fN_g$, and $\{\alpha_p, \alpha_p^{-1}\}$ and $\{\beta_p,\beta_p^{-1}\}$ are the Satake parameters of $f$ and $g$ at $p$, respectively, so that 
\[
1-a_f(p)X + p^{2k-1}X^2 = (1-p^{k-1/2}\alpha_p X)(1-p^{k-1/2}\alpha_p^{-1}X),
\]
\[
1-a_g(p)X + p^{\ell}X^2 = (1-p^{\ell/2}\beta_p X)(1-p^{\ell/2}\beta_p^{-1}X),
\]
then one has
\[
L_p(f\otimes\mathrm{Ad}(g),s) = \det(\mathbf 1_6 - A_p \otimes B_p \cdot p^{-s-\ell})^{-1},
\]
where we put
\[
A_p = p^{k-1/2}\begin{pmatrix} \alpha_p & 0 \\ 0 & \alpha_p^{-1} \end{pmatrix}, \qquad 
B_p = p^{\ell}\begin{pmatrix} \beta_p^2 & 0 & 0 \\ 0 & 1 & 0 \\ 0 & 0 & \beta_p^{-2} \end{pmatrix}.
\]
The above Euler product converges absolutely for $\mathrm{Re}(s)\gg 0$, and the completed $L$-series 
\[
\Lambda(f\otimes\mathrm{Ad}(g),s) = L_{\infty}(f\otimes\mathrm{Ad}(g),s)\prod_{p} L_p(f \otimes\mathrm{Ad}(g), s),
\]
where $L_{\infty}(f\otimes\mathrm{Ad}(g),s) := \Gamma_{\C}(s)\Gamma_{\C}(s+\ell)\Gamma_{\C}(s+\ell-2k+1)$, $\Gamma_{\C}(s) := 2(2\pi)^{-s}\Gamma(s)$, has analytic continuation to the whole complex plane and satisfies the functional equation relating its values at $s$ and $2k-s$, with center of symmetry at $s = k$. In our previous paper \cite{PaldVP}, we proved an explicit central value formula for $\Lambda(f\otimes \mathrm{Ad}(g),k)$ under certain hypotheses, extending a previous formula of Ichino \cite{Ichino-pullbacks}. Such formula was obtained by making explicit a decomposition formula due to Qiu \cite{Qiu} for a certain automorphic $\SL_2$-period, and in classical terms it involves a half-integral weight modular form $h \in S_{k+1/2}^+(N_f)$ in Shimura--Shintani correspondence with $f$ and its Saito--Kurokawa lift. The purpose of this note is two-fold: on one hand, we generalize the central value formula in \cite{PaldVP}, and on the other hand, we make some progress towards the subconvexity problem for automorphic $L$-functions for $\GL_2 \times \GL_3$. For both goals we need new computations of local $\SL_2$-periods, and for the second one we also use recent work of Nelson \cite{Nelson}.

To be more precise, let $\pi$ and $\tau$ be the automorphic representations of $\GL_2(\A)$ (actually, of $\PGL_2(\A)$) associated with $f$ and $g$, respectively. In addition, let $\psi$ denote the standard additive character of $\A/\Q$, $\omega_{\bar{\psi}}$ be the Weil representation of the metaplectic group $\widetilde{\SL}_2(\A)$ on the space $\mathcal S(\A)$ of Bruhat--Schwartz functions (on the one dimensional quadratic space with bilinear form $(x,y)= xy/2$) with respect to $\bar{\psi}=\psi^{-1}$, and $\tilde{\pi} \in \mathrm{Wald}_{\bar{\psi}}(\pi)$ be an automorphic representation of $\widetilde{\SL}_2(\A)$ belonging to the {\em Waldspurger packet} of $\pi$ with respect to $\bar{\psi}$. Associated with the triple $\tilde{\pi}, \tau, \omega_{\bar{\psi}}$, Qiu defines a natural automorphic $\SL_2$-period functional
\[
\mathcal Q: \tilde{\pi}\otimes \tilde{\pi} \otimes \tau \otimes \tau \otimes \omega_{\bar{\psi}}\otimes \omega_{\bar{\psi}} \, \longrightarrow \, \C
\]
and studies its main features. Most importantly, he shows that when $\mathcal Q$ is not identically zero, then it decomposes as a product of local $\SL_2$-periods 
\[
\mathcal I_v: \tilde{\pi}_v \otimes \tilde{\pi}_v \otimes \tau_v \otimes \tau_v \otimes \omega_{\bar{\psi},v} \otimes \omega_{\bar{\psi},v} \, \longrightarrow \, \C
\]
up to certain $L$-values, including the central value $\Lambda(f \otimes\mathrm{Ad}(g), k)$. The non-vanishing of $\mathcal Q$ is well-understood, and it is equivalent to the central value $\Lambda(f\otimes\mathrm{Ad}(g),k)$ being non-zero together with some local conditions on the choice of $\tilde{\pi}$ in $\mathrm{Wald}_{\bar{\psi}}(\pi)$. With this, the strategy followed in \cite{PaldVP} consists in finding a {\em test vector} on which $\mathcal Q$ does not vanish, and then evaluating both the global period $\mathcal Q$ and the local periods $\mathcal I_v$ at such vector. From Qiu's decomposition formula, one can then isolate the desired central value  $\Lambda(f\otimes\mathrm{Ad}(g),k)$.

The assumptions made in \cite{PaldVP}, mainly that $N_g = N_f$ and $\ell = k$, simplified the still involved computations of the local periods $\mathcal I_v$, as well as the evaluation of the global period itself. Both of these assumptions can be relaxed, leading to the following result:

\begin{theorem}\label{thm:centralvalue-intro}
Suppose that $N_f, N_g$ are both odd and squarefree, and that $N_g \mid N_f$. Suppose that $\ell \geq k \geq 1$ are both odd, and set $\ell-k = 2m$. If the Atkin--Lehner eigenvalue of $f$ at $p$ is $+1$ at all primes $p$ dividing $M_g:=N_f/N_g$, then there exists a half-integral weight modular form $h \in S_{k+1/2}^+(N_f)$ in Shimura--Shintani correspondence with $f$ such that 
\[
\Lambda(f \otimes \mathrm{Ad}(g), k) = 2^{6m+k+1-\nu(M_g)} C_0(f,g)C_{\infty}(f,g) \cdot \frac{\langle f,f\rangle}{\langle h,h\rangle} \frac{|\langle \breve{F}_{|\mathcal H\times\mathcal H}, g \times V_{M_g}g\rangle|^2}{\langle g,g \rangle^2},
\]
where $\breve{F} \in S_{\ell+1}^{nh}(\Gamma_0^{(2)}(N_f))$ is a nearly holomorphic Siegel form closely related to the Saito--Kurokawa lift of $h$ (cf. Proposition \ref{prop:Ftheta}), $\nu(M_g)$ denotes the number of primes dividing $M_g$, and $C_0(f,g)$ and $C_{\infty}(f,g)$ are non-zero rational constants that depend on the levels and weights of $f$ and $g$, respectively (cf. Theorem \ref{thm:centralvalue} for their explicit value).
\end{theorem}

When $N_g = N_f$ and $\ell = k$, one has $m = 0$ and $C_{\infty}(f,g) = 1$, $\breve{F} = \mathrm{SK}(h) \in S_{k+1}(\Gamma_0^{(2)}(N_f))$ is the Saito--Kurokawa lift of $h$, and the above formula recovers \cite[Theorem 1.1]{PaldVP} (assuming in loc. cit. that $g$ has trivial Nebentype character, see Remark 1.2 in op. cit.). If in addition $N_g = N_f = 1$, then it recovers the original formula of Ichino \cite{Ichino-pullbacks}. We also point out that the case $N_g=N_f$ and $\ell\geq k$ has been considered in \cite{Chen}, by extending Ichino's approach instead of using Qiu's strategy via $\SL_2$-periods. Finally, we point out that the same strategy used to prove Theorem \ref{thm:centralvalue-intro} can be pushed further to remove the assumption $N_g \mid N_f$, at the cost of performing more local computations. Note, however, that this mild extension would not improve Theorem \ref{thm:sc-intro} below.

In the above statement, the Petersson products $\langle f, f\rangle$, $\langle g, g\rangle$ are defined as usual, namely 
\[
\langle f, f \rangle := \mu_{N_f}^{-1} \int_{\Gamma_0(N_f)\backslash\mathcal H} \hspace{-0.1cm} f(z) \overline{f(z)} y^{2k-2}dz, \qquad
\langle g, g \rangle := \mu_{N_g}^{-1} \int_{\Gamma_0(N_g)\backslash\mathcal H} \hspace{-0.1cm} g(z) \overline{g(z)} y^{\ell-1}dz,
\]
where $z = x + \sqrt{-1}y$ and $\mu_t = [\SL_2(\Z):\Gamma_0(t)]$ for $t \in \Z_{\geq 1}$. For the half-integral weight modular form $h$ we similarly have
\[
\langle h, h \rangle := \mu_{4N_f}^{-1} \int_{\Gamma_0(4N_f)\backslash\mathcal H} \hspace{-0.1cm} h(z) \overline{h(z)} y^{k-3/2}dz.
\]
Finally, the Petersson product $\langle \breve{F}_{|\mathcal H \times\mathcal H}, g\times V_{M_g}g\rangle$ is defined by (notice that $N_f = \mathrm{lcm}(N_f,N_g)$ because of our assumption that $N_g \mid N_f$)
\[
\langle \breve{F}_{|\mathcal H \times\mathcal H}, g\times V_{M_g}g\rangle := \mu_{N_f}^{-2} \int_{\Gamma_0(N_f)\backslash\mathcal H}\int_{\Gamma_0(N_f)\backslash\mathcal H} \breve{F}\left(\begin{pmatrix}z_1 & 0 \\ 0 & z_2\end{pmatrix}\right)\overline{g(z_1)g(M_gz_2)}y_1^{\ell-1}y_2^{\ell-1} dz_1dz_2.
\]

As we have already pointed out above, the above theorem is an extension of the main result of \cite{PaldVP}. The proof requires to extend both the global and local computations involved in our strategy of making explicit Qiu's decomposition formula. Special attention in this paper is deserved to the local side, because the new computations of local $\SL_2$-periods $\mathcal I_v$ at a specific test vector done in this note, together with those already carried out in \cite{PaldVP}, allow us to derive new advances in the subconvexity problem for $\GL_2 \times \GL_3$ by using recent work of Nelson \cite{Nelson}. This is the most interesting novelty of this paper, and also the main motivation that led us to write this note. Namely, in Section \ref{sec:subconvexity} we address the {\em subconvexity problem} for automorphic $L$-functions\footnote{In analogy with classical $L$-series, we follow the convention that automorphic $L$-functions $L(\Pi,s)$ refer always to the finite part of the $L$-function, omitting the $\Gamma$-factors at the archimedean place. When including such factors, we will write $\Lambda(\Pi,s)$.} 
\begin{equation}\label{scproblem}
L(\pi \otimes \mathrm{ad}(\tau),s), \qquad \pi \text{ on } \PGL_2 \textrm{ fixed, } \tau \text{ on } \GL_2 \text{ varying.}
\end{equation}
This problem consists in establishing a {\em subconvex bound} for $L(\pi \otimes \mathrm{ad}(\tau),1/2)$ when $\pi$ is a fixed automorphic representation of $\PGL_2(\A)$ and $\tau$ varies in a family $\mathcal G$ of automorphic representations of $\GL_2$, i.e. proving the existence of constants $c = c(\mathcal G)\geq 0$ and $\delta = \delta(\mathcal G) > 0$ such that 
\begin{equation}\label{scbound}
|L(\pi \otimes \mathrm{ad}(\tau),1/2)| \leq c C(\pi \otimes\mathrm{ad}(\tau))^{1/4-\delta}
\end{equation}
for all $\tau \in \mathcal G$, where $C(\pi \otimes\mathrm{ad}(\tau)) \in \R_{\geq 1}$ denotes the analytic conductor of $\pi \otimes\mathrm{ad}(\tau)$. The inequality analogous to \eqref{scbound} with $\delta = 0$ is the so-called {\em convex bound}, and can be obtained by using the Phragmen--Lindel\"of principle. Therefore, establishing a subconvex bound requires to break this barrier and improve the convex bound. 
Interest in subconvexity problems as the above one relies on their relation to fundamental arithmetic equidistribution questions. In the case of \eqref{scproblem}, it has applications towards the limiting mass distribution of automorphic forms, also known as the `arithmetic quantum unique ergodicity' (see \cite{SarnakAQC}, \cite{IwaniecSarnak-Perspectives}, \cite{HolowinskySoundarajan}, \cite{NelsonPitaleSaha}, \cite{Sarnak-RP}).

Our contribution to the subconvexity problem in \eqref{scproblem}, under some assumptions on the family $\mathcal G$, follows from the observation that our computations of local $\SL_2$-periods provide the required bounds in recent work of Nelson \cite{Nelson} concerning this subconvexity problem. And it is important to remark that the local $\SL_2$-periods computed in \cite{PaldVP} alone would not have been enough to improve Nelson's result. Let us illustrate in this introduction an easy but relevant example in which we can push Nelson's result one step further in the above subconvexity problem, referring the reader to Section \ref{sec:subconvexity} for a more detailed and general statement. 

In line with our notation above, fix at the outset an odd integer $\ell \geq 1$, and let $q$ traverse an infinite increasing sequence $\mathfrak Q$ of (odd) prime numbers. For each prime $q \in \mathfrak Q$, choose a newform $g \in S_{\ell+1}^{new}(q)$ of weight $\ell+1$ and level $\Gamma_0(q)$, and let $\mathcal G$ be the infinite collection of all the automorphic representations $\tau = \tau(g)$ of $\PGL_2(\A)$ associated with the newforms $g$ as $q$ varies. We assume the following hypothesis on the family $\mathcal G$, which is the existence of a subconvex bound for $L(\tau \otimes \tau\otimes\chi, 1/2)$ in the $\tau$-aspect with polynomial dependence upon the Hecke character $\chi$:

\vspace{0.2cm}

\noindent {\bf Hypothesis:} there exist absolute constants $c_0, A_0 \geq 0$, $\delta_0 > 0$ such that for all $\tau \in \mathcal G$ and all unitary characters $\chi$ of $\A^{\times}/\Q^{\times}$, one has
\[
|L(\tau \otimes \tau \otimes\chi, 1/2)| \leq c_0 C(\tau \otimes \tau\otimes\chi)^{1/4-\delta_0}C(\chi)^{A_0}.
\]

\vspace{0.2cm}

The following statement is a particular instance of Theorem \ref{thm:sc} in Section \ref{sec:subconvexity}, which strengthens \cite[Theorem 1]{Nelson} in the sense that we allow $f$ to have arbitrary (odd) squarefree level instead of level $1$. Modulo the above hypothesis, the main novelty of the following result is precisely that we remove the assumption on $f$ having trivial level.

\begin{theorem}\label{thm:sc-intro}
With the above notation, assume that the family $\mathcal G$ satisfies the above hypothesis. Then, there exist absolute constants $c, A \geq 0$ and $\delta > 0$ such that 
\[
L(\pi \otimes \mathrm{ad}(\tau), 1/2) \leq c C(\pi\otimes\mathrm{ad}(\tau))^{1/4-\delta} C(\pi)^A
\]
for all $\tau \in \mathcal G$ and every automorphic representation $\pi = \pi(f)$ of $\PGL_2(\A)$ associated with a newform $f \in S_{2k}^{new}(N_f)$ of weight $2k$, with $1 \leq k \leq \ell$ an odd integer, and odd squarefree level $N_f$.
\end{theorem}

Observe that we have omitted the absolute value on the left hand side of the inequality in the statement. This is because the $L$-value $L(\pi\otimes\mathrm{ad}(\tau),1/2)=L(f\otimes\mathrm{Ad}(g),k)$ is non-negative (this can be deduced from Theorem \ref{thm:centralvalue-intro}).

Concerning the emphasized hypothesis in the above theorem, we note that via the factorization 
\[
L(\tau \otimes \tau \otimes \chi, 1/2) = L(\chi,1/2)L(\mathrm{ad}(\tau)\otimes\chi, 1/2),
\]
the subconvexity problem for $L(\tau \otimes\tau \otimes\chi,s)$ can be reduced to that for $L(\mathrm{ad}(\tau)\otimes\chi, s)$ (with $\tau$ varying and $\chi$ fixed), and a proof for the latter problem was announced by Munshi \cite{Munshi} (at least when $\chi$ is trivial). Assuming the existence of a subconvex bound for $L(\mathrm{ad}(\tau)\otimes\chi, s)$ with $\tau$ varying and $\chi$ fixed, our hypothesis would be fulfilled and could be dropped from the theorem. In addition, Theorem \ref{thm:sc-intro} would (unconditionally) lead to strong quantitative forms of the arithmetic quantum unique ergodicity conjecture in the prime level aspect.

\vspace{0.2cm}

Let us close this introduction by briefly describing the organization of the paper. Section \ref{sec:notation} below collects general notation that is used through all the text, mainly about metaplectic groups and orthogonal groups. In Section \ref{sec:theta}, we recall some general notions on quadratic spaces and theta correspondences, with special attention to the cases needed for our purposes. In Section \ref{sec:SL2periods-thm} we explain in more detail the strategy to prove Theorem \ref{thm:centralvalue-intro}, and state again the result in more precise terms (see Theorem \ref{thm:centralvalue}). After that, Sections \ref{sec:testvector} and \ref{sec:localperiods} are devoted to describe our choice of test vector and to compute the local periods alluded to above. Section \ref{sec:localperiods} deserves special attention, since the local period computations therein (together with those in \cite{PaldVP}) are the key ingredient for our application in Section \ref{sec:subconvexity} towards the subconvexity problem for automorphic $L$-functions for $\GL_2 \times \GL_3$ and the proof of Theorem \ref{thm:sc-intro} above (which is a particular case of the more general version in  Theorem \ref{thm:sc}). Section \ref{sec:global} is devoted to complete the proof of the central value formula stated in Theorem \ref{thm:centralvalue}, and can be skipped by the reader interested in the subconvexity problem considered in Theorem \ref{thm:sc-intro}.

\vspace{0.2cm}

\noindent {\bf Acknowledgements.} We are pleased to thank the referee for her/his valuable suggestions that helped to improve the exposition of this article. During the elaboration of this work, A. Pal was financially supported by the SFB/TR 45 ``Periods, Moduli Spaces, and Arithmetic of Algebraic Varieties'', and C. de Vera-Piquero received financial support from a Junior Research Grant (through AGAUR PDJ 2012) and also from the European Research Council (ERC) under the European Union's Horizon 2020 research and innovation programme (grant agreement No 682152).

\section{Notation}\label{sec:notation}

Through all the paper, we write $\A = \A_{\Q}$ for the ring of adeles over $\Q$. We write $\zeta(s)$ for Riemann's zeta function, admitting the usual Euler product representation $\zeta(s) = \prod_p \zeta_p(s)$ for $\mathrm{Re}(s) > 1$, where $\zeta_p(s) = (1-p^{-s})^{-1}$. We write $\zeta_{\Q}(s)$ for the completed Riemann zeta function given by $\zeta_{\Q}(s) := \Gamma_{\R}(s)\zeta(s)$, where $\Gamma_{\R}(s) := \pi^{-s/2}\Gamma(s)$ and $\Gamma(s)$ is the usual Gamma function. We will also use $\Gamma_{\C}(s) := 2(2\pi)^{-s}\Gamma(s)$.

If $r, M \geq 1$ are integers, and $\psi$ is a Dirichlet character modulo $M$, we denote by $S_r(M,\psi)$ the (complex) space of cusp forms of weight $r$, level $M$ and character $\psi$. When $\psi$ is trivial, we just write $S_r(M)$ or $S_r(\Gamma_0(M))$, where $\Gamma_0(M) \subseteq \SL_2(\Z)$ is the usual Hecke congruence subgroup of level $M$. Similarly $S_r(\Gamma_0^{(2)}(M))$ will stand for the (complex) space of Siegel forms of degree $2$ and weight $r$ for the Hecke congruence subgroup $\Gamma_0^{(2)}(M) \subseteq \mathrm{Sp}_2(\Z)$ of level $M$.

If $M\geq 1$ is an odd integer and $k\geq 0$ is an integer, we write $S_{k+1/2}(4M,(\frac{4}{\cdot}))$ for the (complex) space of cusp forms of half-integral weight $k+1/2$, level $4M$ and character $(\frac{4}{\cdot})$, in the sense of Shimura \cite{Shimura73}. We denote by $S_{k+1/2}^+(M)$ Kohnen's plus subspace in $S_{k+1/2}(4M,(\frac{4}{\cdot}))$ consisting of those forms $h$ whose $q$-expansion has the form 
\[
h = \sum_{\substack{n\geq 1,\\(-1)^kn\equiv 0,1 \, (4)}} c(n) q^n.
\]
We refer the reader to \cite{Kohnen-newforms} for a careful study of these spaces.

If $v$ is a place of $\Q$, we write $\widetilde{\SL}_2(\Q_v)$ for the metaplectic double cover of $\SL_2(\Q_v)$, and similarly, we denote by $\widetilde{\SL}_2(\A)$ the metaplectic double cover of $\SL_2(\A)$. We will identify $\widetilde{\SL}_2(\Q_v)$, resp. $\widetilde{\SL}_2(\A)$, with $\SL_2(\Q_v) \times \{\pm 1\}$, resp. $\SL_2(\A) \times \{\pm 1\}$, where the product is given by the rule 
\[
 [g_1, \epsilon_1][g_2, \epsilon_2] = [g_1g_2, \epsilon_v(g_1,g_2)\epsilon_1\epsilon_2] \qquad (\text{resp. } [g_1, \epsilon_1][g_2, \epsilon_2] = [g_1g_2, \epsilon(g_1,g_2)\epsilon_1\epsilon_2]).
\]
At each place $v$, $\epsilon_v(g_1,g_2)$ is defined as follows. First one defines $x: \SL_2(\Q_p) \to \Q_p$ by
\[
 g = \left(\begin{smallmatrix} a & b \\ c & d\end{smallmatrix}\right) \longmapsto x(g) = \begin{cases}
 c & \text{if } c \neq 0,\\
 d & \text{if } c = 0;
 \end{cases}
\]
then, $\epsilon_v(g_1,g_2) = (x(g_1)x(g_1g_2), x(g_2)x(g_1g_2))_v$. When $g_1, g_2 \in \SL_2(\A)$, we set $\epsilon(g_1,g_2) = \prod_v \epsilon_v(g_1,g_2)$. When $v = \infty$, we put $s_{\infty}(g) = 1$ for all $g \in \SL_2(\R)$, and when $v = p$ is a finite place, we set
\[
 s_p\left(\left(\begin{smallmatrix} a & b \\ c & d\end{smallmatrix}\right)\right) = \begin{cases}
                (c,d)_p & \text{if } cd \neq 0, \, \mathrm{ord}_p(c) \text{ odd}, \\
                1 & \text{otherwise},
               \end{cases}
\]
for $g = \left(\begin{smallmatrix} a & b \\ c & d\end{smallmatrix}\right) \in \SL_2(\Q_p)$. 
If $p$ is an odd prime, then the homomorphism $g \mapsto [g, s_p(g)]$ gives a splitting of $\widetilde{\SL}_2(\Q_p)$ over the maximal compact subgroup $\SL_2(\Z_p)$, while for $p = 2$ this is only a splitting over $\Gamma_1(4;\Z_2) \subset \SL_2(\Z_2)$. If $p$ is an odd prime (resp. if $p = 2$), and $G$ is a subgroup of $\SL_2(\Z_p)$ (resp. of $\Gamma_1(4;\Z_2)$), then we will write $\tilde G \subseteq \widetilde{\SL}_2(\Z_p)$ for the image of $G$ under the previous splitting. We will also regard $\SL_2(\Q)$ as a subgroup of $\widetilde{\SL}_2(\A)$ through the homomorphism $g \mapsto [g, \prod_v s_v(g)]$.

When working in $\SL_2(\Q_v)$ (or $\SL_2(\A)$), we will often use the notation
\[
u(x) = \begin{pmatrix} 1 & x \\ 0 & 1\end{pmatrix}, \quad t(a) = \begin{pmatrix} a & 0 \\ 0 & a^{-1}\end{pmatrix}, \quad s = \begin{pmatrix} 0 & 1 \\ -1 & 0\end{pmatrix},
\]
for $x \in \Q_v$ (or $\A$) and $a \in \Q_v^{\times}$ (or $\A^{\times}$).

If $V$ is a finite-dimensional quadratic space over $\Q$, with bilinear form $(\,,\,)$, and $\psi$ is an additive character of $\A/\Q$, we equip $V(\A)$ with the Haar measure which is self-dual with respect to $\psi$, unless otherwise stated. In other words, we consider the Haar measure such that $\mathcal F(\mathcal F(\phi))(x) = \phi(-x)$, where $\mathcal F(x) = \int_{V(\A)} \phi(y) \psi((x,y))dy$ is the Fourier transform of $\phi$. We note that the orthogonal group $\mathrm{O}(V)$ is {\em not} connected, and choose a measure on $\mathrm{O}(V)(\A)$ as follows. First, we equip $\mathrm{SO}(V)(\A)$ with the Tamagawa measure. Secondly, at each place $v$ we extend the local measure on $\mathrm{SO}(V)(\Q_v)$ to the non-identity component of $\mathrm{O}(V)(\Q_v)$. And finally, we consider the measure $dh_v$ on $\mathrm{O}(V)(\Q_v)$ to be half of this extended measure, and define $dh = \prod_v dh_v$. This is the Tamagawa measure on $\mathrm{O}(V)(\A)$, and $[\mathrm{O}(V)] = \mathrm{O}(V)(\Q) \backslash \mathrm{O}(V)(\A)$ has volume $1$ with respect to $dh$. If $\mathcal S(V(\A))$ denotes the space of Bruhat--Schwartz functions on $V(\A)$, and $\phi_1$, $\phi_2 \in \mathcal S(V(\A))$, we set $\langle \phi_1,\phi_2\rangle = \int_{V(\A)} \phi_1(x)\overline{\phi_2(x)} dx$, where $dx$ is the Haar measure that is self-dual with respect to $\psi$. If $\pi$ is an irreducible cuspidal unitary representation of $G(\A)$, and $f_1, f_2 \in \pi$, we define the pairing $\langle f_1, f_2 \rangle$ to be:
\begin{itemize}
 \item[i)] $\int_{[\SL_2]} f_1(g) \overline{f_2(g)} dg$, if $G = \widetilde{\SL}_2$;
 \item[ii)] $\int_{[\PGL_2]} f_1(g) \overline{f_2(g)} dg$, if $G = \GL_2$;
 \item[iii)] $\int_{[G]} f_1(g) \overline{f_2(g)} dg$, if $G = \mathrm{SO}(V)$ or $\mathrm{O}(V)$.
\end{itemize}

\section{Quadratic spaces and theta correspondences}\label{sec:theta}

\subsection{Quadratic spaces}

Let $F$ be a field with $\mathrm{char}(F)\neq 2$, and $V$ be a quadratic space over $F$, i.e. a finite-dimensional vector space over $F$ equipped with a non-degenerate symmetric bilinear form $(\, , \,)$. We denote by $Q$ the associated quadratic form on $V$, given by 
\[
 Q(x) = \frac{1}{2}(x,x), \quad x \in V.
\]
If $m = \mathrm{dim}(V)$, fixing a basis $v_1,\dots v_m$ of $V$ and identifying $V$ with the space of column vectors $F^m$, the bilinear form $(\, , \,)$ determines a matrix (that we still denote with the same letter) $Q \in \GL_m(F)$ by setting $Q = ((v_i,v_j))_{i,j}$. Then we have $(x,y) = {}^t x Q y \quad \text{for } x,y \in V$. We define $\det(V)$ to be the image of $\det(Q)$ in $F^{\times}/(F^{\times})^2$. The orthogonal similitude group of $V$ is 
\[
 \mathrm{GO}(V) = \{h \in \GL_m: {}^t hQh = \nu(h)Q, \, \nu(h) \in \mathbb G_m\},
\]
and $\nu: \mathrm{GO}(V) \to \mathbb G_m$ is the similitude morphism (also called scale map). From the very definition, observe that $\det(h)^2 = \nu(h)^m$ for every $h \in \mathrm{GO}(V)$. If $m$ is even, we also set 
\[
 \mathrm{GSO}(V) = \{h \in \mathrm{GO}(V): \det(h) = \nu(h)^{m/2}\}.
\]
Finally, we let $\mathrm O(V) = \ker(\nu)$ denote the orthogonal group of $V$, and write $\mathrm{SO}(V) = \mathrm O(V) \cap \SL_m$ for the special orthogonal group.

\subsection{Explicit realizations in low rank}\label{spaces:lowrank}

We are particularly interested in orthogonal groups for vector spaces of dimension $3$, $4$ and $5$. We fix in this paragraph the explicit realizations that will be used later on to describe automorphic representations for $\mathrm{SO}(V)(\A)$ and $\mathrm{GSO}(V)(\A)$. We keep the same choices as in \cite{PaldVP}, which follow quite closely the ones in \cite{Ichino-pullbacks, Qiu}.

When $\mathrm{dim}(V) = 3$, there exist a unique quaternion algebra $B$ over $F$ and an element $a \in F^{\times}$ such that $(V,q) \simeq (V_B, aq_B)$, where $V_B = \{x \in B: \mathrm{Tr}_B(x) = 0\}$ is the subspace of elements in $B$ with zero trace, and $q_B(x) = -\mathrm{Nm}_B(x)$. The group of invertible elements $B^{\times}$ acts on $V_B$ by conjugation, and this action gives rise to an isomorphism 
\[
 PB^{\times} \, \stackrel{\simeq}{\longrightarrow} \, \mathrm{SO}(V_B,q_B) \simeq \mathrm{SO}(V,q).
\]
When $B = \mathrm{M}_2$ is the split algebra of $2$-by-$2$ matrices, $PB^{\times} = \PGL_2$ and the above identifies $\PGL_2$ with the special orthogonal group of a three-dimensional quadratic space.

In dimension $4$, we consider the vector space $V_4 := \mathrm{M}_2(F)$ of $2$-by-$2$ matrices, equipped with the quadratic form $q(x) = \det(x)$. The associated non-degenerate bilinear form is $(x,y) = \mathrm{Tr}(xy^*)$, where 
\[
 x^* = \left(\begin{array}{cc} x_4 & -x_2 \\ -x_3 & x_1\end{array}\right) \quad \text{for } x = \left(\begin{array}{cc} x_1 & x_2 \\ x_3 & x_4\end{array}\right) \in \mathrm M_2(F).
\]
There is an exact sequence 
\begin{equation}\label{exseq:GSO4}
 1 \, \longrightarrow \, \mathbb G_m \, \stackrel{\iota}{\longrightarrow} \, \GL_2 \times \GL_2 \, \stackrel{\rho}{\longrightarrow} \mathrm{GSO}(V_4) \, \longrightarrow 1,
\end{equation}
where $\iota(a) = (a\mathbf 1_2, a^{-1}\mathbf 1_2)$, $\rho(h_1,h_2)x = h_1xh_2^*$. One has $\nu(\rho(h_1,h_2)) = \det(h_1h_2) = \det(h_1)\det(h_2)$. In particular, when $F$ is a number field, automorphic representations of $\mathrm{GSO}(V_4)$ can be seen as automorphic representations of $\GL_2 \times \GL_2$ through the homomorphism $\rho$ in the above short exact sequence. Here we warn the reader that our choice for $\rho$ in \eqref{exseq:GSO4} agrees with the one on \cite{Qiu} and \cite{GanTakeda}, but differs from the one considered in \cite{Ichino-pullbacks} (or \cite{IchinoIkeda}), leading to a slightly different model for $\mathrm{GSO}(V_4)$.

Finally, in dimension $5$ we will describe a realization of $\mathrm{SO}(3,2)$, the special orthogonal group of a $5$-dimensional quadratic space $(V,q)$ of Witt index $2$. Although the isomorphism class of such a quadratic space depends on $\det(V)$, the group $\mathrm{SO}(V,q)$ does not. We describe a model $V_5$ of such a quadratic space with determinant $1$ (modulo $F^{\times,2}$). Namely, start considering the $4$-dimensional space $F^4$ of column vectors, on which $\GSp_2 \subset \GL_4$ acts on the left. Let $ e_1 = {}^t(1,0,0,0)$, $\dots$, $e_4 = {}^t(0,0,0,1)$ denote the standard basis on $F^4$, and equip $\tilde V := \wedge^2 F^4$ with the non-degenerate symmetric bilinear form $(\, , \,)$ defined by the rule
\[
 x \wedge y = (x,y) \cdot (e_1 \wedge e_2 \wedge e_3 \wedge e_4), \quad \text{for all } x, y \in \tilde V.
\]
Set $x_0 := e_1 \wedge e_3 + e_2 \wedge e_4$, and define the $5$-dimensional subspace $V_5 \subset \tilde V$ to be the orthogonal complement of the span of $x_0$, i.e.
\[
 V_5 := \{x \in \tilde V: (x,x_0) = 0\}.
\]
Then the homomorphism $\tilde{\rho}:\GSp_2 \to \SO(\tilde V)$ given by $\tilde{\rho}(h) = \nu(h)^{-1} \wedge^2 (h)$ satisfies $\tilde{\rho}(h)x_0 = x_0$, and therefore induces an exact sequence 
\begin{equation}\label{exseq:SO5}
  1 \, \longrightarrow \, \mathbb G_m \, \stackrel{\iota}{\longrightarrow} \, \GSp_2 \, \stackrel{\rho}{\longrightarrow} \SO(V_5) \, \longrightarrow 1,
\end{equation}
where $\iota(a) = a\mathbf 1_4$ for $a \in \mathbb G_m$. This short exact sequence induces an isomorphism $\PGSp_2 \simeq \SO(V_5)$.

We fix an identification of $V_5$ with the $5$-dimensional space $F^5$ of column vectors by 
\[
 \sum_{i=1}^5 x_i v_i \, \longmapsto \, {}^t(x_1,x_2,x_3,x_4,x_5),
\]
where $v_1 = e_2\wedge e_1, \,\, v_2 = e_1 \wedge e_4, \,\, v_3 = e_1\wedge e_3 - e_2\wedge e_4, \,\, v_4 = e_2\wedge e_3, \,\, v_5 = e_3\wedge e_4$. Upon this identification, we consider the non-degenerate bilinear symmetric form $(\,,\,)$ on $V$ defined by $(x,y) = {}^txQy$ for $x,y\in F^5$, where 
\[
 Q = \left(\begin{array}{ccc} & & -1 \\ & Q_1 \\ -1 \end{array}\right), \quad Q_1 = \left(\begin{array}{ccc} 0 & 0 & 1 \\ 0 & 2 & 0 \\ 1 & 0 & 0 \end{array}\right).
\]

We shall distinguish the $3$-dimensional subspace $V_3 \subset V_5$ spanned by $v_2$, $v_3$, $v_4$, equipped with the bilinear form $(x,y) = {}^txQ_1y$, for $x,y \in F^3$, under the identification $V_3 = F^3$ induced by restricting the above one for $V = F^5$. Notice that $V_5 = \langle v_1\rangle \oplus V_3 \oplus \langle -v_5\rangle$, where $v_1$ and $-v_5$ are isotropic vectors with $(v_1,-v_5) = 1$, and $V_3$ is the orthogonal complement of $\langle v_1,-v_5\rangle = \langle v_1,v_5\rangle$.

Also, we shall distinguish a $4$-dimensional subspace of $V_5$. Indeed, the subspace $\{x \in V: (x,v_3) = 0\} = \langle v_3 \rangle^{\perp} \subset V_5$ is a quadratic $4$-dimensional subspace of $V_5$, and it can be identified with the space $V_4$ defined above by means of the linear map
 \[
  \langle v_3 \rangle^{\perp} \, \longrightarrow \, V_4, \quad x_1 v_2 + x_2 v_1 + x_3 v_5 + x_4 v_4 \, \longmapsto \, \left(\begin{array}{cc} x_1 & x_2\\ x_3 & x_4\end{array}\right).
 \]

By restricting the homomorphism $\rho$ from the exact sequence in \eqref{exseq:GSO4} to 
\[
 G(\SL_2\times\SL_2)^- := \{(h_1,h_2)\in \GL_2\times\GL_2 : \det(h_1)\det(h_2) = 1\} \subseteq \GL_2 \times \GL_2,
\]
one gets an exact sequence 
\begin{equation}\label{exseq:SO4}
  1 \, \longrightarrow \, \mathbb G_m \, \stackrel{\iota}{\longrightarrow} \, G(\SL_2 \times \SL_2)^- \, \stackrel{\rho}{\longrightarrow} \SO(V_4) \, \longrightarrow 1.
\end{equation}
Now notice that $G(\SL_2 \times \SL_2)^-$ is isomorphic to 
\[
  G(\SL_2\times\SL_2) := \{(h_1,h_2)\in \GL_2\times\GL_2 : \det(h_1)\det(h_2)^{-1} = 1\} \subseteq \GL_2 \times \GL_2
\]
through the morphism $(h_1,h_2) \mapsto (h_1,\det(h_2)^{-1}h_2)$. Composing this isomorphism with the natural embedding $G(\SL_2 \times \SL_2) \hookrightarrow \GSp_2$ given by 
\[
  \left(\left(\begin{array}{cc} a_1 & b_1 \\ c_1 & d_1 \end{array}\right), \left(\begin{array}{cc} a_2 & b_2 \\ c_2 & d_2 \end{array}\right)\right) \, \longmapsto \, \left(\begin{array}{cccc} a_1 & 0 & b_1 & 0 \\ 0 & a_2 & 0 & b_2 \\ c_1 & 0 & d_1 & 0 \\ 0 & c_2 & 0 & d_2\end{array}\right),
 \]
one gets a commutative diagram
\[
 \xymatrix{
 1 \ar[r] & \mathbb G_m \ar[r]^{\iota \qquad} \ar@{=}[d] & G(\SL_2 \times \SL_2)^- \ar@{^{(}->}[d] \ar[r]^{\quad \rho} & \mathrm{SO}(V_4) \ar[r] \ar@{^{(}->}[d] & 1 \\
 1 \ar[r] & \mathbb G_m \ar[r]^{\iota} & \GSp_2 \ar[r]^{\rho} & \mathrm{SO}(V_5) \ar[r] & 1
 }
\]
and hence an embedding $\SO(V_4) \subset \SO(V_5)$. This embedding will be of crucial interest later on.

\subsection{Weil representations}\label{sec:Weilreps}

Let now $F$ be a local field with $\mathrm{char}(F)\neq 2$ (for the purposes of this paper, we can think of $F$ being $\Q_v$ for a rational place $v$), and $(V,Q)$ be a quadratic space over $F$ of dimension $m$ as above. Let $\mathcal S(V)$ denote the space of locally constant and compactly supported complex-valued functions on $V$. This is usually referred to as the space of Bruhat--Schwartz functions on $V$. If $F$ is archimedean, we rather consider $\mathcal S(V)$ to be the Fock model (which is a smaller subspace, see \cite[Section 2.1.2]{YZZ}). 

We fix a non-trivial additive character $\psi$ of $F$. The Weil representation $\omega_{\psi,V}$ of $\widetilde{\SL}_2(F) \times \mathrm{O}(V)$ on $\mathcal S(V)$, which depends on the choice of the character $\psi$, is given by the following formulae. If $a \in F^{\times}$, $b \in F$, $h \in \mathrm{O}(V)$, and $\phi \in \mathcal S(V)$, then 
\begin{align*}
 \omega_{\psi,V}(h)\phi(x) & = \phi(h^{-1}x), \\
 \omega_{\psi,V} \left([t(a), \epsilon]\right)\phi(x) & = 
 \epsilon^m \chi_{\psi,V}(a) |a|^{m/2}\phi(ax) \\
 \omega_{\psi,V} \left([u(b),1] \right)\phi(x) & = \psi(Q(x)b)\phi(x), \\
 \omega_{\psi,V} \left([s,1]\right)\phi(x) & = \gamma(\psi,V) \int_V \phi(y)\psi((x,y))dy.
\end{align*}

Here, $\gamma(\psi,V)$ is the Weil index, which is an $8$-th root of unity, and $\chi_{\psi,V}: F^{\times} \to S^1$ is a function satisfying $\chi_{\psi,V}(ab)=(a,b)_F^m\chi_{\psi,V}(a)\chi_{\psi,V}(b)$ for $a,b \in F^{\times}$, where $(\cdot ,\cdot )_F$ denotes the Hilbert symbol. When $m=1$ and $Q(x) = x^2$, we will simply write $\omega_{\psi}, \chi_{\psi}$, and $\gamma(\psi)$ for $\omega_{\psi,V}, \chi_{\psi,V}$, and $\gamma(\psi,V)$, respectively. In this case, the function $\chi_{\psi}$ can be written as  
\[
 \chi_{\psi}(a) = (a,-1)_F \gamma(a,\psi) = (a,-1)_F \frac{\gamma(\psi^a)}{\gamma(\psi)},
\]
where the function $\gamma(\cdot,\psi): F^{\times} \to S^1$ is defined by $\gamma(a,\psi) = \gamma(\psi^a)/\gamma(\psi)$ and satisfies 
\[
\gamma(ab,\psi)=(a,b)_p \gamma(a,\psi) \gamma(b,\psi), \qquad \gamma(ab^2,\psi)=\gamma(a,\psi) \quad \text{for all } a, b \in F^{\times}.
\]
Thus $\chi_{\psi}(ab) = (a,b)_p\chi_{\psi}(a)\chi_{\psi}(b)$ and $\chi_{\psi}(ab^2) = \chi_{\psi}(a)$ for all $a, b \in F^{\times}$, and furthermore $\chi_{\psi^a} = \chi_{\psi} \cdot \chi_a$, where $\chi_a$ stands for the quadratic character $(a, \cdot)_F$. 

For the standard additive character $\psi_p$ of $F=\Q_p$, with $p$ an {\em odd} prime, one has $\gamma(\psi_p) = 1$ and
\[
\gamma(a,\psi_p)=1 \, \text{ for all } a \in \Z_p^{\times}, \qquad \gamma(p,\psi_p) = \begin{cases}
1 & \text{if } p \equiv 1 \pmod 4,\\
-\sqrt{-1} & \text{if } p \equiv 3 \pmod 4.
\end{cases}
\]
This completely determines the functions $\gamma(\cdot, \psi_p)$ and $\chi_{\psi_p}$ by the above properties. One can easily deduce similar formulae for twists $\psi_p^d$ of the standard additive character.

For a general quadratic space $V$, if $Q(x) = a_1x_1^2 + \cdots + a_mx_m^2$ with respect to some orthogonal basis, then 
\[
\gamma(\psi,V) = \prod_{i} \gamma(\psi^{a_i}) \quad \text{and} \quad \chi_{\psi,V} = \prod_i \chi_{\psi^{a_i}}.
\]
This does not depend on the chosen basis. For example, consider the $3$-dimensional quadratic space $V_3$ as before, with quadratic form whose matrix is $Q_1$. The eigenvalues of this matrix are $1$, $-1$, and $2$, thus $\gamma(\psi, V_3) = \gamma(\psi)\gamma(\psi^{-1})\gamma(\psi^2)$. If $\psi = \psi_p^a$ for some unit $a \in \Z_p^{\times}$, this yields $\gamma(\psi, V_3) = \gamma(\psi_p)^3 = 1$. Besides, we also have $\chi_{\psi,V_3} = \chi_{\psi} \cdot \chi_{\psi^{-1}} \cdot \chi_{\psi^2} = \chi_{\psi}^3 \cdot \chi_{-2}$.

When $m$ is even, the above simplifies considerably. Indeed, if $m$ is even the Weil representation descends to a representation of $\SL_2(F) \times \mathrm{O}(V)$ on $\mathcal S(V)$. Further, the Weil index $\gamma(\psi,V)$ is a $4$-th root of unity in this case, and $\chi_{\psi,V}$ becomes the quadratic character associated with the quadratic space $(V,Q)$. This means that 
\[
 \chi_{\psi,V}(a) = (a, (-1)^{m/2}\det(V))_F, \quad a \in F^{\times}.
\]

It will be useful in some settings to extend the Weil representation $\omega_{\psi,V}$. If $m$ is even, one defines 
\[
 R = \mathrm G(\SL_2 \times \mathrm O(V)) = \{(g,h) \in \GL_2 \times \mathrm{GO}(V): \det(g) = \nu(h)\},
\]
and then $\omega_{\psi,V}$ extends to a representation of $R(F)$ on $\mathcal S(V)$ by setting
\[
 \omega_{\psi,V}(g,h)\phi = \omega_{\psi,V}\left(g\left(\begin{array}{cc} 1 & 0 \\ 0 & \det(g)^{-1}\end{array}\right),1\right)L(h)\phi \quad \text{for } (g,h) \in R(F) \text{ and } \phi \in \mathcal S(V), 
\]
where $L(h)\phi(x) = |\nu(h)|_F^{-m/4} \phi(h^{-1}x)$ for $x \in V$.

\subsection{Theta functions and theta lifts}

Now let $F$ be a number field (for our purposes, we can think of $F = \Q$), and consider a quadratic space $V$ over $F$ of dimension $m$. Fix a non-trivial additive character $\psi$ of $\A_F/F$ and let $\omega = \omega_{\psi,V}$ be the Weil representation of $\widetilde{\SL}_2(\A_F) \times \mathrm O(V)(\A_F)$ on $\mathcal S(V(\A_F))$ with respect to $\psi$. Given $(g,h) \in \widetilde{\SL}_2(\A_F) \times \mathrm O(V)(\A_F)$ and $\phi \in \mathcal S(V(\A_F))$, let 
\[
 \theta(g,h;\phi) := \sum_{x \in V(F)} \omega(g,h) \phi(x).
\]
Then $(g,h) \mapsto \theta(g,h;\phi)$ defines an automorphic form on $\widetilde{\SL}_2(\A_F) \times \mathrm O(V)(\A_F)$, called a {\em theta function}. When $m$ is even, this may be regarded as an automorphic form on $\SL_2(\A_F) \times \mathrm O(V)(\A_F)$.

Let $f$ be a cusp form on $\SL_2(\A_F)$ if $m$ is even, and a genuine cusp form on $\widetilde{\SL}_2(\A_F)$ if $m$ is odd. If $\phi \in \mathcal S(V(\A_F))$, put 
\[
 \theta(h;f,\phi) = \int_{\SL_2(F)\backslash\SL_2(\A_F)} \theta(g,h;\phi)f(g) dg, \quad h \in \mathrm O(V)(\A_F).
\]
Then $\theta(f,\phi): h \mapsto \theta(h;f,\phi)$ defines an automorphic form on $\mathrm O(V)(\A_F)$. If $m$ is even and $\pi$ is an irreducible cuspidal automorphic representation of $\SL_2(\A_F)$, or if $m$ is odd and $\pi$ is an irreducible genuine cuspidal automorphic representation of $\widetilde{\SL}_2(\A_F)$, put 
\[
 \Theta_{\widetilde{\SL}_2 \times \mathrm O(V)}(\pi) := \{\theta(f,\phi): f \in \pi, \phi \in \mathcal S(V(\A_F))\}.
\]
Then $\Theta_{\widetilde{\SL}_2 \times \mathrm O(V)}(\pi)$ is an automorphic representation of $\mathrm O(V)(\A_F)$, called the {\em theta lift of $\pi$}. Going in the opposite direction, one defines similarly the theta lift $\theta(f',\phi)$ of a cusp form $f'$ on $\mathrm O(V)(\A_F)$ and the theta lift $\Theta_{\mathrm O(V) \times \widetilde{\SL}_2}(\pi')$ of an irreducible cuspidal automorphic representation $\pi'$ of $\mathrm O(V)(\A)$.

Suppose that $m$ is even. As we did for the Weil representation, theta lifts can also be extended. If $(g,h) \in R(\A_F)$ and $\phi \in \mathcal S(V(\A_F))$, one defines $\theta(g,h;\phi)$ via the same expression as above (using the extended Weil representation). Then, if $f$ is a cusp form on $\GL_2(\A_F)$ and $h \in \mathrm{GO}(V)(\A_F)$, choose $g' \in \GL_2(\A_F)$ with $\det(g') = \nu(h)$ and set 
\[
 \theta(h;f,\phi) = \int_{\SL_2(F)\backslash \SL_2(\A_F)} \theta(gg',h;\phi) f(gg') dg.
\]
The integral does not depend on the choice of the auxiliary element $g'$, and $\theta(f,\phi): h \mapsto \theta(h;f,\phi)$ defines now an automorphic form on $\mathrm{GO}(V)(\A_F)$. If $\pi$ is an irreducible cuspidal automorphic representation of $\GL_2(\A_F)$, then its theta lift $\Theta_{\GL_2 \times \mathrm{GO}(V)}(\pi)$ is formally defined exactly as before (and the same applies for $\Theta_{\mathrm{GO}(V) \times \GL_2}(\pi')$ if $\pi'$ is an irreducible cuspidal automorphic representation of $\mathrm{GO}(V)$).

\section{\texorpdfstring{$\SL_2$}{SL2}-periods and a central value formula}\label{sec:SL2periods-thm}

Let $f \in S_{2k}^{new}(N_f)$ and $g \in S_{\ell+1}^{new}(N_g)$ be two normalized newforms as in the Introduction. Thus $\ell \geq k\geq 1$ are odd integers, and $N_f, N_g \geq 1$ are odd squarefree integers with $N_g \mid N_f$. We let $m \in \Z$ be such that $\ell-k = 2m$.

Let $\pi$ and $\tau$ denote the irreducible cuspidal automorphic representation associated with $f$ and $g$, respectively. These are automorphic representations of $\PGL_2(\A)$, although we will often regard them as automorphic representations of $\GL_2(\A)$ with trivial central character. 

Let $\psi: \A/ \Q \to S^1$ be the standard additive character of $\A$, and $\overline{\psi}$ be its twist by $-1$. Fix a fundamental discriminant $D < 0$ such that $\chi_D(p) = w_p$  for all primes $p \mid N_f$, where $w_p$ denotes the eigenvalue of the $p$-th Atkin--Lehner involution acting on $f$, and consider the automorphic representation $\tilde{\pi}:=\Theta(\pi \otimes \chi_D, \overline{\psi}^D)$. The assumptions on $D$ guarantee that $\tilde{\pi} \neq 0$, and hence it belongs to the so-called {\em (global) Waldspurger packet}
\[
\mathrm{Wald}_{\overline{\psi}}(\pi) = \{\text{non-zero } \Theta(\pi \otimes \chi_a, \overline{\psi}^a): a \in \Q^{\times}/(\Q^{\times})^2\} = \mathrm{Wald}_{\overline{\psi}^D}(\pi \otimes \chi_D).
\]
Waldspurger's theory (see \cite{Waldspurger91}) tells us that the set $\mathrm{Wald}_{\overline{\psi}}(\pi)$ is finite. Further, these global packets can be conveniently described by means of {\em local} Waldspurger packets. Namely, let $v$ be a rational place and $B_v$ be the quaternion division algebra over $\Q_v$. Set $\tilde{\pi}_v^+ = \Theta(\pi_v, \overline{\psi}_v)$ and $\tilde{\pi}_v^- = \Theta(\mathrm{JL}(\pi_v),\overline{\psi}_v)$, where $\mathrm{JL}(\pi_v)$ is the Jacquet--Langlands lift of $\pi_v$ to $PB_v^{\times}$. Then the local Waldspurger packet $\mathrm{Wald}_{\overline{\psi}_v}(\pi_v)$ is defined as the singleton $\{\tilde{\pi}_v^+\}$ if $\pi_v$ is not square-integrable, and as the set $\{\tilde{\pi}_v^+, \tilde{\pi}_v^-\}$ if $\pi_v$ is square-integrable. If $\epsilon = (\epsilon_v)_v$ is a collection of signs $\epsilon_v \in \{\pm 1\}$, one for each rational place, such that $\epsilon_v = +1$ whenever $\pi_v$ is not square-integrable (or equivalently, for each $\epsilon \in \{\pm 1\}^{|\Sigma(\pi)|}$, where $\Sigma(\pi)$ is the set of rational places where $\pi$ is square-integrable), we set $\tilde{\pi}^{\epsilon} = \otimes \tilde{\pi}_v^{\epsilon_v}$. Then 
\[
\mathrm{Wald}_{\overline{\psi}}(\pi) = \{\tilde{\pi}^{\epsilon} : \prod_v \epsilon_v = \epsilon(1/2,\pi)\}.
\]
The labelling $\pm$ of a given element in $\mathrm{Wald}_{\overline{\psi}}(\pi)$ at each place $v \in \Sigma(\pi)$ depends on the choice of the additive character. For the representation $\tilde{\pi} = \Theta(\pi\otimes \chi_D, \overline{\psi}_D)$, we have $\epsilon_{\infty} = -1$ and $\epsilon_p = \chi_D(p) = w_p$ for each prime $p \mid N_f$.

We let $h \in S_{k+1/2}^{+,new}(N_f)$ be any (non-zero) newform in Shimura--Shintani correspondence with $f$. Then the adelization of $h$ belongs to $\tilde{\pi}$, and $h$ is unique up to non-zero multiples. We let also $F = \mathrm{SK}(h) \in S_{k+1}(\Gamma_0^{(2)}(N_f))$ be the Saito--Kurokawa lift of $h$, and $\Pi$ be the automorphic representation of $\mathrm{PGSp}_2(\A)$ associated with it. The Siegel modular form $F$ admits a nice Fourier expansion
\[
F(Z) = \sum_B A_F(B) e^{2\pi\sqrt{-1}\mathrm{Tr}(BZ)}, \quad Z =X+\sqrt{-1}Y \in \mathcal H_2,
\]
where $B$ runs over the half-integral, positive definite symmetric two-by-two matrices, and $A_F(B)$ is given in an elementary way in terms of the Fourier coefficients of $h$ (see \eqref{FCoeff-SK}). 

For each integer $\kappa \geq 1$, consider the classical Maass differential operator (see \eqref{Maass-Delta} below for the precise definition)
\[
\Delta_{\kappa}: S_{\kappa}^{nh}(\Gamma_0^{(2)}(N_f)) \, \longrightarrow S_{\kappa+2}^{nh}(\Gamma_0^{(2)}(N_f))
\]
sending nearly holomorphic Siegel forms of weight $\kappa$ (and level $\Gamma_0^{(2)}(N_f)$) to nearly holomorphic Siegel forms of weight $\kappa+2$ (and level $\Gamma_0^{(2)}(N_f)$). By applying $\Delta_{k+1}^m := \Delta_{\ell-1}\circ \Delta_{\ell-3} \circ \cdots \circ \Delta_{k+1}$ to $F$, one obtains a nearly holomorphic Siegel form 
\[
\Delta_{k+1}^m F \in S_{\ell+1}^{nh}(\Gamma_0^{(2)}(N_f))
\]
of weight $\ell+1$ and level $\Gamma_0^{(2)}(N_f)$. By using the definition of the Maass differential operator, one shows that the Fourier expansion of $\Delta_{k+1}^m F$ is expressed as
\[
\Delta_{k+1}^mF(Z) = \sum_B A_F(B)C(B,Y) e^{2\pi\sqrt{-1}\mathrm{Tr}(BZ)},
\]
where $C(B,Y)$ can be written down explicitly by an induction argument (see \eqref{const_diff}).

\begin{theorem}\label{thm:centralvalue}
With the above notation, suppose that $w_p = 1$ for each prime dividing $M_g:=N_f/N_g$. Then  
\[
\Lambda(f \otimes \mathrm{Ad}(g), k) = 2^{6m+k+1-\nu(M_g)} C_0(f,g)C_{\infty}(f,g) \cdot \frac{\langle f,f\rangle}{\langle h,h\rangle} \frac{|\langle \breve{F}_{|\mathcal H\times\mathcal H}, g \times V_{M_g}g\rangle|^2}{\langle g,g \rangle^2},
\]
where $\breve{F} \in S_{\ell+1}^{nh}(\Gamma_0^{(2)}(N_f))$ is a Siegel modular form closely related to $\Delta_{k+1}^m F$ defined explicitly in Proposition \ref{prop:Ftheta}, $\breve{F}_{|\mathcal H\times\mathcal H}$ denotes its restriction or pullback to $\mathcal H \times \mathcal H \subset \mathcal H_2$, $\nu(M_g)$ denotes the number of primes dividing $M_g$, and the constants $C_0(f,g)$, $C_{\infty}(f,g) \in \Q^{\times}$ are 
\begin{align*}
    C_0(f,g) & = \frac{M_g^2\mu_{N_g}^2}{N_f} = \frac{M_g^2}{N_f} \prod_{p\mid N_g} (p+1)^2,\\
    C_{\infty}(f,g) & = \frac{(2m)!}{m!}\frac{(k+m-1)!}{(\ell-1)!} \sum_{\substack{0\leq s \leq 2m,\\ s \text{ even}}} \prod_{\substack{0\leq j \leq s-2, \\ j \text{ even}}} \frac{(2m-j)(2m-j-1)}{(j+2)(2k+j+1)}.
\end{align*}
\end{theorem}

The strategy to prove the central value formula in this theorem is the same as in \cite{PaldVP}. Indeed, let $\omega = \omega_{\overline{\psi}}$ denote the Weil representation of $\widetilde{\SL}_2(\A)$ acting on the space $\mathcal S(\A)$ of Bruhat--Schwartz functions (for the one dimensional quadratic space endowed with bilinear form $(x,y) = 2xy$) with respect to the additive character $\overline{\psi}$ (note that $\tilde{\pi}$ belongs to $\mathrm{Wald}_{\overline{\psi}}(\pi)$). Associated with $\tilde{\pi}$, $\tau$ and $\omega$, there is a (global) $\SL_2$-period functional
\[
  \mathcal Q: \tilde{\pi} \otimes \tilde{\pi} \otimes \tau \otimes \tau \otimes \omega \otimes \omega \,\, \longrightarrow \,\, \C
\]
defined by associating to each choice of decomposable vectors $\mathbf h_1, \mathbf h_2 \in \tilde{\pi}$, $\mathbf g_1, \mathbf g_2 \in \tau$, $\pmb{\phi}_1, \pmb{\phi}_2 \in \omega$, the product of integrals 
\[
 \mathcal Q(\mathbf h_1, \mathbf h_2,\mathbf g_1, \mathbf g_2,\pmb{\phi}_1, \pmb{\phi}_2) := 
 \left(\int_{[\SL_2]}\overline{\mathbf h_1(g)}\mathbf g_1(g)\Theta_{\pmb{\phi}_1}(g)dg\right) \cdot 
 \overline{\left(\int_{[\SL_2]}\overline{\mathbf h_2(g)}\mathbf g_2(g)\Theta_{\pmb{\phi}_2}(g)dg\right)}.
\]

Let us assume for now that the global $\SL_2$-period functional $\mathcal Q$ is non-vanishing (which is true under the assumptions of Theorem \ref{thm:centralvalue}, see Proposition \ref{prop:nonvanishing} and Corollary \ref{cor:nonvanishing} below). Then, we know by  \cite[Theorem 4.5]{Qiu} that $\mathcal Q$ decomposes as a product of local $\SL_2$-periods up to certain $L$-values. Namely, one has (notice that Qiu's formula in loc. cit. involves a factor $\zeta_{\Q}(2)$ due to a different choice of Haar measures, cf. \cite[Remark 6.1]{PaldVP})
\begin{equation}\label{SL2period-decomposition}
 \mathcal Q(\mathbf h_1, \mathbf h_2,\mathbf g_1, \mathbf g_2,\pmb{\phi}_1, \pmb{\phi}_2) = 
 \frac{1}{4}\frac{\Lambda(\pi \otimes \mathrm{ad}(\tau),1/2)}{\Lambda(1,\pi,\mathrm{ad})\Lambda(1,\tau,\mathrm{ad})} 
 \prod_v \mathcal I_v(\mathbf h_{1},\mathbf h_{2}, \mathbf g_{1}, \mathbf g_{2},\pmb{\phi}_{1}, \pmb{\phi}_{2}),
\end{equation}
where for each rational place $v$, the local period $\mathcal I_v(\mathbf h_{1},\mathbf h_{2}, \mathbf g_{1}, \mathbf g_{2},\pmb{\phi}_{1}, \pmb{\phi}_{2})$ is defined by integrating a product of matrix coefficients, and equals
\[
\frac{L(1,\pi_v,\mathrm{ad})L(1,\tau_v,\mathrm{ad})}{L(\pi_v\otimes \mathrm{ad}(\tau_v),1/2)}
\int_{\SL_2(\Q_v)} \overline{\langle\tilde{\pi}(g_v)\mathbf h_{1,v}, \mathbf h_{2,v}\rangle} \langle\tau(g_v)\mathbf g_{1,v},\mathbf g_{2,v} \rangle \langle \omega_v(g_v)\pmb{\phi}_{1,v},\pmb{\phi}_{2,v} \rangle dg_v.
\]
Now, the $L$-value $\Lambda(\pi \otimes \mathrm{ad}(\tau),1/2)$ coincides with the central value $\Lambda(f\otimes \mathrm{Ad}(g),k)$ in the above theorem. Thus, Qiu's decomposition formula provides a way to compute this central value, by finding a {\em test vector} at which the global period does not vanish, and then computing both the global period and all the corresponding local periods evaluated at this test vector.

Let us elaborate a bit on formula \eqref{SL2period-decomposition} above, writing $\mathcal Q(\breve{\mathbf h},\breve{\mathbf g},\breve{\pmb{\phi}}) := \mathcal Q(\breve{\mathbf h},\breve{\mathbf h},\breve{\mathbf g},\breve{\mathbf g},\breve{\pmb{\phi}},\breve{\pmb{\phi}})$ for each pure tensor $\breve{\mathbf h}\otimes\breve{\mathbf g}\otimes\breve{\pmb{\phi}} \in \tilde{\pi}\otimes\tau\otimes\omega$, and using similar conventions with the local periods. Setting
\[
 \mathcal I_v^{\sharp}  (\breve{\mathbf h}, \breve{\mathbf g}, \breve{\pmb{\phi}}) := 
\frac{\mathcal I_v(\breve{\mathbf h}, \breve{\mathbf g}, \breve{\pmb{\phi}})}{\langle \breve{\mathbf h}_v, \breve{\mathbf h}_v \rangle\langle \breve{\mathbf g}_v, \breve{\mathbf g}_v \rangle \langle \breve{\pmb{\phi}}_v, \breve{\pmb{\phi}}_v \rangle} = 
\frac{\mathcal I_v(\breve{\mathbf h}, \breve{\mathbf g}, \breve{\pmb{\phi}})}{||\breve{\mathbf h}_v||^2 ||\breve{\mathbf g}_v||^2||\breve{\pmb{\phi}}_v||^2},
\]
one has 
\begin{equation}\label{Iv-alphav}
  \mathcal I_v^{\sharp} (\breve{\mathbf h}, \breve{\mathbf g}, \breve{\pmb{\phi}}) = \frac{L(1,\pi_v,\mathrm{ad})L(1,\tau_v,\mathrm{ad})}{L(\pi_v\otimes \mathrm{ad}(\tau_v),1/2)} \alpha_v^{\sharp} (\breve{\mathbf h}, \breve{\mathbf g}, \breve{\pmb{\phi}}),
\end{equation}
with 
\begin{equation}\label{alphav}
 \alpha_v^{\sharp} (\breve{\mathbf h}, \breve{\mathbf g}, \breve{\pmb{\phi}}) := \int_{\SL_2(\Q_v)} \frac{\overline{\langle\tilde{\pi}(g_v)\breve{\mathbf h}_v, \breve{\mathbf h}_v\rangle}}{||\breve{\mathbf h}_v||^2} \frac{\langle\tau(g_v)\breve{\mathbf g}_v,\breve{\mathbf g}_v \rangle}{||\breve{\mathbf g}_v||^2} \frac{\langle \omega_{\overline{\psi}_v}(g_v)\breve{\pmb{\phi}}_v,\breve{\pmb{\phi}}_v \rangle}{||\breve{\pmb{\phi}}_v||^2} dg_v.
\end{equation}
If $\breve{\mathbf h}$, $\breve{\mathbf g}$ and $\breve{\pmb{\phi}}$ are chosen so that $\mathcal Q(\breve{\mathbf h}, \breve{\mathbf g}, \breve{\pmb{\phi}})$ is non-zero, then we deduce from \eqref{SL2period-decomposition} that
\begin{equation}\label{centralvalueformula-Q}
 \Lambda(f \otimes \mathrm{Ad}(g), k) = \frac{4 \Lambda(1,\pi,\mathrm{ad})\Lambda(1,\tau,\mathrm{ad})}{\langle \breve{\mathbf h}, \breve{\mathbf h} \rangle\langle \breve{\mathbf g}, \breve{\mathbf g}\rangle\langle \breve{\pmb{\phi}}, \breve{\pmb{\phi}} \rangle} \left( \prod_v \mathcal I_v^{\sharp}(\breve{\mathbf h}, \breve{\mathbf g}, \breve{\pmb{\phi}})^{-1}\right) \mathcal Q(\breve{\mathbf h}, \breve{\mathbf g}, \breve{\pmb{\phi}}).
\end{equation}

We will choose a suitable {\em test vector} $\breve{\mathbf h} \otimes \breve{\mathbf g} \otimes \breve{\pmb{\phi}} \in \tilde{\pi} \otimes \tau \otimes \omega$ such that $\mathcal Q(\breve{\mathbf h}, \breve{\mathbf g}, \breve{\pmb{\phi}}) \neq 0$, and we will compute the terms on the right hand side of the above expression to obtain the central value formula claimed in the theorem. By virtue of a comparison theorem between the global $\SL_2$-period $\mathcal Q$ and a global $\mathrm{SO}(4)$-period due to Qiu (see \cite[Theorem 5.4]{Qiu}, or Section \ref{sec:proof} below), the global contribution $\mathcal Q(\breve{\mathbf h}, \breve{\mathbf g}, \breve{\pmb{\phi}})$ is the responsible of the term $|\langle \breve{F}_{|\mathcal H\times\mathcal H}, g \times V_{M_g}g\rangle|^2/\langle g, g\rangle^2$. Hence, the proof of Theorem \ref{thm:centralvalue} follows by making explicit the right hand side of \eqref{centralvalueformula-Q}.

For the sketched strategy to work, it is essential that the $\SL_2$-period $\mathcal Q$ is non-vanishing. A criterion for this is proved in \cite[Proposition 4.1]{Qiu} (see also \cite[Theorem 7.1]{GanGurevich}):

\begin{proposition}\label{prop:nonvanishing}
The functional $\mathcal Q$ is non-zero if and only if the following conditions hold:
\begin{itemize}
    \item[i)] $\Lambda(\pi\otimes\mathrm{ad}(\tau),1/2) \neq 0$;
    \item[ii)] $\tilde{\pi} = \tilde{\pi}^{\epsilon}$ with $\epsilon_v = \epsilon(1/2,\pi_v \otimes\tau_v\otimes\tau_v^{\vee})$;
    \item[iii)] $\epsilon(1/2,\pi_v \otimes\tau_v\otimes\tau_v^{\vee}) = 1$ whenever $\pi_v$ is not square-integrable.
\end{itemize}
\end{proposition}

In our current setting, the non-vanishing of $\mathcal Q$ is equivalent to the non-vanishing of  $\Lambda(f\otimes\mathrm{Ad}(g),k)$:

\begin{corollary}\label{cor:nonvanishing}
With the same assumptions as in Theorem \ref{thm:centralvalue}, the functional $\mathcal Q$ is non-zero if and only if $\Lambda(f\otimes\mathrm{Ad}(g),k)\neq 0$.
\end{corollary}
\begin{proof}
Condition iii) in the above proposition clearly holds, thus we may prove that ii) is satisfied under the sign assumptions made in Theorem \ref{thm:centralvalue}. We only need to consider places $v \mid N_f\infty$. At $v = \infty$, we have $\epsilon_{\infty} = -1$ and also $\epsilon(1/2,\pi_v \otimes\tau_v \otimes \tau_v^{\vee})$ because of our choice of weights with $\ell \geq k$. Let $p$ be a prime dividing $N_g$. Then both $\pi_p$ and $\tau_p$ are (quadratic twists of) Steinberg representations, and in this case \cite[Proposition 8.6]{Prasad} implies that $\epsilon(1/2,\pi_p \otimes\tau_p \otimes \tau_p^{\vee}) = \epsilon(1/2,\pi_p) = w_p$, which agrees with $\chi_D(p) = \epsilon_p$. At primes $p \mid N_f/N_g$, the representation $\tau_p$ is an unramified principal series instead, and 
\cite[Proposition 8.4]{Prasad} tells us that $\epsilon(1/2,\pi_p \otimes\tau_p \otimes \tau_p^{\vee}) = 1$. Since we assume that $\chi_D(p) = w_p = 1$ at such primes, we see that this coincides with $\epsilon_p$, and hence condition ii) in the previous proposition holds.
\end{proof}

\section{Choice of the test vector}\label{sec:testvector}

We keep the notation and assumptions as in the previous section, and proceed now to describe our choice of {\em test vector}
\[
\breve{\mathbf h} \otimes \breve{\mathbf g} \otimes \breve{\pmb{\phi}} \in \tilde{\pi}\otimes\tau\otimes\omega
\] 
that will be used to prove Theorem \ref{thm:centralvalue} following the already explained strategy. 

To begin with, let us describe the Bruhat--Schwartz function $\breve{\pmb{\phi}} = \otimes_v \breve{\pmb{\phi}}_v \in \mathcal S(\A)$. Recall that we regard $\mathcal S(\A)$ as the space of Bruhat--Schwartz functions on the one-dimensional quadratic space endowed with quadratic form $Q(x) = x^2$. Our choice $\breve{\pmb{\phi}}$ is determined by its local components, which are defined as follows:
\begin{itemize}
    \item[i)] if $v = p$ is a prime, then we let $\breve{\pmb{\phi}}_p= \mathbf 1_{\Z_p}$ be the characteristic function of $\Z_p$ in the space $\mathcal S(\Q_p)$ of Bruhat--Schwartz functions; 
    \item[ii)] at the archimedean place $v = \infty$, we define $\breve{\pmb{\phi}}_{\infty}$ by setting $\breve{\pmb{\phi}}_{\infty}(x) = e^{-2\pi x^2}$ for all $x \in \R$.
\end{itemize}

\begin{lemma}\label{lemma:phi-invariance-norm}
For each rational prime $p$, $\breve{\pmb{\phi}}_p$ is invariant under the action of $\SL_2(\Z_p)\subseteq \SL_2(\Q_p)$, and in addition $||\breve{\pmb{\phi}}_p||^2 = \langle \breve{\pmb{\phi}}_p,\breve{\pmb{\phi}}_p \rangle = 1$. At the archimedean place, one has $||\breve{\pmb{\phi}}_{\infty}||^2 = 2^{-1}$, hence also $||\breve{\pmb{\phi}}||^2 = 2^{-1}$.
\end{lemma}
\begin{proof}
The invariance assertion follows easily from the definitions. If $p$ is a prime, then  
\[
||\breve{\pmb{\phi}}_p||^2 =  \langle \breve{\pmb{\phi}}_p,\breve{\pmb{\phi}}_p \rangle = \int_{\Q_p} \breve{\pmb{\phi}}_p(x) \overline{\breve{\pmb{\phi}}_p(x)} dx = \mathrm{vol}(\Z_p) = 1.
\]
And besides, $||\breve{\pmb{\phi}}_{\infty}||^2 = \int_{\R} e^{-4\pi x^2} dx = 1/2$.
\end{proof}

Now let us describe our choice for $\breve{\mathbf h} \in \tilde{\pi}$ and $\breve{\mathbf g}\in \tau$. To do so, let $\mathbf h = \otimes_v \mathbf h_v \in \tilde{\pi}$ and $\mathbf g = \otimes_v \mathbf g_v \in \tau$ denote the adelizations of the cuspidal forms $h \in S_{k+1/2}^+(N_f)$ and $g \in S_{\ell+1}(N_g)$, respectively. At each rational prime $p \nmid N_f$ (resp. $p \nmid N_g$), the local component $\mathbf h_p$ (resp. $\mathbf g_p$) is an unramified or spherical vector in the local representation $\tilde{\pi}_p$ (resp. $\tau_p$). These are unique up to scalar multiples. If instead $p$ is a prime dividing $N_f$ (resp. $N_g$), then $\mathbf h_p$ (resp. $\mathbf g_p$) is a newvector in $\tilde{\pi}_p$ (resp. $\tau_p$) fixed under the action of $\tilde{\Gamma}_0(p)$ (resp. $K_0(p)$). Such local newforms are also unique up to scalar multiples. At the archimedean place, $\tau_{\infty}$ is a discrete series representation of $\PGL_2(\R)$ of weight $\ell+1$, and $\mathbf g_{\infty}$ is a lowest weight vector in $\tau_{\infty}$. Similarly, $\tilde{\pi}_{\infty}$ is a discrete series representation of $\widetilde{\SL}_2(\R)$ of lowest $\widetilde{\mathrm{SO}}(2)$-type $k+1/2$, and $\mathbf h_{\infty}$ is a lowest weight vector in $\tilde{\pi}_{\infty}$. Again, such lowest weight vectors are uniquely determined up to multiples. We will define $\breve{\mathbf h} = \otimes_v \breve{\mathbf h}_v$ and $\breve{\mathbf g} = \otimes_v \breve{\mathbf g}_v$ by describing their local components at each place $v$, according to the following cases:
\begin{itemize}
    \item[(1)] $v = p$ is a prime not dividing $N_f$;
    \item[(2)] $v = p$ is a prime dividing $N_g$;
    \item[(3)] $v = p$ is a prime dividing $M_g = N_f/N_g$;
    \item[(4)] $v=\infty$ is the archimedean place.
\end{itemize}

\subsection{Primes not dividing \texorpdfstring{$N_f$}{Nf}}\label{sec:unramifiedp}

If $p$ is a prime not dividing $2N_f$, then both $\tilde{\pi}_p$ and $\tau_p$ are unramified principal series representations, of $\widetilde{\SL}_2(\Q_p)$ and $\PGL_2(\Q_p)$ respectively. At such primes, we choose both $\breve{\mathbf h}_p = \mathbf h_p$ and $\breve{\mathbf g}_p = \mathbf g_p$ to be an unramified (or spherical) vector in $\tilde{\pi}_p$ and $\tau_p$, respectively. At $p = 2$, we adopt the same choice as the one explained in \cite[Section 9]{PaldVP}: we let $\breve{\mathbf g}_2 = \mathbf g_2^{\sharp} := \tau_2(t(2)^{-1})\mathbf g_2$ and $\breve{\mathbf h}_2 = \tilde{\pi}_2((t(2),1))\mathbf h_2$, where $t(2) = \left(\begin{smallmatrix}2 & 0 \\ 0 & 2^{-1}\end{smallmatrix}\right) \in \SL_2(\Q_2)$.

\subsection{Primes dividing \texorpdfstring{$N_g$}{Ng}}\label{sec:primesNg}

Let $p$ be a prime dividing $N_g$. By our assumption that $N_g$ is squarefree, $\tau_p$ is a twist of the Steinberg representation by an unramified quadratic character $\chi: \Q_p^{\times} \to \C^{\times}$. That is to say, $\tau_p$ is the unique irreducible subrepresentation of the induced representation $\pi(\chi|\cdot|_p^{1/2},\chi^{-1}|\cdot|_p^{-1/2})$. The subspace $\tau_p^{K_0}\subseteq \tau_p$ of vectors fixed by 
\[
K_0 = K_0(p) = \left\lbrace \left(\begin{array}{cc}a & b \\c & d\end{array}\right )\in \GL_2(\Z_p): c \equiv 0 \pmod p\right\rbrace
\]
is one-dimensional, and it is generated by the  newvector $\mathbf g_p: \GL_2(\Q_p) \to \C$ in the induced model characterized by the property that 
\[
\mathbf g_{p|\GL_2(\Z_p)} = \mathbf 1_{K_0} - \frac{1}{p} \mathbf 1_{K_0 wK_0},
\]
where $w = \left(\begin{smallmatrix} 0 & 1 \\ 1 & 0 \end{smallmatrix}\right)$. We choose the $p$-th component of $\breve{\mathbf g}$ to be this newvector: $\breve{\mathbf g}_p = \mathbf g_p$.

As for $\tilde{\pi}_p$, observe first that $\pi_p$ is also a twist of the Steinberg representation by an unramified quadratic character. Then, by our choice of $\tilde{\pi} = \Theta(\pi\otimes\chi_D,\overline{\psi}^D)$ in the Waldspurger packet $\mathrm{Wald}_{\overline{\psi}}(\pi) = \mathrm{Wald}_{\overline{\psi}^D}(\pi\otimes\chi_D)$, the local representation $\tilde{\pi}_p = \Theta(\pi_p\otimes\chi_D,\overline{\psi}_p^D)$ is the special representation $\tilde{\sigma}^{\delta}(\overline{\psi}_p^D)$ of $\widetilde{\SL}_2(\Q_p)$, where $\delta \in \Z_p^{\times}$ is any non-quadratic residue. This representation is realized as the space of functions $\tilde{\varphi}: \widetilde{\SL}_2(\Q_p) \to \C$ such that 
\begin{equation}\label{transfproperty-hp}
\tilde{\varphi}\left(\left[\begin{pmatrix} a & \ast \\ & a^{-1} \end{pmatrix}, \epsilon\right] g\right) = \epsilon \chi_{\bar{\psi}_p^D}(a)\chi_{\delta}(a)|a|_p^{3/2}\tilde{\varphi}(g),
\end{equation}
where $\chi_{\delta} = (\cdot, \delta)_p$ is the quadratic character associated with $\delta \in \Z_p^{\times}$ and $\chi_{\bar{\psi}_p^D} = \chi_{\psi_p^{-D}}: \Q_p^{\times} \to S^1$ is as in Section \ref{sec:Weilreps}. Recall that the fundamental discriminant $D \in \Q^{\times}$ has been chosen so that $D \in \Z_p^{\times}$ and $\chi_D(p) = w_p$.

If $\tilde{\Gamma}$ denotes the image of $\Gamma=\Gamma_0(p)\subseteq \SL_2(\Z_p)$ into $\widetilde{\SL}_2(\Z_p)$ under the canonical splitting, then the space of $\tilde{\Gamma}$-fixed vectors in $\tilde{\pi}_p$ is one-dimensional. Moreover, this space is generated by the newvector $\mathbf h_p:\widetilde{\SL}_2(\Q_p) \to \C$ characterized by the property that 
\[
\mathbf h_{p|\widetilde{\SL}_2(\Z_p)} = \mathbf 1_{\widetilde{\SL}_2(\Z_p)} - (p+1)\mathbf 1_{\tilde{\Gamma}}.
\]
Here, $\mathbf 1_{\widetilde{\SL}_2(\Z_p)}$ denotes the (genuine) function on $\widetilde{\SL}_2(\Q_p)$ sending $[g,\epsilon]$ to $0$ if $g \not\in\SL_2(\Z_p)$, and to $\epsilon s_p(g)$ otherwise. Similarly, $\mathbf 1_{\tilde{\Gamma}}$ is the function on $\widetilde{\SL}_2(\Q_p)$ that sends $[g,\epsilon]$ to $0$ if $g \not\in \Gamma$, and to $\mathbf 1_{\widetilde{\SL}_2(\Q_p)}([g,\epsilon])$ otherwise. We choose the $p$-th component of $\breve{\mathbf h}$ to be this newvector: $\breve{\mathbf h}_p = \mathbf h_p$.

\subsection{Primes dividing \texorpdfstring{$M_g$}{Mg}}
\label{sec:primesMg}

Let now $p$ be a prime dividing $N_f$ but not $N_g$, i.e. $p$ divides $M_g$. In this case, the local type of $\tilde{\pi}_p$ is as in the previous paragraph, and we continue to choose $\breve{\mathbf h}_p = \mathbf h_p$ to be the newvector as described there. Besides, $\tau_p$ is now the unramified principal series representation $\pi(\chi,\chi^{-1})$ associated with an unramified character $\chi: \Q_p^{\times} \to \C^{\times}$. The representation being unitary, we have $\chi^{-1} = \overline{\chi}$. The subspace $\tau_p^{\GL_2(\Z_p)} \subseteq \tau_p$ of $\GL_2(\Z_p)$-fixed vectors is one-dimensional, and generated by the unramified (or spherical) vector $\mathbf g_p: \GL_2(\Q_p) \to \C$ characterized by the property that 
\[
\mathbf g_p\left(\left(\begin{smallmatrix} a & \ast \\ & d \end{smallmatrix}\right)x\right) = \begin{cases} \chi(a)\overline{\chi}(d)|ad^{-1}|_p^{1/2} & \text{if } x \in \GL_2(\Z_p), \,  \left(\begin{smallmatrix} a & \ast \\ & d \end{smallmatrix}\right) \in B(\Q_p) \text{ with } a, d \in \Q_p^{\times}, \\
0 & \text{otherwise},
\end{cases}
\]
where $B$ denotes the Borel subgroup of $\GL_2$ of upper-triangular matrices. In particular, notice that $\mathbf g_p(x)=1$ for all $x \in \GL_2(\Z_p)$. This gives a well-defined element $\mathbf g_p$ by virtue of Iwasawa decomposition for $\GL_2(\Z_p)$. We define the $p$-th component $\breve{\mathbf g}_p$ of $\breve{\mathbf g}$ to be the old vector 
\[
\breve{\mathbf g}_p := \mathbf V_p \mathbf g_p = \tau_p(\varpi_p)\mathbf g_p, \quad \text{where } \varpi_p = \left(\begin{smallmatrix} p^{-1} & 0 \\ 0 & 1 \end{smallmatrix}\right) \in \GL_2(\Q_p).
\]
It is elementary to check that the vector $\breve{\mathbf g}_p$ is now fixed by $K_0=K_0(\Z_p)$, and it is not fixed by $\GL_2(\Z_p)$. One can also easily give an explicit description of $\breve{\mathbf g}_p(x)$ for $x \in \GL_2(\Q_p)$, but we will not need it. We just note for later use that since the spherical vector $\mathbf g_p$ is normalized to have norm $1$, the same holds true for $\breve{\mathbf g}_p$ because the norm pairing is $\GL_2(\Q_p)$-invariant.

\begin{lemma}\label{lemma:norm-gpbreve}
 We have $||\breve{\mathbf g}_p||^2 = 1$.
\end{lemma}

\subsection{The archimedean place}\label{sec:realplace}

We consider now the archimedean components of $\tau$ and $\tilde{\pi}$, and choose the corresponding local vectors $\breve{\mathbf g}_{\infty}\in \tau_{\infty}$ and $\breve{\mathbf h}_{\infty} \in \tilde{\pi}_{\infty}$. On the one hand, $\tau_{\infty}$ is a discrete series representation of $\PGL_2(\R)$ of weight $\ell+1$, and we choose $\breve{\mathbf g}_{\infty} \in \tau_{\infty}$ to be a lowest weight vector. Similarly, $\pi_{\infty}$ is a discrete series representation of weight $2k$, and consequently $\tilde{\pi}_{\infty}$ is a discrete series representation of $\widetilde{\SL}_2(\R)$ of lowest $\widetilde{\mathrm{SO}}(2)$-type $k+1/2$. We choose a lowest weight vector $\mathbf h_{\infty} \in \tilde{\pi}_{\infty}$, and define $\breve{\mathbf h}_{\infty}$ as follows. Let $\mathfrak{gl}(2,\R)$ be the Lie algebra of $\GL_2(\R)$, and $\mathfrak{gl}(2,\R)_{\C}$ be its complexification. Consider the weight raising element 
\[
V_+ := \begin{pmatrix} 1 & 0 \\ 0 & -1\end{pmatrix} \otimes 1 + \begin{pmatrix} 1 & 0 \\ 0 & 1\end{pmatrix} \otimes \sqrt{-1} \in \mathfrak{gl}(2,\R)_{\C},
\]
and normalize it setting $\tilde V_+ := -\frac{1}{8\pi} V_+$. Then we define $\breve{\mathbf h}_{\infty} := \tilde V_+^m\mathbf h_{\infty}$, where recall that $2m = \ell-k$. Thus $\breve{\mathbf h}_{\infty}$ is a weight $\ell+1/2$ vector in $\tilde{\pi}_{\infty}$. The vector $\mathbf h_{\infty}$ is (up to a non-zero multiple) the archimedean component of the adelization of the modular form $h \in S_{k+1/2}^+(N_f)$, while $\breve{\mathbf h}_{\infty}$ is (up to a non-zero multiple) the archimedean component of the adelization of the nearly holomorphic modular form $\delta_{k+1/2}^m h \in S_{\ell+1/2}^{+,nh}(N_f)$, where 
\[
\delta_{k+1/2}: S_{k+1/2}^{nh}(N_f) \, \longrightarrow \, S_{k+5/2}^{nh}(N_f)
\]
is the usual Shimura--Maass differential operator sending nearly holomorphic modular forms of weight $k+1/2$ to nearly holomorphic modular forms of weight $k+5/2$. It is defined as 
\[
\delta_{k+1/2} := \frac{1}{2\pi\sqrt{-1}}\left(\frac{\partial}{\partial \tau}+\frac{2k+1}{4y\sqrt{-1}}\right), \qquad \tau = x + \sqrt{-1} y \in \mathcal H,
\]
and we set $\delta_{k+1/2}^m := \delta_{\ell-3/2}\circ \cdots \circ \delta_{k+1/2}$.

\section{Computation of local periods}\label{sec:localperiods}

Let $\breve{\mathbf h} \otimes \breve{\mathbf g} \otimes \breve{\pmb{\phi}} \in \tilde{\pi}\otimes\tau\otimes\omega$ be the test vector as described in the previous section. The goal of this section is to compute the value of all the normalized local periods $\mathcal I_v^{\sharp}(\breve{\mathbf h},\breve{\mathbf g},\breve{\pmb{\phi}})$, for $v$ a rational place.

For every rational prime $p \nmid M_g = N_f/N_g$, the local components $\breve{\mathbf h}_p$, $\breve{\mathbf g}_p$, and $\breve{\pmb{\phi}}_p$ are the same as in \cite{PaldVP}, and therefore the computations done there still apply:

\begin{proposition}\label{prop:Ip-PaldVP}
If $p$ is a prime not dividing $M_g$, then 
\[
\mathcal I_p^{\sharp}(\breve{\mathbf h},\breve{\mathbf g},\breve{\pmb{\phi}}) = \begin{cases}
1 & \text{if } p \nmid N_f, \\
p^{-1} & \text{if } p \mid N_g.
\end{cases}
\]
\end{proposition}
\begin{proof}
The case $p \nmid 2N_f$ actually follows already from \cite[Lemma 4.4]{Qiu}, and the case $p = 2$ is proved in \cite[Proposition 9.2]{PaldVP}. The case $p \mid N_g$ is covered in \cite[Proposition 7.15]{PaldVP}.
\end{proof}

It only remains to perform the computation of the normalized local periods $\mathcal I_p^{\sharp}(\breve{\mathbf h},\breve{\mathbf g},\breve{\pmb{\phi}})$ at primes dividing $M_g$ and of the archimedean period $\mathcal I_{\infty}^{\sharp}(\breve{\mathbf h},\breve{\mathbf g},\breve{\pmb{\phi}})$.

\subsection{The normalized local period at primes \texorpdfstring{$p\mid M_g$}{p|Mg}}

Let $p$ be a prime dividing $M_g = N_f/N_g$, as in Section \ref{sec:primesMg}. In this case, the three vectors $\breve{\mathbf h}_p$, $\breve{\mathbf g}_p$ and $\breve{\pmb{\phi}}_p$ are fixed under the action of $\Gamma = \Gamma_0(p) \subseteq \SL_2(\Z_p)$. Therefore, the matrix coefficients involved in the computation of the local integral \eqref{alphav} will be $\Gamma$-biinvariant. In particular, one can compute $\alpha_p^{\sharp}(\breve{\mathbf h},\breve{\mathbf g},\breve{\pmb{\phi}})$ as a sum 
\[
\alpha_p^{\sharp}(\breve{\mathbf h},\breve{\mathbf g},\breve{\pmb{\phi}}) = \sum_{r\in \mathcal R} \overline{\Phi_{\breve{\mathbf h}_p}(r)}\Phi_{\breve{\mathbf g}_p}(r)\Phi_{\breve{\pmb{\phi}}_p}(r)\mathrm{vol}(\Gamma r \Gamma),
\]
where $\mathcal R$ is a set of representatives for a decomposition of $\SL_2(\Q_p)$ into double cosets for $\Gamma$, and we abbreviate
\[
\Phi_{\breve{\mathbf h}_p}(r) = \frac{\langle\tilde{\pi}_p(r)\breve{\mathbf h}_p,\breve{\mathbf h}_p \rangle}{||\breve{\mathbf h}_p||^2} \quad \text{for } r \in \SL_2(\Q_p),
\]
and similarly for $\Phi_{\breve{\mathbf g}_p}$ and $\Phi_{\breve{\pmb{\phi}}_p}$. A set of representatives $\mathcal R$ as required above is furnished by the elements 
\[
\alpha_n = \left(\begin{array}{cc} p^n & 0 \\ 0 & p^{-n} \end{array} \right), \quad \beta_m = s \alpha_m= \left(\begin{array}{cc} 0 & p^{-m} \\ -p^m & 0 \end{array} \right),
\]
with $n$ and $m$ varying over all the integers, and where $s = \left(\begin{smallmatrix}0 & 1 \\ -1 & 0\end{smallmatrix}\right)$. Indeed, by combining the Cartan decomposition for $\SL_2(\Q_p)$ relative to the maximal compact open subgroup $\SL_2(\Z_p)$ with the so-called Bruhat decomposition for $\SL_2$ over $\mathbb F_p$ yields a double coset decomposition
\[
\SL_2(\Q_p) = \bigsqcup_{n\in \Z} \Gamma \alpha_n \Gamma \, \sqcup \, \bigsqcup_{m\in \Z} \Gamma \beta_m\Gamma.
\]
Hence, 
\begin{equation}\label{alphap:sums-n-m}
\alpha_p^{\sharp}(\breve{\mathbf h},\breve{\mathbf g},\breve{\pmb{\phi}}) = \sum_{n\in \Z} \overline{\Phi_{\breve{\mathbf h}_p}(\alpha_n)}\Phi_{\breve{\mathbf g}_p}(\alpha_n)\Phi_{\breve{\pmb{\phi}}_p}(\alpha_n)\mathrm{vol}(\Gamma\alpha_n\Gamma) + \sum_{m\in \Z} \overline{\Phi_{\breve{\mathbf h}_p}(\beta_m)}\Phi_{\breve{\mathbf g}_p}(\beta_m)\Phi_{\breve{\pmb{\phi}}_p}(\beta_m)\mathrm{vol}(\Gamma\beta_m\Gamma).
\end{equation}

For later reference, let us also add that the volumes of these double cosets are given by the following formulae:
\begin{equation}\label{volumes-doublecosets}
\mathrm{vol}(\Gamma \alpha_n\Gamma) = \begin{cases}
 p^{2n-2}(p-1) & \text{if } n > 0,\\
 p^{-2n-2}(p-1) & \text{if } n \leq 0,
\end{cases}
\quad 
\mathrm{vol}(\Gamma \beta_m\Gamma) = \begin{cases}
 p^{2m-3}(p-1) & \text{if } m > 0,\\
 p^{-2m-1}(p-1) & \text{if } m \leq 0.
\end{cases}
\end{equation}

Since $p$ divides $M_g$, note that $\breve{\mathbf h}_p = \mathbf h_p \in \tilde{\pi}_p^{\Gamma}$ is a newvector in the one-dimensional subspace $\tilde{\pi}_p^{\Gamma} \subset \tilde{\pi}_p$ of $\Gamma$-invariant vectors, and so the same computation as in \cite[Propositions 7.9, 7.12]{PaldVP} applies for $\Phi_{\breve{\mathbf h}_p}(\alpha_n)$ and $\Phi_{\breve{\mathbf h}_p}(\beta_m)$. Similarly, the values $\Phi_{\breve{\pmb{\phi}}_p}(\alpha_n)$ and $\Phi_{\breve{\pmb{\phi}}_p}(\beta_m)$ were computed in \cite[Proposition 7.13]{PaldVP}. We collect these computation for later reference:

\begin{proposition}\label{prop:MC-PaldVP}
If $p$ divides $M_g$ and $n, m \in \Z$, then 
\[
\Phi_{\breve{\pmb{\phi}}_p}(\alpha_n) = \chi_{\overline{\psi}_p}(p^n)p^{-|n|/2}, \quad \Phi_{\breve{\pmb{\phi}}_p}(\beta_m) = \chi_{\overline{\psi}_p}(p^m)p^{-|m|/2},
\]
and 
\[
\Phi_{\breve{\mathbf h}_p}(\alpha_n) = (-1)^n\chi_{\overline{\psi}_p^{D}}(p^n)p^{-3|n|/2}, \quad \Phi_{\breve{\pmb{\phi}}_p}(\beta_m) = (-1)^{m+1}\chi_{\overline{\psi}_p^{D}}(p^m)p^{-|3m/2-1|}.
\]
\end{proposition}

It thus remains to compute the normalized matrix coefficients $\Phi_{\breve{\mathbf g}_p}(\alpha_n)$ ($n\in \Z$) and $\Phi_{\breve{\mathbf g}_p}(\beta_m)$ ($m\in \Z$). Notice that, since $||\breve{\mathbf g}_p||^2 = 1$ by Lemma \ref{lemma:norm-gpbreve}, we have $\Phi_{\breve{\mathbf g}_p}(g)=\langle \tau_p(g)\breve{\mathbf g}_p,\breve{\mathbf g}_p\rangle$. Recall that $\tau_p = \pi(\chi,\chi^{-1})$ is the (unramified) induced representation associated with an unramified character $\chi:\Q_p^{\times} \to \C^{\times}$, and that $\breve{\mathbf g}_p = \tau_p(\varpi_p) \mathbf g_p \in \tau_p^{K_0}$ is fixed by $K_0 = K_0(p) \supseteq \Gamma_0(p)=\Gamma$, where $\mathbf g_p \in \tau_p^{\GL_2(\Z_p)}$ is the spherical vector normalized so that $\mathbf g_p(1) = 1$. To simplify the notation, we set $\xi:=\chi(p)\overline{\chi}(p)^{-1}=\chi(p)^2$.

\begin{proposition}\label{mc-alphan}
With the above notation, for all integers $n$ it holds
\[
\Phi_{\breve{\mathbf g}_p}(\alpha_n) = 
\frac{p^{-|n|}}{p+1}\left(\frac{\xi^{|n|}(p\xi-1)+\xi^{-|n|}(\xi-p)}{\xi-1}\right).
\]
\end{proposition}
\begin{proof}
 As noted above, since $||\breve{\mathbf g}_p||^2 = 1$ by Lemma \ref{lemma:norm-gpbreve}, we have
 \[
  \Phi_{\breve{\mathbf g}_p}(\alpha_n) = \langle \tau_p(\alpha_n)\breve{\mathbf g}_p,\breve{\mathbf g}_p\rangle = \langle \tau_p(\alpha_n)\tau_p(\varpi_p) \mathbf g_p, \tau_p(\varpi_p)\mathbf g_p\rangle = \langle \tau_p(\varpi_p^{-1}\alpha_n\varpi_p)\mathbf g_p, \mathbf g_p\rangle.
 \]
 Since $\mathbf g_p$ is the unramified vector normalized so that $||\mathbf g_p||=1$, the right hand side equals $\Phi_{\mathbf g_p}(\varpi_p^{-1}\alpha_n\varpi_p)$. Writing 
 \[
  \varpi_p^{-1}\alpha_n\varpi_p = \alpha_n = p^{-n}\begin{pmatrix} p^{2n} & 0 \\ 0 & 1 \end{pmatrix} \quad \text{if } n \geq 0,
 \]
 and  
 \[
  \varpi_p^{-1}\alpha_n\varpi_p = \alpha_n = p^n\begin{pmatrix} 1 & 0 \\ 0 & p^{-2n} \end{pmatrix} = p^n \begin{pmatrix} 0 & 1 \\ 1 & 0 \end{pmatrix}\begin{pmatrix} p^{-2n} & 0 \\ 0 & 1 \end{pmatrix}\begin{pmatrix} 0 & 1 \\ 1 & 0 \end{pmatrix} \quad \text{if } n < 0,
 \]
 we see that 
 \[
  \Phi_{\breve{\mathbf g}_p}(\alpha_n) = \Phi_{\mathbf g_p}\left(\begin{pmatrix} p^{|2n|} & 0 \\ 0 & 1 \end{pmatrix}\right),
 \]
 and the statement then follows from Macdonald's formula (cf. \cite[Theorem 4.6.6]{Bump}).
\end{proof}

\begin{proposition}
With notation as above, for all integers $m$ it holds
\[
\Phi_{\breve{\mathbf g}_p}(\beta_m) = \frac{p^{-|m-1|}}{p+1}\left(\frac{\xi^{|m-1|}(p\xi-1) + \xi^{-|m-1|}(\xi-p)}{\xi-1}\right).
\]
\end{proposition}
\begin{proof}
 One can argue similarly as in the previous proposition. Indeed, we have 
 \[
  \Phi_{\breve{\mathbf g}_p}(\beta_m) = \langle \tau_p(\beta_m)\breve{\mathbf g}_p,\breve{\mathbf g}_p\rangle = \langle \tau_p(\beta_m)\tau_p(\varpi_p) \mathbf g_p, \tau_p(\varpi_p)\mathbf g_p\rangle = \langle \tau_p(\varpi_p^{-1}s\alpha_m\varpi_p)\mathbf g_p, \mathbf g_p\rangle,
 \]
 where recall that $s=\left(\begin{smallmatrix} 0 & 1 \\ -1 & 0 \end{smallmatrix}\right)$. Now we may observe that 
 \[
   \varpi_p^{-1}s\alpha_m\varpi_p = \begin{pmatrix} 0 & p^{1-m} \\ -p^{m-1} & 0 \end{pmatrix} = p^{1-m} \begin{pmatrix} 0 & 1 \\ -1 & 0 \end{pmatrix} \begin{pmatrix} p^{2m-2} & 0 \\ 0 & 1 \end{pmatrix} \quad \text{if } m \geq 1,
 \]
 and 
 \[
   \varpi_p^{-1}s\alpha_m\varpi_p = \begin{pmatrix} 0 & p^{1-m} \\ -p^{m-1} & 0 \end{pmatrix} = p^{m-1} \begin{pmatrix} p^{2-2m} & 0 \\ 0 & 1 \end{pmatrix}\begin{pmatrix} 0 & 1 \\ -1 & 0 \end{pmatrix}  \quad \text{if } m < 1,
 \]
 to deduce that in both cases 
 \[
  \Phi_{\breve{\mathbf g}_p}(\beta_m) = \Phi_{\mathbf g_p}\left(\begin{pmatrix} p^{2|m-1|} & 0 \\ 0 & 1 \end{pmatrix}\right).
 \]
 The statement now follows again by invoking Macdonald's formula (cf. \cite[Theorem 4.6.6]{Bump}).
\end{proof}

\begin{remark}\label{rmk:alphabeta}
Observe from the previous propositions that for every integer $n$ one has $\Phi_{\breve{\mathbf g}_p}(\beta_n) = \Phi_{\breve{\mathbf g}_p}(\alpha_{n-1})$.
\end{remark}

Having computed the matrix coefficients that were missing for this case, we can finally tackle the computation of the normalized local period. First, we compute the local integral (cf. \eqref{alphap:sums-n-m})
\[
\alpha_p^{\sharp}(\breve{\mathbf h},\breve{\mathbf g},\breve{\pmb{\phi}}) =  \sum_{n\in \Z} \Omega_p(\alpha_n)\mathrm{vol}(\Gamma\alpha_n\Gamma) + 
 \sum_{m\in \Z} \Omega_p(\beta_m)\mathrm{vol}(\Gamma\beta_m\Gamma)
\]
where we abbreviate $\Omega_p(g) := \overline{\Phi_{\breve{\mathbf h}_p}(g)}\Phi_{\breve{\mathbf g}_p}(g)\Phi_{\breve{\pmb{\phi}}_p}(g)$ for $g \in \SL_2(\Q_p)$.

\begin{proposition}\label{prop:case2-localintegral}
Let $p$ be a prime dividing $M_g$. Then the local integral $\alpha_p^{\sharp} (\breve{\mathbf h}, \breve{\mathbf g}, \breve{\pmb{\phi}})$ {\em vanishes} if $w_p = -1$. And if $w_p = 1$, then one has
\[
\alpha_p^{\sharp} (\breve{\mathbf h}, \breve{\mathbf g}, \breve{\pmb{\phi}}) = \frac{2(p-1)^2(p-\xi)(p\xi-1)}{p^2(p+1)(p+\xi)(p\xi+1)}.
\]
\end{proposition}
\begin{proof}
We focus first on the computation of $\Omega_p(\alpha_n)\mathrm{vol}(\Gamma\alpha_n\Gamma)$. From Proposition \ref{prop:MC-PaldVP}, using that $\chi_{\overline{\psi}_p^D} = \chi_{\overline{\psi}_p} \cdot \chi_D$ and $(D,p)_p = w_p$, we have
\[
\overline{\Phi_{\breve{\mathbf h}_p}(\alpha_n)}\Phi_{\breve{\pmb{\phi}}_p}(\alpha_n) = p^{-2|n|}(-1)^nw_p^n.
\]
Since  $\mathrm{vol}(\Gamma\alpha_n\Gamma) = p^{|2n|-2}(p-1)$ (cf. \eqref{volumes-doublecosets}), we get
\[
\Omega_p(\alpha_n)\mathrm{vol}(\Gamma\alpha_n\Gamma) = p^{-2}(p-1)(-1)^nw_p^n\Phi_{\breve{\mathbf g}_p}(\alpha_n).
\]

Now we now look at $\Omega_p(\beta_m)\mathrm{vol}(\Gamma\beta_m\Gamma)$. Similarly as before, from Proposition \ref{prop:MC-PaldVP} one has $\overline{\Phi_{\breve{\mathbf h}_p}(\beta_m)}\Phi_{\breve{\pmb{\phi}}_p}(\beta_m) = p^{-|2m-1|}(-1)^{m+1}w_p^m$, and using that $\mathrm{vol}(\Gamma\beta_m\Gamma) = p^{|2m-1|}p^{-2}(p-1)$ we find
\[
\Omega_p(\beta_m)\mathrm{vol}(\Gamma_0\beta_m\Gamma_0) = p^{-2}(p-1)(-1)^{m+1}w_p^m\Phi_{\breve{\mathbf g}_p}(\beta_m).
\]
Altogether, the above yields (using Remark \ref{rmk:alphabeta})
\[
\alpha_p^{\sharp}(\breve{\mathbf h},\breve{\mathbf g},\breve{\pmb{\phi}}) = \frac{p-1}{p^2}\sum_{n\in \Z} ((-1)^nw_p^n\Phi_{\breve{\mathbf g}_p}(\alpha_n) + (-1)^{n-1}w_p^n\Phi_{\breve{\mathbf g}_p}(\alpha_{n-1})) = (1+w_p)\frac{p-1}{p^2}\sum_{n\in \Z}(-1)^nw_p^n\Phi_{\breve{\mathbf g}_p}(\alpha_n).
\]
Here we see that the desired local integral {\em vanishes if $w_p = -1$}. Assume in the following that $w_p=1$. By using that $\Phi_{\breve{\mathbf g}_p}(\alpha_{-n})=\Phi_{\breve{\mathbf g}_p}(\alpha_n)$ (cf. Proposition \ref{mc-alphan}), we have  
\[
\alpha_p^{\sharp}(\breve{\mathbf h},\breve{\mathbf g},\breve{\pmb{\phi}}) = \frac{2(p-1)}{p^2}\sum_{n\in \Z}(-1)^nw_p^n\Phi_{\breve{\mathbf g}_p}(\alpha_n) = \frac{2(p-1)}{p^2}\left( \Phi_{\breve{\mathbf g}_p}(1) + 2\sum_{n > 0}(-1)^n\Phi_{\breve{\mathbf g}_p}(\alpha_n)\right).
\]
Now, $\Phi_{\breve{\mathbf g}_p}(1)=1$, and by Proposition \ref{mc-alphan}
\[
\sum_{n > 0}(-1)^n\Phi_{\breve{\mathbf g}_p}(\alpha_n) = \frac{p\xi-1}{(p+1)(\xi-1)}\sum_{n>0}(-p^{-1}\xi)^{n} + \frac{\xi-p}{(p+1)(\xi-1)}\sum_{n>0}(-p^{-1}\xi^{-1})^{n}.
\]
These geometric series are computed easily, and one eventually finds
\[
\sum_{n>0}(-p^{-1}\xi)^{n} = \frac{-1}{(p+1)} \cdot \frac{(\xi+p^2\xi^2+p^2\xi+p^2)}{(p+\xi)(p\xi+1)}.
\]
Back to the computation of $\alpha_p^{\sharp}(\breve{\mathbf h} ,\breve{\mathbf g},\breve{\pmb{\phi}})$, the above yields
\[
    \Phi_{\breve{\mathbf g}_p}(1) + 2\sum_{n > 0}(-1)^n\Phi_{\breve{\mathbf g}_p}(\alpha_n) = 1 - \frac{2(\xi+p^2\xi^2+p^2\xi+p^2)}{(p+1)(p+\xi)(p\xi+1)} = \frac{(p-1)(p-\xi)(p\xi-1)}{(p+1)(p+\xi)(p\xi+1)}.
\]
This gives the claimed formula.
\end{proof}

Finally, in the next proposition we bring the local $L$-values into the picture to conclude the computation of the normalized local period $\mathcal I_p^{\sharp}(\breve{\mathbf h},\breve{\mathbf g},\breve{\pmb{\phi}})$:

\begin{proposition}\label{prop:Ip-Mg}
 Let $p$ be a prime dividing $N_f$ but not $N_g$. Then the normalized local period $\mathcal I_p^{\sharp} (\breve{\mathbf h}, \breve{\mathbf g}, \breve{\pmb{\phi}})$ {\em vanishes} if $w_p = -1$. And if $w_p=1$, one has
 \[
  \mathcal I_p^{\sharp} (\breve{\mathbf h}, \breve{\mathbf g}, \breve{\pmb{\phi}}) = \frac{2}{p+1}.
 \]
\end{proposition}
\begin{proof}
The vanishing statement in the case $w_p=-1$ follows from the previous proposition, so we may assume that $w_p = 1$. Let us look at the local $L$-values involved in the definition of $\mathcal I_p^{\sharp}(\breve{\mathbf h}, \breve{\mathbf g}, \breve{\pmb{\phi}})$. On the one hand, using that $\pi_p$ is a special representation and $\tau_p$ is an unramified principal series representation, we have (see \cite[Section 10]{Hida-Galreps} or \cite{GelbartJacquet})
\[
L(1,\pi_p,\mathrm{ad}) = \frac{p^2}{p^2-1} = \frac{p^2}{(p+1)(p-1)}, \quad L(1,\tau_p,\mathrm{ad}) = \frac{p^3}{(p-1)(p-\xi)(p-\xi^{-1})}.
\]
Besides, it is well-known that $L(\pi_p,1/2) = \frac{p}{p+w_p} = \frac{p}{p+1}$, whereas for the triple product $L$-function we have (see \cite[Section 3]{Kudla}, for example)
\[
L(\pi_p\otimes\tau_p\otimes\tau_p,1/2) = \frac{p^4}{(p+w_p)^2(p+w_p\xi)(p+w_p\xi^{-1})} = \frac{p^4}{(p+1)^2(p+\xi)(p+\xi^{-1})}.
\]
Therefore, we have 
\[
L(\pi_p\otimes \mathrm{ad}(\tau_p),1/2) = \frac{L(\pi_p\otimes\tau_p\otimes\tau_p,1/2)}{L(\pi_p,1/2)} = \frac{p^3}{(p+1)(p+\xi)(p+\xi^{-1})},
\]
and as a consequence 
\[
\frac{L(1,\pi_p,\mathrm{ad})L(1,\tau_p,\mathrm{ad})}{L(\pi_p\otimes \mathrm{ad}(\tau_p),1/2)} = \frac{p^5(p+1)(p+\xi)(p+\xi^{-1})}{p^3(p+1)(p-1)^2(p-\xi)(p-\xi^{-1})} =  \frac{p^2(p+\xi)(p\xi+1)}{(p-1)^2(p-\xi)(p\xi-1)}.
\]
Multiplying with the value of $\alpha_p^{\sharp}(\breve{\mathbf h},\breve{\mathbf g},\breve{\pmb{\phi}})$ we get as claimed
\begin{align*}
\mathcal I_p^{\sharp} (\breve{\mathbf h}, \breve{\mathbf g}, \breve{\pmb{\phi}}) = \frac{2(p-1)^2 (p-\xi)(p\xi-1)}{p^2(p+1)(p+\xi)(p\xi+1)} \cdot \frac{p^2(p+\xi)(p\xi+1)}{(p-1)^2(p-\xi)(p\xi-1)} = \frac{2}{p+1}.
\end{align*}
\end{proof}

\subsection{The normalized local period at the archimedean place}

To address the computation of the normalized period $\mathcal I_{\infty}^{\sharp}(\breve{\mathbf h},\breve{\mathbf g},\breve{\pmb{\phi}})$, we follow the approach of Xue \cite{Xue}. We fulfill some details missing in loc. cit. in order to provide an explicit formula. 

In order to lighten the notation, let us write in this paragraph $\psi = \overline{\psi_{\infty}}$, so that $\psi(x) = e^{-2\pi\sqrt{-1}x}$ and $\omega_{\infty} = \omega_{\psi}$. By Iwasawa decomposition, recall that every element $g \in \SL_2(\R)$ can be written as 
\[
 g = \left(\begin{array}{cc} y & 0 \\ 0 & y^{-1}\end{array}\right)\left(\begin{array}{cc} 1 & x \\ 0 & 1\end{array}\right) k
\]
for some $y \in \R_{>0}$, $x \in \R$ and $k \in \mathrm{SO}(2)$. We consider the Haar measure $dg = y^{-2}dxdydk$, where $dx$ and $dy$ are the usual Lebesgue measure on $\R$, and $dk$ is the Haar measure on $\mathrm{SO}(2)$ for which the volume of $\mathrm{SO}(2)$ is $\pi$.

Recall from Section \ref{sec:realplace} that $\tau_{\infty}$ is a discrete series representation of $\PGL_2(\R)$ of weight $\ell+1$, and that $\breve{\mathbf g}_{\infty} \in \tau_{\infty}$ is a lowest weight vector. Similarly, recall that $\tilde{\pi}_{\infty}$ is a discrete series representation of $\widetilde{\SL}_2(\R)$ of lowest $\widetilde{\SO}(2)$-type $k+1/2$, and $\breve{\mathbf h}_{\infty} = \tilde V_+^m\mathbf h_{\infty}$ with $\mathbf h_{\infty}$ a lowest weight vector in $\tilde{\pi}_{\infty}$.

Let $J$ be the Jacobi group, which arises as the semidirect product of $\SL_2$ with the so-called Heisenberg group $H$, and it can be realized as a subgroup of $\mathrm{Sp}_2$ (see \cite[Section 1.1]{BerndtSchmidt}). In explicit terms, elements in $J$ can be written as products  
\[
\left(\begin{array}{cc} a & b \\ c & d\end{array}\right)(\lambda, \mu, \xi) = 
\left(\begin{array}{cccc} a & & b \\ & 1 \\ c & & d \\ & & & 1\end{array}\right)
\left(\begin{array}{cccc} 1 & & & \mu \\ \lambda & 1 & \mu & \xi \\ & & 1 & -\lambda \\ & & & 1\end{array}\right), 
\quad \left(\begin{array}{cc} a & b\\ c & d \end{array}\right) \in \SL_2, (\lambda,\mu,\xi) \in H.
\]

By virtue of \cite[Theorem 7.3.3]{BerndtSchmidt}, $\tilde{\pi}_{\infty} \otimes \omega_{\infty}$ is isomorphic to a discrete series representation $\rho_{\infty}$ of $J(\R)$ of lowest $K$-type $k+1$. In particular, the vector $\mathbf h_{\infty} \otimes \pmb{\phi}_{\infty} \in \tilde{\pi}_{\infty} \otimes \omega_{\infty}$ is then identified under the previous isomorphism with a lowest weight vector in $\rho_{\infty}$, which we shall call $\mathbf J_{\infty} \in \rho_{\infty}$. By an abuse of notation, we will simply write $\mathbf J_{\infty} = \mathbf h_{\infty} \otimes \pmb{\phi}_{\infty}$, keeping in mind that this equality is through the isomorphism between $\tilde{\pi}_{\infty} \otimes \omega_{\infty}$ and $\rho_{\infty}$.

Before entering in the computation of the archimedean period $\mathcal I_{\infty}^{\sharp}(\breve{\mathbf h}, \breve{\mathbf g}, \breve{\pmb{\phi}})$, it will be useful to fix once and for all an explicit model $D(k+1,N_f)$ of the discrete series representation $\rho_{\infty}$, which can be found in \cite[Chapter 3]{BerndtSchmidt}, and to describe its main features. As vector spaces, one has
\[
 D(k+1,N_f) = \bigoplus_{\substack{r,s \geq 0},\\ s \text{ even}} \C \cdot v_{r,s},
\]
and $\mathrm{SO}_2(\R)$ acts on $v_{r,s}$ through the character $u \mapsto u^{k+1+r+s}$. The element $v_{0,0}$ is a lowest weight vector, and $\mathrm{SO}_2(\R)$ acts on the line spanned by $v_{0,0}$ through the character $u \mapsto u^{k+1}$. Let $\mathfrak r$ be the Lie algebra of $J(\R)$, and denote by $\mathfrak r_{\C}$ its complexification. There are certain operators $X_+$, $X_-$, $Y_+$, $Y_-$ acting on $\mathfrak r_{\C}$ (see loc. cit. for the precise definition) satisfying $\mathrm d\rho_{\infty} X_- \mathbf J_{\infty} = \mathrm d\rho_{\infty}Y_- \mathbf J_{\infty} = 0$. The action of these operators on the above model is given by the following recipe:
\begin{align*}
 \mathrm d\rho_{\infty}Y_+ v_{r,s} & = v_{r+1,s}, \quad \mathrm d\rho_{\infty}X_+ v_{r,s} = -\frac{1}{2\pi N_f} v_{r+2,s}, \\
 \mathrm d\rho_{\infty}Y_- v_{r,s} & = -2\pi N_f r v_{r-1,s}, \quad \mathrm d\rho_{\infty}X_- v_{r,s} = \pi N_f r(r-1) v_{r-2,s} - \frac{s}{4}(2k+s-1)v_{r,s-2}.
\end{align*}

The space $D(k+1,N_f)$ is further endowed with an inner product $\langle \, , \, \rangle$, and the vectors $v_{r,s}$ form an orthogonal basis with respect to this inner product. Setting $||v||^2 = \langle v, v \rangle$, from \cite[pages 46, 47]{BerndtSchmidt} we know that 
\begin{equation}\label{recursive-norms}
 ||v_{r,s+2}||^2 = \frac{(s+2)(2k+s+1)}{4}||v_{r,s}||^2, \quad 
 ||v_{r+1,s}||^2 = 2\pi N_f(r+1)||v_{r,s}||^2.
\end{equation}

From now on, we normalize the inner product by requiring that $||v_{2m,0}||^2 = \langle v_{2m,0}, v_{2m,0} \rangle = 1$.

\begin{lemma}\label{lemma:normv}
 With the above notation, if $s$ is an even integer with $2 \leq s \leq 2m$, then 
 \[
  ||v_{2m-s,s}||^2 = (4\pi N_f)^{-s} \prod_{\substack{0\leq j\leq s-2,\\ j \text{ even}}}
  \frac{(j+2)(2k+j+1)}{(2m-j-1)(2m-j)}.
 \]
\end{lemma}
\begin{proof}
 The claimed identity follows by applying recursively the relations in \eqref{recursive-norms}. Indeed, using the first of them one easily gets
 \begin{equation}\label{normv:recursive1}
  ||v_{2m-s,s}||^2 = 4^{-s/2}||v_{2m-s,0}||^2\prod_{\substack{0\leq j\leq s-2,\\ j \text{ even}}} (s-j)(2k+s-j-1) = 2^{-s} ||v_{2m-s,0}||^2 \prod_{\substack{0\leq j\leq s-2,\\ j \text{ even}}} (j+2)(2k+j+1).
 \end{equation}
 In a similar manner, we can now use recursively the second identity in \eqref{recursive-norms} to deduce that
 \begin{align}
 ||v_{2m-s,0}||^2 & = (2\pi N_f)^{-s} \prod_{0\leq i \leq s-1} \frac{1}{2m-s+i+1} ||v_{2m,0}||^2 = (2\pi N_f)^{-s} \prod_{0\leq i \leq s-1} \frac{1}{2m-s+i+1} = \label{normv:recursive2}\\
 & = \prod_{\substack{0\leq j\leq s-2,\\ j \text{ even}}} \frac{1}{(2m-(s-j-2)-1)(2m-(s-j-2))} = \prod_{\substack{0\leq j\leq s-2,\\ j \text{ even}}} \frac{1}{(2m-j-1)(2m-j)}, \nonumber
 \end{align}
 using that $||v_{2m,0}||^2 = 1$ according to our normalization. The statement follows by combining \eqref{normv:recursive1} and \eqref{normv:recursive2}.
\end{proof}

We shall now focus our attention on the space 
\[
 D(k+1,N_f;2m) = \bigoplus_{\substack{r+s = 2m, \\ s \text{ even}}} \C \cdot v_{r,s},
\]
which is the largest subspace of $D(k+1,N_f)$ on which $\SO_2(\R)$ acts through the character $u \mapsto u^{\ell+1}$.

\begin{proposition}\label{prop:v2mhol}
 Up to a scalar, there is a unique non-zero vector $v_{2m}^{\mathrm{hol}} \in D(k+1,N_f;2m)$ such that $\mathrm d\rho_{\infty} X_- v_{2m}^{\mathrm{hol}} = 0$. Such a vector is given, up to scalar, by 
 \[
  v_{2m}^{\mathrm{hol}} = \sum_{\substack{0\leq s\leq 2m,\\ s \text{ even}}} c_{s} v_{2m-s,s}, \quad 
  c_0 = 1, \quad c_{s} = (4\pi N_f)^{s/2} \prod_{\substack{0\leq j \leq s-2, \\ j \text{ even}}} \frac{(2m-j)(2m-j-1)}{(j+2)(2k+j+1)} \quad (s \geq 2).
 \]
\end{proposition}
\begin{proof}
 Let $v_{2m}^{\mathrm{hol}} \in D(k+1,N_f;2m)$ be a putative solution of $\mathrm d\rho_{\infty} X_- v_{2m}^{\mathrm{hol}} = 0$, and let
 \[
  v_{2m}^{\mathrm{hol}} = \sum_{\substack{0\leq s\leq 2m,\\ s \text{ even}}} c_{s} v_{2m-s,s}
 \]
 be its representation in terms of the basis $\{v_{2m-s,s}: 0 \leq s \leq 2m \text{ even}\}$. From the description of $X_-$, 
 \[
  d\rho_{\infty}X_- v_{2m-s,s} = \pi N_f (2m-s)(2m-s-1) v_{2m-s-2,s} - \frac{s}{4}(2k+s-1)v_{2m-s,s-2}.
 \]
 By linearity, we can then write down explicitly $d\rho_{\infty}X_- v_{2m}^{\mathrm{hol}}$ in the form 
 \[
  d\rho_{\infty}X_- v_{2m}^{\mathrm{hol}} = \sum_{\substack{0\leq s\leq 2m,\\ s \text{ even}}} d_s v_{2m-(s+2),s},
 \]
 where we understand that $v_{-2,2m} = 0$. Imposing that $d\rho_{\infty}X_- v_{2m}^{\mathrm{hol}} = 0$ then means that $d_s$ must be zero for all $s$. From the description of $d\rho_{\infty}X_-$, one easily checks that 
 \[
  d_s = \pi N_f(2m-s)(2m-s-1)c_s - \frac{s+2}{4}(2k+s+1)c_{s+2},
 \]
 hence $d_s = 0$ if and only if the recursive formula 
 \[
  c_{s+2} = 4 \pi N_f \frac{(2m-s)(2m-s-1)}{(s+2)(2k+s+1)} c_s
 \]
 holds. In particular, for each non-zero $c_0$ one can solve recursively all the $c_s$ for $2\leq s \leq 2m$ even. Setting $c_0 = 1$, this yields the expression in the statement.
\end{proof}

\begin{corollary}\label{cor:norm-vmhol}
For the vector $v_{2m}^{\mathrm{hol}}$ in the previous proposition, one has 
\[
||v_{2m}^{\mathrm{hol}}||^2 = \sum_{\substack{0\leq s \leq 2m,\\ s \text{ even}}} \prod_{\substack{0\leq j \leq s-2, \\ j \text{ even}}} \frac{(2m-j)(2m-j-1)}{(j+2)(2k+j+1)}.
\]
\end{corollary}
\begin{proof}
With notation as in Lemma \ref{lemma:normv} and Proposition \ref{prop:v2mhol}, observe that  if $s$ is an even integer with $0\leq s \leq 2m$, then $c_2 = ||v_{2m-s,s}||^{-2}(4\pi N_f)^{-s/2}$. Using that the basis $v_{r,s}$ is orthogonal with respect to the inner product, we find out that 
 \[
  ||v_{2m}^{\mathrm{hol}}||^2 =  
  \sum_{\substack{0\leq s \leq 2m,\\ s \text{ even}}} c_s^2 ||v_{2m-s,s}||^2 = 
  \sum_{\substack{0\leq s \leq 2m,\\ s \text{ even}}} (4\pi N_f)^{-s} ||v_{2m-s,s}||^{-2} = \sum_{\substack{0\leq s \leq 2m,\\ s \text{ even}}} \prod_{\substack{0\leq j \leq s-2, \\ j \text{ even}}} \frac{(2m-j)(2m-j-1)}{(j+2)(2k+j+1)}.
 \]
\end{proof}

Now we come back to the isomorphism $\tilde{\pi}_{\infty} \otimes \omega_{\infty} \simeq \rho_{\infty}$, under which we identify $\mathbf J_{\infty} = \mathbf h_{\infty} \otimes \breve{\pmb{\phi}}_{\infty}$. One can check further that $\breve{\mathbf h}_{\infty} \otimes \breve{\pmb{\phi}}_{\infty} = \tilde V_+^m\mathbf h_{\infty}\otimes \breve{\pmb{\phi}}_{\infty}$ is identified with a multiple of $\breve{\mathbf J}_{\infty} := Y_+^{2m} \mathbf J_{\infty}$. Since the local normalized period we want to compute does not depend on replacing $\breve{\mathbf h}_{\infty} \otimes \breve{\pmb{\phi}}_{\infty}$ by a multiple, and $\tilde{\pi}_{\infty} \otimes \omega_{\infty} \simeq \rho_{\infty}$ is an isometry, we may assume that $\breve{\mathbf h}_{\infty} \otimes \breve{\pmb{\phi}}_{\infty}$ is identified exactly with  $\breve{\mathbf J}_{\infty}$, so that we will simply write $\breve{\mathbf J}_{\infty} = \breve{\mathbf h}_{\infty} \otimes \breve{\pmb{\phi}}_{\infty}$. Therefore, 
\[
\alpha_{\infty}^{\sharp}(\breve{\mathbf h},\breve{\mathbf g},\breve{\pmb{\phi}}) = 
\int_{\SL_2(\R)} \hspace{-0.2cm}
\frac{\langle \tau(g)\breve{\mathbf g}_{\infty},\breve{\mathbf g}_{\infty}\rangle}{||\breve{\mathbf g}_{\infty}||^2} 
\frac{\overline{\langle \tilde{\pi}(g)\breve{\mathbf h}_{\infty},\breve{\mathbf h}_{\infty} \rangle}}{||\breve{\mathbf h}_{\infty}||^2}
\frac{\langle \omega_{\infty}(g)\breve{\pmb{\phi}}_{\infty}, \breve{\pmb{\phi}}_{\infty}\rangle}{||\breve{\pmb{\phi}}_{\infty}||^2} dg = \int_{\SL_2(\R)}  \hspace{-0.2cm}
\frac{\langle \tau(g)\breve{\mathbf g}_{\infty},\breve{\mathbf g}_{\infty}\rangle}{||\breve{\mathbf g}_{\infty}||^2} 
\frac{\overline{\langle \rho(g)\breve{\mathbf J}_{\infty}, \breve{\mathbf J}_{\infty} \rangle}}{||\breve{\mathbf J}_{\infty}||^2} dg.
\]
In view of this, we will compute the local integral $\alpha_{\infty}^{\sharp}(\breve{\mathbf h},\breve{\mathbf g},\breve{\pmb{\phi}})$ by actually computing the integral on the right hand side, which we will denote by $\alpha_{\infty}^{\sharp}(\breve{\mathbf g}_{\infty},\breve{ \mathbf J}_{\infty})$.

\begin{proposition}
 With the above notation, we have
 \[
  \alpha_{\infty}^{\sharp}(\breve{\mathbf h},\breve{\mathbf g},\breve{\pmb{\phi}}) = \frac{2 \pi^2}{\ell ||v_{2m}^{\textrm{hol}}||^2}.
 \]
\end{proposition}
\begin{proof}
With respect to the above model, $\tau_{\infty}$ might be realized as a subrepresentation of $\rho_{\infty}|_{\SL_2(\R)}$, spanned by $v_{2m}^{\mathrm{hol}}$, and hence we can assume the inner product for $\tau_{\infty}$ to be given by the restriction of the inner product for $\rho_{\infty}$. Besides, $\alpha_{\infty}^{\sharp}(\breve{\mathbf g}_{\infty},\breve{\mathbf J}_{\infty})$ is invariant when replacing $\breve{\mathbf g}_{\infty}$ and $\breve{\mathbf J}_{\infty}$ by multiples of them, so that we may choose $\breve{\mathbf g}_{\infty} = v_{2m}^{\mathrm{hol}}$ and $\breve{\mathbf J}_{\infty} = v_{2m,0}$. Therefore, 
\[
 \alpha_{\infty}^{\sharp}(\breve{\mathbf g}_{\infty}, \breve{\mathbf J}_{\infty}) =  \int_{\SL_2(\R)} \hspace{-0.2cm} \frac{\langle \tau(g)v_{2m}^{\mathrm{hol}},v_{2m}^{\mathrm{hol}}\rangle}{||v_{2m}^{\mathrm{hol}}||^2} \frac{\overline{\langle \rho(g)v_{2m,0},v_{2m,0} \rangle}}{||v_{2m,0}||^2} dg = \frac{1}{||v_{2m}^{\mathrm{hol}}||^2}\int_{\SL_2(\R)} \langle \tau(g)v_{2m}^{\mathrm{hol}},v_{2m}^{\mathrm{hol}}\rangle\overline{\langle \rho(g)v_{2m,0},v_{2m,0} \rangle} dg,
\]
where we have used that $||v_{2m,0}||^2 = 1$ according to our normalization of the inner product. Now, the orthogonal projection of $v_{2m,0}$ to the line generated by $v_{2m}^{\mathrm{hol}}$ is $\mathrm{pr}_{2m}^{\mathrm{hol}}(v_{2m,0}) = ||v_{2m}^{\mathrm{hol}}||^{-2}v_{2m}^{\mathrm{hol}}$. Therefore, 
\[
 \langle \rho(g)v_{2m,0},v_{2m,0} \rangle = \langle \tau(g)\mathrm{pr}_{2m}^{\mathrm{hol}}(v_{2m,0}),\mathrm{pr}_{2m}^{\mathrm{hol}}(v_{2m,0}) \rangle = \frac{1}{||v_{2m}^{\mathrm{hol}}||^4} \langle \tau(g)v_{2m}^{\mathrm{hol}},v_{2m}^{\mathrm{hol}}\rangle,
\]
and we deduce that 
\[
 \alpha_{\infty}^{\sharp}(\breve{\mathbf g}_{\infty}, \breve{\mathbf J}_{\infty}) = 
 \frac{1}{||v_{2m}^{\mathrm{hol}}||^6}\int_{\SL_2(\R)} |\langle \tau(g)v_{2m}^{\mathrm{hol}},v_{2m}^{\mathrm{hol}}\rangle|^2 dg.
\]

At this point, recall that using Iwasawa decomposition for $\SL_2(\R)$, which tells us that any element $g \in \SL_2(\R)$ is written in the form 
\[
 g = \left(\begin{array}{cc} y & 0 \\ 0 & y^{-1}\end{array}\right) \left(\begin{array}{cc} 1 & x \\ 0 & 1 \end{array}\right) k
\]
for some $y \in \R_{>0}$, $x \in \R$ and $k \in \SO_2(\R)$, we have chosen $dg$ to be the Haar measure $y^{-2} dx dy dk$, where $dx$ and $dy$ are the usual Lebesgue measure on $\R$ and $dk$ is the Haar measure on $\SO_2(\R)$ for which the total volume is $\pi$. Define now
\[
 A^+ := \left\lbrace \left(\begin{array}{cc} e^t & \\ & e^{-t}\end{array}\right): t \geq 0\right\rbrace,
\]
and consider the map 
 \begin{align*}
(\SO_2(\R) \times A^+ \times \SO_2(\R))/\{\pm 1\} \, & \longrightarrow \, \SL_2(\R), \\
\left(k, \left(\begin{array}{cc} e^t & \\ & e^{-t}\end{array}\right), k' \right) & \, \longmapsto \, 
k \left(\begin{array}{cc} e^t & \\ & e^{-t}\end{array}\right) k',
\end{align*}
where on the left hand side $-1 = (-1,1,-1)$. This map is a bijection outside the boundary of $A^+$, by virtue of Cartan decomposition. The product measure $dk dt dk'$ on $\SO_2(\R) \times A^+ \times \SO_2(\R)$, where $dt$ is the Lebesgue measure on $\R$ and both $dk$ and $dk'$ are the Haar measure on $\SO_2(\R)$ for which the total volume is $\pi$, induces a measure on the quotient $(\SO_2(\R) \times A^+ \times \SO_2(\R))/\{\pm 1\}$. Under the above bijection, one deduces (by a similar argument as the one in \cite[Section 12]{IchinoIkeda-periods}) that $dg = 2 \cdot \sinh(2t) dk dt dk'$. On the other hand, it is well-known (cf. \cite{Knapp}) that 
\[
 \left\langle \tau_{\infty}\left(\left(\begin{array}{cc} e^t & \\ & e^{-t}\end{array}\right)\right)v_{2m}^{\mathrm{hol}}, v_{2m}^{\mathrm{hol}} \right \rangle = ||v_{2m}^{\mathrm{hol}}||^2 \mathrm{cosh}(t)^{-(\ell+1)},
\]
and therefore 
\[
 \alpha_{\infty}^{\sharp}(\breve{\mathbf g}_{\infty},\breve{\mathbf J}_{\infty}) = 
 \frac{1}{||v_{2m}^{\mathrm{hol}}||^6}\int_{\SL_2(\R)} |\langle \tau(g)v_{2m}^{\mathrm{hol}},v_{2m}^{\mathrm{hol}}\rangle|^2 dg = 
 \frac{2 \pi^2}{||v_{2m}^{\mathrm{hol}}||^2}\int_0^{\infty} \cosh(t)^{-2(\ell+1)}\sinh(2t) dt = \frac{2 \pi^2}{\ell ||v_{2m}^{\mathrm{hol}}||^2}.
\]
\end{proof}

\begin{proposition}\label{prop:Ireal}
We have $\mathcal I_{\infty}^{\sharp}(\breve{\mathbf h},\breve{\mathbf g},\breve{\pmb{\phi}}) = 
\pi^{2m}2^{2m} C_{\infty}(k,\ell)^{-1}$, where the constant $C_{\infty}(k,\ell) \in \Q^{\times}$ is given (setting $\ell-k = 2m$) by
\[
C_{\infty}(k,\ell) := (2m)!\frac{(\ell+k-1)!(k-1)!}{(2k-1)!(\ell-1)!}\sum_{\substack{0\leq s \leq 2m,\\ s \text{ even}}} \prod_{\substack{0\leq j \leq s-2, \\ j \text{ even}}} \frac{(2m-j)(2m-j-1)}{(j+2)(2k+j+1)}.
\]
\end{proposition}
\begin{proof}
From the previous proposition, we know that $\alpha_{\infty}^{\sharp}(\breve{\mathbf h},\breve{\mathbf g},\breve{\pmb{\phi}}) = 2\pi^2 \ell^{-1}||v_{2m}^{\mathrm{hol}}||^{-2}$. Besides, we have
\begin{align*}
 \frac{L(1,\pi_{\infty},\mathrm{ad})L(1,\tau_{\infty},\mathrm{ad})}{L(\pi_{\infty}\otimes \mathrm{ad}( \tau_{\infty}),1/2)} & = 
 \frac{2(2\pi)^{-\ell-1}\Gamma(\ell+1)\pi^{-1}\Gamma(1)2(2\pi)^{-2k}\Gamma(2k)\pi^{-1}\Gamma(1)}{ 2^2(2\pi)^{-2\ell-1}\Gamma(\ell+k)\Gamma(\ell-k+1)2(2\pi)^{-k}\Gamma(k)} = \\
 & = \frac{2^{1-\ell-2k}\pi^{-\ell-2k-3}\ell!(2k-1)!}{2^{2-2\ell-k}\pi^{-2\ell-k-1}(\ell+k-1)!(\ell-k)!(k-1)!} =  \frac{\pi^{\ell-k-2}2^{\ell-k-1} \ell! (2k-1)! }{(\ell+k-1)!(\ell-k)!(k-1)!}.
\end{align*}
Recalling that $2m = \ell-k$, it follows from the definition of $\mathcal I_{\infty}^{\sharp}(\breve{\mathbf h},\breve{\mathbf g},\breve{\pmb{\phi}})$ that 
\[
\mathcal I_{\infty}^{\sharp}(\breve{\mathbf h},\breve{\mathbf g},\breve{\pmb{\phi}}) = \pi^{2m}2^{2m} \frac{1}{||v_{2m}^{\mathrm{hol}}||^2 \cdot (2m)!} \frac{(2k-1)!(\ell-1)!}{(\ell+k-1)!(k-1)!},
\]
and the claimed expression results from replacing $||v_{2m}^{\mathrm{hol}}||^2$ by its expression computed in Corollary \ref{cor:norm-vmhol}.
\end{proof}

\begin{remark}
 Observe that when $\ell = k$, i.e. $m=0$, the above expression reduces to $\mathcal I_{\infty}^{\sharp}(\breve{\mathbf h},\breve{\mathbf g},\breve{\pmb{\phi}}) = 1$ (when $m=0$, one has $v_{2m}^{\mathrm{hol}} = v_{0,0}$, and the inner product is normalized in this case so that $||v_{0,0}||^2 = 1$), which is coherent with \cite[Proposition 9.4]{PaldVP}.
\end{remark}

\section{Global computations and proof of Theorem \ref{thm:centralvalue}}\label{sec:global}

After the computation of normalized local $\SL_2$-periods, this section is devoted to prove the explicit (global) theta identities that will allow us to conclude the proof of Theorem \ref{thm:centralvalue} following the strategy explained in Section \ref{sec:SL2periods-thm}.

\subsection{An explicit theta identity for the pair \texorpdfstring{$(\GL_2,\mathrm{GO}_{2,2})$}{(GL2,GO(2,2))}}\label{sec:thetaidentityGg}

Let $\tau$ be the automorphic representation of $\GL_2(\A)$ associated with $g$. We can regard $\tau \boxtimes \tau$ as a representation of $\mathrm{GSO}_{2,2}(\A)$, and it extends to a unique automorphic representation $\Upsilon$ of $\mathrm{GO}_{2,2}(\A)$ having a non-zero $\mathrm{O}(V_4')(\A)$-invariant distribution, where $V_4' = \{x \in V_4: \mathrm{tr}(x) = 0\}$. Then, the representations $\tau$ and $\Upsilon$ are in theta correspondence for the pair $(\GL_2,\mathrm{GO}_{2,2})$:
\[
\Theta(\tau) = \Upsilon, \quad \Theta(\Upsilon) = \tau.
\]

As in previous sections, write $\mathbf g \in \tau$ for the adelization of the newform $g$. The cusp form $\mathbf g \otimes \mathbf g \in \tau \boxtimes \tau$ extends to a cusp form $
\mathbf G \in \Upsilon$ on $\mathrm{GO}_{2,2}(\A)$ satisfying $\mathbf G(hh') = \mathbf G(h)$ for all $h \in \mathrm{GO}_{2,2}(\A)$ and $h' \in \mu_2(\A)$, where $\mu_2$ is the subgroup of $\mathrm{O}_{2,2}$ generated by the involution $\ast$ on $V_4$. Observe also that, by construction,
\[
\mathbf G_{|\GL_2\times\GL_2} = \mathbf g \otimes \mathbf g \in \tau \boxtimes \tau.
\]
Associated with $\mathbf g$, we define a Bruhat--Schwartz function $\phi_{\mathbf g} = \otimes_v \phi_{\mathbf g,v} \in \mathcal S(V_4(\A))$ by describing its local components as follows:
\begin{itemize}
    \item[i)] $\phi_{\mathbf g,q} = \mathbf 1_{\mathrm M_2(\Z_q)}$ at all primes $q\nmid N_g$;
    \item[ii)] at primes $p \mid N_g$,
    \[
    \phi_{\mathbf g,p}\left(\left(\begin{smallmatrix} x_1 & x_2 \\ x_3 & x_4\end{smallmatrix}\right)\right) = \mathbf 1_{\Z_p}(x_1)\mathbf 1_{\Z_p}(x_4)\mathbf 1_{p\Z_p}(x_3)\left(\mathbf 1_{\Z_p}(x_2)-p^{-1}\mathbf 1_{p^{-1}\Z_p}(x_2)\right);
    \]
    \item[iii)] at the archimedean place, 
    \[
    \phi_{\mathbf g,\infty} \left(\left(\begin{smallmatrix} x_1 & x_2 \\ x_3 & x_4\end{smallmatrix}\right)\right) = (x_1 + \sqrt{-1}x_2 + \sqrt{-1}x_3 -x_4)^{\ell+1}\exp(-\pi \mathrm{tr}(x^tx)). 
    \]
\end{itemize}

By using the rules of the Weil representation of $\widetilde{\SL}_2 \times \mathrm{GO}_{2,2}(\A)$, one can easily check the following:

\begin{lemma}
Let $p$ be a finite prime. Then the following properties hold. 
\begin{itemize}
    \item If $p \nmid N_g$, then $\phi_{\mathbf g,p}$ is fixed by $\SL_2(\Z_p) \subseteq \SL_2(\Q_p)$ and by $\GL_2(\Z_p) \times \GL_2(\Z_p) \subseteq \mathrm{GO}_{2,2}(\Q_p)$.
    \item If $p \mid N_g$, then $\phi_{\mathbf g,p}$ is fixed by $\Gamma_0(p) \subseteq \SL_2(\Z_p)$ and by $K_0(p) \times K_0(p) \subseteq \GL_2(\Z_p) \times \GL_2(\Z_p)$.
\end{itemize}
\end{lemma}

With this, \cite[Corollary 5.4]{PaldVP} shows that
\begin{equation}\label{thetaidentity-G-basic}
\theta(\mathbf G,\phi_{\mathbf g}) = 2^{\ell+1}\zeta_{\Q}(2)^{-2}\mu_{N_g}^{-1}\langle g,g\rangle \mathbf g,
\end{equation}
where $\mu_{N_g}=[\SL_2(\Z):\Gamma_0(N_g)]$. We want to derive an explicit theta identity analogous to \eqref{thetaidentity-G-basic} involving the old form $\breve{\mathbf g}$ instead of $\mathbf g$. To begin with, define $\mathbf g^{\sharp} = \tau(t(2^{-2})_2)\mathbf g \in \tau$, where the element $t(2^{-1})_2$ is concentrated at the place $2$. In parallel, we also define a Bruhat--Schwartz function by $\phi_{\mathbf g^{\sharp}} = 2^{-2}\omega(t(2^{-1})_2,1)\phi_{\mathbf g}$. That is, if $\phi_{\mathbf g^{\sharp}} = \otimes_v \phi_{\mathbf g^{\sharp},v}$, then we keep $\phi_{\mathbf g^{\sharp},v} = \phi_{\mathbf g,v}$ for all $v \neq 2$ and set 
$\phi_{\mathbf g^{\sharp},2} = 2^{-2}\omega_2(t(2^{-1})_2,1)\phi_{\mathbf g,2}$. One can easily check that $\phi_{\mathbf g^{\sharp},2} = \mathbf 1_{V_4(2\Z_2)}$. With this slight modification at the prime $2$, \cite[Corollary 5.5]{PaldVP} shows that
\begin{equation}\label{thetaidentity-G-sharp}
\theta(\mathbf G,\phi_{\mathbf g^{\sharp}}) = 2^{\ell-1}\zeta_{\Q}(2)^{-2}\mu_{N_g}^{-1}\langle g,g\rangle \mathbf g^{\sharp}.
\end{equation}

Next, with this modification at $p=2$ observe  from Section \ref{sec:testvector} that $\breve{\mathbf g}$ is obtained from $\mathbf g^{\sharp}$ by applying the level raising operator on $\tau$ defined by
\[
\mathbf V_{M_g}: \varphi \mapsto \tau(\varpi_{M_g})\varphi,
\]
where $M_g = N_f/N_g$ and $\varpi_{M_g} \in \GL_2(\A)$ is $1$ away from $M_g$, and equals $\varpi_p = \left(\begin{smallmatrix} p^{-1} & 0 \\ 0 & 1 \end{smallmatrix}\right) \in \GL_2(\Q_p)$ at primes $p \mid M_g$, hence $\breve{\mathbf g} = \mathbf V_{M_g} \mathbf g^{\sharp}$. Besides, consider the element $h_p = (1,\varpi_p) \in \GL_2(\Q_p) \times \GL_2(\Q_p)$ for each prime $p \mid M_g$, and identify it with its image $\rho(h_p) \in \mathrm{GSO}_{2,2}(\Q_p) \subseteq \mathrm{GO}_{2,2}(\Q_p)$. Let $\mathbf Y_p$ be the operator acting on $\Upsilon$ by $\Upsilon(h_p)$, and $\mathbf Y_{M_g}$ be defined as the product $\prod_{p\mid M_g} \mathbf Y_p$ (each acting on the corresponding component). Equivalently, we may write $h_{M_g} \in \GL_2(\A)\times\GL_2(\A)$ for the element which is trivial at all places $v \nmid M_g$ and which equals $h_p$ at each prime $p \mid M_g$, and identify it with its image $\rho(h_{M_g}) \in \mathrm{GSO}_{2,2}(\A)$. Then $\mathbf Y_{M_g}$ is the operator acting on $\Upsilon$ by $\Upsilon(h_{M_g})$. With this, consider the automorphic form
\[
\breve{\mathbf G}:=\mathbf Y_{M_g}\mathbf G \in \Upsilon,
\]
and observe that 
\[
\breve{\mathbf G}_{|\GL_2\times \GL_2} = \mathbf g \otimes \mathbf V_{M_g}\mathbf g \in \tau \boxtimes\tau.
\]
For each prime $p \mid M_g$ the automorphic form $\breve{\mathbf G}$ is fixed by the action of $\GL_2(\Z_p) \times K_0(p) \subseteq \mathrm{GO}_{2,2}(\Q_p)$.

Along similar lines, define a Bruhat--Schwartz function $\phi_{\breve{\mathbf g}} = \otimes_v \phi_{\breve{\mathbf g},v}\in \mathcal S(V_4(\A))$ by keeping $\phi_{\breve{\mathbf g},v} = \phi_{\mathbf g^{\sharp},v}$ at places $v \nmid M_g$, and setting
\[
\phi_{\breve{\mathbf g},p} :=  p^{-1}\omega_p(\varpi_p,h_p)\phi_{\mathbf g^{\sharp},p}
\]
at each prime $p \mid M_g$. Note that in this definition we are using the extended Weil representation. Again, a routinary check easily shows the following:

\begin{lemma}
Let $p$ be a prime dividing $M_g$. Then $\phi_{\breve{\mathbf g},p}$ is fixed by the action of $\Gamma_0(p) \subseteq \SL_2(\Q_p)$, and by the action of $\GL_2(\Z_p) \times K_0(p) \subseteq \mathrm{GO}_{2,2}(\Q_p)$.
\end{lemma}
\begin{proof}
The proof just uses the invariance properties of $\phi_{\mathbf g^{\sharp}, p} = \phi_{\mathbf g,p}$, the definition of $\phi_{\breve{\mathbf g},p}$, and the fact that if $\gamma \in K_0(p) \subseteq \GL_2(\Z_p)$ (resp. $\Gamma_0(p) \subseteq \SL_2(\Z_p)$), then $\gamma \varpi_p \in \varpi_p\GL_2(\Z_p)$ (resp. $\varpi_p \SL_2(\Z_p)$).
\end{proof}

With these definitions, the above explicit theta identities recalled from \cite{PaldVP} can be adapted easily to identities relating $\breve{\mathbf g}$ and $\breve{\mathbf G}$ through the theta correspondence with respect to $\phi_{\breve{\mathbf g}}$. Indeed, most importantly for our purposes we have the following:

\begin{proposition}\label{prop:globalG}
With the above notation, 
\begin{equation}\label{thetaidentity-G-breve}
\theta(\breve{\mathbf G},\phi_{\breve{\mathbf g}}) = 2^{\ell-1}M_g^{-1}\mu_{N_g}^{-1}\zeta_{\Q}(2)^{-2}\langle g,g\rangle \breve{\mathbf g}.
\end{equation}
\end{proposition}
\begin{proof}
If $x \in \GL_2(\A)$ and $y'\in \mathrm{GO}_{2,2}(\A)$ is such that $\det(x)=\nu(y')$, then notice that $\det(x\varpi_{M_g}) = \nu(y'h_{M_g})$, hence we can write by applying the definitions
\begin{align*}
    \theta(\breve{\mathbf G},\phi_{\breve{\mathbf g}})(x) & = M_g^{-1} \int_{[\mathrm{O}_{2,2}]} \left(\sum_{v\in V_4(\Q)} \omega(x\varpi_{M_g},y'yh_{M_g})\phi_{\mathbf g^{\sharp}}(v)\right)\mathbf G(y'yh_{M_g})dy = \\
    & = M_g^{-1} \int_{[\mathrm{O}_{2,2}]} \left(\sum_{v\in V_4(\Q)} \omega(x\varpi_{M_g},y'h_{M_g}y)\phi_{\mathbf g^{\sharp}}(v)\right)\mathbf G(y'h_{M_g}y)dy,
\end{align*}
and from this we deduce that $\theta(\breve{\mathbf G},\phi_{\breve{\mathbf g}}) = M_g^{-1}\tau(\varpi_{M_g})\theta(\mathbf G,\phi_{\mathbf g^{\sharp}})$. The statement follows directly from \eqref{thetaidentity-G-sharp}.
\end{proof}

For later use, we compute in the following lemma the precise description of $\phi_{\breve{\mathbf g}}$ at primes $p$ dividing $N_f$. Recall that $N_f = N_gM_g$, and $\mathrm{gcd}(N_g,M_g)=1$.

\begin{lemma}\label{lem:phibreveg}
With the above notation, if $p$ is a prime dividing $N_f = N_g M_g$ we have 
\[
\phi_{\breve{\mathbf g},p}\left(\left(\begin{smallmatrix}x_1& x_2 \\ x_3 & x_4\end{smallmatrix}\right)\right) = \begin{cases}
 \mathbf 1_{\Z_p}(x_1)\mathbf 1_{\Z_p}(x_4)\mathbf 1_{p\Z_p}(x_3)\left(\mathbf 1_{\Z_p}(x_2)-p^{-1}\mathbf 1_{p^{-1}\Z_p}(x_2)\right) & \text{if } p \mid N_g, \\
 \mathbf 1_{p\Z_p}(x_1)\mathbf 1_{\Z_p}(x_2)\mathbf 1_{p\Z_p}(x_3)\mathbf 1_{\Z_p}(x_4) & \text{if } p \mid M_g.
\end{cases}
\]
\end{lemma}
\begin{proof}
The case $p \mid N_g$ was already recalled above. When $p \mid M_g$, one just has to compute 
\[
\phi_{\breve{\mathbf g},p} = p^{-1}\omega_p(\varpi_p,h_p)\phi_{\mathbf g^{\sharp},p} =  p^{-1}\omega_p(\varpi_p,h_p)\phi_{\mathbf g, p} = p^{-1}\omega(\varpi_p,h_p)\mathbf 1_{\mathrm{M}_2(\Z_p)}
\]
using the rules of the (extended) Weil representation. Recall that if $g \in \widetilde{\SL}_2(\Q_p)$ and $h \in \mathrm{GO}_{2,2}(\Q_p)$, then
\[
\omega_p(g,h) \phi = \omega\left(g\left(\begin{smallmatrix} 1 & 0 \\ 0 & \det(g)^{-1}\end{smallmatrix}\right),1\right) L(h)\phi,
\]
where $L(h)\phi(x) =  |\nu(h)|_p^{-1}\phi(h^{-1}\cdot x)$. Applying this for $(g,h)=(\varpi_p,h_p)$, where $h_p = (1,\varpi_p) \in \GL_2(\Q_p)\times \GL_2(\Q_p)$ and $\nu(h_p)=p^{-1}$, one obtains the expression in the statement. 
\end{proof}

\subsection{An explicit theta identity for the pair \texorpdfstring{$(\widetilde{\SL}_2,\mathrm{PGSp}_2)$}{(SL2,PGSp2)}}\label{sec:global-h}

We now focus on an explicit theta identity for the pair $(\widetilde{\SL}_2,\mathrm{PGSp}_2)$. In \cite[Proposition 5.10]{PaldVP} we proved an explicit theta identity relating the adelization $\mathbf h$ of the newform 
\[
h = \sum_{n\geq 1} c(n)q^n \in S_{k+1/2}^+(N_f)
\]
with the adelization $\mathbf F \in \Pi$ of its Saito--Kurokawa lift $F \in S_{k+1}(\Gamma_0^{(2)}(N_f))$. We will now proceed along the same lines to prove an analogous identity between $\breve{\mathbf h}$ and the adelization of $\Delta_{k+1}^m F$. Before doing so, let us first recall some properties concerning the classical forms $h$ and $F$, as well as of their adelizations.

Let $\xi \in \Q_{>0}$, and write $\xi = \mathfrak d_{\xi}\mathfrak f_{\xi}^2$, where $\mathfrak d_{\xi} \in \mathbb N$ is such that $-\mathfrak d_{\xi}$ is the discriminant of $\Q(\sqrt{-\xi})/\Q$ and $\mathfrak f_{\xi} > 0$. If $\xi$ is an integer, then it is well-known that 
\[
c(\xi) = c(\mathfrak d_{\xi}) \sum_{\substack{0<d\mid \mathfrak f_{\xi}, \\ (d,N_f)=1}} \mu(d)\chi_{-\xi}(d)d^{k-1}a_f(\mathfrak f_{\xi}/d),
\]
where we write $a_f(n)$ for the Fourier coefficients of $f$. If $p$ is a prime not dividing $N_f$, we let $\{\alpha_p,\alpha_p^{-1}\}$ be the Satake parameter of $f$ at $p$. If $p$ is a prime dividing $N_f$, we instead define $\alpha_p := p^{1/2-k}a_f(p) = -p^{-1/2}w_p$. 

Besides, writing $\xi = \mathfrak d_{\xi}\mathfrak f_{\xi}^2 \in \Q_{>0}$ as before, let $e_p := \mathrm{val}_p(\xi)$ and define $\Psi_p(\xi;X) \in \C[X, X^{-1}]$ by
\begin{equation}\label{Psip}
\Psi_p(\xi;X) = \begin{cases}
\frac{X^{e_p+1}-X^{-e_p-1}}{X-X^{-1}} - p^{-1/2}\chi_{-\xi}(p)\frac{X^{e_p}-X^{-e_p}}{X-X^{-1}} & \text{if } p \nmid N_f, e_p \geq 0, \\
\chi_{-\xi}(p)(\chi_{-\xi}(p)+w_p)X^{e_p} & \text{if } p \mid N_f, e_p \geq 0, \\
0 & \text{if } e_p < 0.
\end{cases}
\end{equation}
As explained in \cite[Lemma 3.1]{PaldVP}, one has the identity 
\begin{equation}\label{cxi:Psip}
c(\xi) = 2^{-\nu(N)} c(\mathfrak d_{\xi})\mathfrak f_{\xi}^{k-1/2} \prod_p \Psi_p(\xi;\alpha_p),
\end{equation}
where one reads $c(\xi)=0$ if $\xi$ is not an integer. On the other hand, one also has $c(\xi) = e^{2\pi\xi}W_{\mathbf h,\xi}(1)$, where 
\[
W_{\mathbf h,\xi}(g) = \int_{\Q\backslash\A} \mathbf h(u(x)g)\overline{\psi(\xi x)} dx
\]
is the $\xi$-th Fourier coefficient of $\mathbf h$ with respect to the standard additive character $\psi$ of $\A$.

As for the Saito--Kurokawa lift $F = \mathrm{SK}(h) \in S_{k+1}(\Gamma_0^{(2)}(N_f))$ of $h$, its Fourier expansion 
\[
F(Z) = \sum_B A_F(B) e^{2\pi\sqrt{-1}\mathrm{Tr}(BZ)}, \quad Z = X + \sqrt{-1}Y \in \mathcal H_2,
\]
can be explicitly given in terms of the coefficients $c(n)$. Indeed, for each symmetric, half-integral two-by-two matrix $B = \left(\begin{smallmatrix} b_1 & b_2/2\\ b_2/2 & b_3\end{smallmatrix}\right)$ one has
\begin{equation}\label{FCoeff-SK}
A_F(B) = \sum_{\substack{0<d\mid\mathrm{gcd}(b_1,b_2,b_3), \\ (d,N_f)=1}} d^k c(4\xi/d^2),
\end{equation}
where $\xi = \det(B)$. The adelization of $F$ is the automorphic form $\mathbf F: \mathrm{GSp}_2(\A) \to \C$ determined by 
\[
\mathbf F(\gamma g_{\infty} k) = \det(g_{\infty})^{(k+1)/2}\det(C\sqrt{-1}+D)^{-k-1}F(g_{\infty}\sqrt{-1}),
\]
whenever $\gamma \in \mathrm{GSp}_2(\Q)$, $k \in K_0^{(2)}(N_f)$, and $g_{\infty} = \left(\begin{smallmatrix} \ast & \ast \\ C & D \end{smallmatrix}\right) \in \mathrm{GSp}_2^+(\R)$. Here, $K_0^{(2)}(N_f) = \prod_p K_0^{(2)}(N_f;\Z_p)$ with 
\[
K_0^{(2)}(N_f;\Z_p) = \left\lbrace \begin{pmatrix} A & B \\ C & D\end{pmatrix} \in \mathrm{GSp}_2(\Z_p): C \equiv 0 \pmod{ N_f} \right\rbrace
\]

If $B \in \mathrm{Sym}_2(\Q)$ is a two-by-two symmetric matrix, then the $B$-th Fourier coefficient of $\mathbf F$ is defined as the function
\[
\mathcal W_{\mathbf F,B}(h) = \int_{\mathrm{Sym}_2(\Q)\backslash\mathrm{Sym}_2(\A)} \mathbf F(n(X)h)\overline{\psi(\mathrm{Tr}(BX))} dX, \quad h \in \mathrm{GSp}_2(\A).
\]
This Fourier coefficient is determined by its values at elements 
\begin{equation}\label{Whinfty}
h_{\infty} = n(X) m(A,1) = \begin{pmatrix} \mathbf 1_2 & X \\ & \mathbf 1_2 \end{pmatrix} \begin{pmatrix} A \\ & {}^t A^{-1}\end{pmatrix} \in \mathrm{GSp}_2(\R),
\end{equation}
with $X \in \mathrm{Sym}_2(\R)$ and $A \in \GL_2^+(\R)$, and one has
\[
\mathcal W_{\mathbf F,B}(h_{\infty}) = A_F(B)\det(Y)^{(k+1)/2}e^{2\pi\sqrt{-1}\mathrm{Tr}(BZ)},
\]
where $Y = A {}^tA$ and $Z = X+\sqrt{-1}Y \in \mathcal H_2$.

Finally, let $\Delta_{k+1}: S_{k+1}^{nh}(\Gamma_0^{(2)}(N_f)) \to S_{k+3}^{nh}(\Gamma_0^{(2)}(N_f))$ be the Maass differential operator sending nearly holomorphic Siegel forms of weight $k+1$ (and level $\Gamma_0^{(2)}(N_f)$) to nearly holomorphic Siegel forms of weight $k+3$ (and level $\Gamma_0^{(2)}(N_f)$). Writing 
\[
Z = \begin{pmatrix}\tau_1 & z \\ z & \tau_2 \end{pmatrix}, \quad \tau_i = x_i+\sqrt{-1} y_i, \, z = u + \sqrt{-1} v, \quad Z = X + \sqrt{-1} Y,
\]
the Maass differential operator $\Delta_{k+1}$ is defined as (see \cite{MaassBook})
\begin{equation}\label{Maass-Delta}
\Delta_{k+1} = \frac{1}{32\pi^2}\left[\frac{(k+1)(2k+1)}{\det(Y)} - 8 \frac{\partial^2}{\partial\tau_1\partial\tau_2}+2\frac{\partial^2}{\partial^2z}+\frac{2(2k+1)\sqrt{-1}}{\det(Y)}\left(y_1\frac{\partial}{\partial \tau_1}+y_2\frac{\partial}{\partial \tau_2}+v\frac{\partial}{\partial z}\right)\right],
\end{equation}
and $\Delta_{k+1}^mF \in S_{\ell+1}^{nh}(\Gamma_0^{(2)}(N_f))$ has Fourier expansion 
\[
\Delta_{k+1}^mF(Z) = \sum_B A_F(B)C(B,Y)e^{2\pi\sqrt{-1}\mathrm{Tr}(BZ)},
\]
where for each $B$ one has 
\begin{align}\label{const_diff}
C(B,Y) = & \sum_{j=0}^{m} (-4\pi)^{j-m}\frac{\Gamma(\ell-m+\frac{1}{2})}{\Gamma(\ell-2m+j+\frac{1}{2})}\binom{m}{j}\det(B)^{j}\det(Y)^{j-m} \times  \\
& \sum_{i=0}^{m-j}\frac{(2m-2j-i)!}{i!(m-j-i)!}(4\pi)^{i+j-m} \times \sum_{n=0}^{i}\frac{(\ell+1)!(-4\pi)^{-n}}{(\ell+1-n)!} \binom{i}{n} \mathrm{Tr}(BY)^{i-n}. \nonumber
\end{align}

The adelization of $\Delta_{k+1}^m F$ is the automorphic form $\tilde D_+^m\mathbf F$, where $\tilde D_+ = -\frac{1}{64\pi^2} D_+$ for a certain standard weight raising element $D_+ \in \mathcal U(\mathfrak{sp}(2,\R)_{\C})$ (see \cite{PitaleSahaSchmidt}). One defines analogously the $B$-th Fourier coefficients of $\tilde D_+^m\mathbf F$, which are again determined by their values at elements $h_{\infty}$ as before, and one has 
\begin{equation}\label{BFourier-DF}
\mathcal W_{\tilde D_+^m\mathbf F,B}(h_{\infty}) =  A_F(B)C(B,Y)\det(Y)^{(\ell+1)/2}e^{2\pi\sqrt{-1}\mathrm{Tr}(BZ)}.
\end{equation}

Having collected these facts, we now proceed with our main goal of this paragraph. We need to define a Bruhat--Schwartz function $\phi_{\breve{\mathbf h}} \in \mathcal S(V_5(\A))$ with respect to which we will compute the theta lift of $\breve{\mathbf h}$. To do so, we use the same model for $V_5$ as explained above, together with the embedding $V_4 \subset V_5$ obtained by identifying the former with the four-dimensional subspace $\langle v_3\rangle^{\perp}$ of $V_5$. With respect to this embedding, we define the Bruhat--Schwartz function $\phi_{\breve{\mathbf h}}$ as a product of two Bruhat--Schwartz functions, namely 
\[
\phi_{\breve{\mathbf h}} = \phi_{\breve{\mathbf h}}^{(1)}\phi_{\breve{\mathbf h}}^{(4)},
\]
where $\phi_{\breve{\mathbf h}}^{(1)} = \otimes_v \phi_{\breve{\mathbf h},v}^{(1)} \in \mathcal S(\langle v_3\rangle) \simeq \mathcal S(\A)$ is given by 
\[
\phi_{\breve{\mathbf h},v}^{(1)}(x) = \begin{cases}
\mathbf 1_{\Z_q}(x) & \text{if } v = q, \\
e^{-2\pi x^2} & \text{if } v = \infty,
\end{cases}
\]
and $\phi_{\breve{\mathbf h}}^{(4)} = \phi_{\breve{\mathbf g}} \in \mathcal S(V_4)$. In precise terms, with respect to the basis $v_1,\dots,v_5$ of $V_5$, for an arbitrary element $z = x_1v_1 + x_2v_2 + x_3v_3 + x_4v_4 + x_5v_5$, we have \[
\phi_{\breve{\mathbf h}}(z) := \phi_{\breve{\mathbf h}}^{(1)}(x_3)\phi_{\breve{\mathbf h}}^{(4)}\left(\left(\begin{smallmatrix} x_2 & x_1 \\ x_5 & x_4\end{smallmatrix}\right)\right) = \phi_{\breve{\mathbf h}}^{(1)}(x_3)\phi_{\breve{\mathbf g}}\left(\left(\begin{smallmatrix} x_2 & x_1 \\ x_5 & x_4\end{smallmatrix}\right)\right).
\]
Recalling the description of the Bruhat--Schwartz function $\phi_{\breve{\mathbf g}}$ (see Section \ref{sec:thetaidentityGg}, especially Lemma \ref{lem:phibreveg}), the function  $\phi_{\breve{\mathbf h}} = \otimes_v \phi_{\breve{\mathbf h},v}$ is described locally at each place as follows.
\begin{itemize}
    \item[i)] At $v = 2$, 
    \[
    \phi_{\breve{\mathbf h},2}(z) = \mathbf 1_{\Z_2}(x_3)\phi_{\breve{\mathbf g},2}\left(\begin{pmatrix} x_2 & x_1 \\ x_5 & x_4 \end{pmatrix}\right) = \mathbf 1_{2\Z_2}(x_1)\mathbf 1_{2\Z_2}(x_2)\mathbf 1_{\Z_2}(x_3)\mathbf 1_{2\Z_2}(x_4)\mathbf 1_{2\Z_2}(x_5).
    \]
    \item[ii)] If $v = p$ is a prime not dividing $2N_f$, then 
    \[
    \phi_{\breve{\mathbf h},p}(z) = \mathbf 1_{\Z_p}(x_3)\phi_{\breve{\mathbf g},p}\left(\begin{pmatrix} x_2 & x_1 \\ x_5 & x_4 \end{pmatrix}\right) = \mathbf 1_{\Z_p}(x_1)\mathbf 1_{\Z_p}(x_2)\mathbf 1_{\Z_p}(x_3)\mathbf 1_{\Z_p}(x_4)\mathbf 1_{\Z_p}(x_5).
    \]
    \item[iii)] If $v = p$ is a prime dividing $N_g$, then 
    \[
    \phi_{\breve{\mathbf h},p}(z) = \mathbf 1_{\Z_p}(x_3)\phi_{\breve{\mathbf g},p}\left(\begin{pmatrix} x_2 & x_1 \\ x_5 & x_4 \end{pmatrix}\right) = \left(\mathbf 1_{\Z_p}(x_1) - p^{-1}\mathbf 1_{p^{-1}\Z_p}(x_1)\right)\mathbf 1_{\Z_p}(x_2)\mathbf 1_{\Z_p}(x_3)\mathbf 1_{\Z_p}(x_4)\mathbf 1_{p\Z_p}(x_5).
    \]
    \item[iv)] If $v = p$ is a prime dividing $M_g = N_f/N_g$, then
    \[
    \phi_{\breve{\mathbf h},p}(z)  = \mathbf 1_{\Z_p}(x_3)\phi_{\breve{\mathbf g},p}\left(\begin{pmatrix} x_2 & x_1 \\ x_5 & x_4 \end{pmatrix}\right) = \mathbf 1_{\Z_p}(x_1)\mathbf 1_{p\Z_p}(x_2)\mathbf 1_{\Z_p}(x_3)\mathbf 1_{\Z_p}(x_4)\mathbf 1_{p\Z_p}(x_5).
    \]
    \item[v)] If $v = \infty$, then
    \[
    \phi_{\breve{\mathbf h},\infty}(z) = e^{-2\pi x_3^2}\phi_{\breve{\mathbf g},\infty}\left(\begin{pmatrix} x_2 & x_1 \\ x_5 & x_4 \end{pmatrix}\right) = (x_2+\sqrt{-1}x_1+\sqrt{-1}x_5 - x_4)^{\ell+1}\mathrm{exp}(-\pi(x_1^2+x_2^2+2x_3^2+x_4^2+x_5^2)).
    \]
\end{itemize}

At each finite prime $p$, the invariance properties of the Bruhat--Schwartz function $\phi_{\breve{\mathbf h},p}$
with respect to the actions of $\SL_2(\Q_p)$ and $\mathrm{GSp}_2(\Q_p)$ are collected in the following lemma.

\begin{lemma}
Let $p$ be an odd finite prime. Then the following assertions hold.
\begin{itemize}
    \item If $p \nmid N_f$, then $\phi_{\breve{\mathbf h},p}$ is fixed by $\SL_2(\Z_p) \subseteq \SL_2(\Q_p)$ and by $\mathrm{Sp}_2(\Z_p)$.
    \item If $p \mid N_f$, then $\phi_{\breve{\mathbf h},p}$ is fixed by $\Gamma_0(p) \subseteq \SL_2(\Z_p)$ and by $\Gamma_0^{(2)}(p) \subseteq \mathrm{Sp}_2(\Z_p)$.
\end{itemize}
\end{lemma}

By construction, it follows that the theta lift $ \theta(\breve{\mathbf h},\phi_{\breve{\mathbf h}})$ belongs to the space of $K_0^{(2)}(N_f)$-fixed vectors in $\Pi$, and hence it is the adelization of a classical (nearly holomorphic) Siegel modular form in $S_{\ell+1}^{nh}(\Gamma_0^{(2)}(N_f))$. In order to prove a relation between $\theta(\breve{\mathbf h},\phi_{\breve{\mathbf h}})$ and the adelization $\mathbf F$ of the Saito--Kurokawa lift $F$, we will compute the $B$-th Fourier coefficients
\[
h \mapsto \mathcal W_{\theta(\breve{\mathbf h},\phi_{\breve{\mathbf h}}),B}(h) = \int_{\mathrm{Sym}_2(\Q)\backslash\mathrm{Sym}_2(\A)} \theta(\breve{\mathbf h},\phi_{\breve{\mathbf h}})(n(X)h)\overline{\psi(\mathrm{tr}(BX))} dX, \quad h \in \mathrm{GSp}_2(\A),
\]
of the theta lift $\theta(\breve{\mathbf h},\phi_{\breve{\mathbf h}})$, for each positive definite rational symmetric two-by-two matrix
\[
B = \begin{pmatrix} b_1 & b_2/2 \\ b_2/2 & b_3\end{pmatrix} \in \mathrm{Sym}_2(\Q).
\]
The Fourier coefficients $\mathcal W_{\theta(\breve{\mathbf h},\phi_{\breve{\mathbf h}}),B}$ are completely determined by their value at elements $h_{\infty} \in \GSp_2(\R)$ as in \eqref{Whinfty}. Setting $\xi = \det(B)$ and $\beta = (b_3,b_2/2,-b_1)$, it follows from \cite[Lemma 4.2]{Ichino-pullbacks} that
\[
\mathcal W_{\theta(\breve{\mathbf h},\phi_{\breve{\mathbf h}}),B}(h) = \int_{U(\A)\backslash\SL_2(\A)} \hat{\omega} (g,h) \hat{\phi}_{\breve{\mathbf h}} (\beta;0,1) W_{\breve{\mathbf h},\xi}(g) dg,
\]
where 
\[
g \mapsto W_{\breve{\mathbf h},\xi}(g) = \int_{\Q\backslash\A} \breve{\mathbf h}(u(x)g) \overline{\psi(\xi x)} dx
\]
is the $\xi$-th Fourier coefficient of $\breve{\mathbf h}$, $\hat{\phi}_{\breve{\mathbf h}} = \otimes_v \hat{\phi}_{\breve{\mathbf h},v} \in \mathcal S(V_3(\A)) \otimes \mathcal S(\A^2)$ is the Bruhat--Schwartz function obtained from $\phi_{\breve{\mathbf h}}$ by applying a change of polarization, and $\hat{\omega}$ denotes the Weil representation acting on $\mathcal S(V_3(\A)) \otimes \mathcal S(\A^2)$ (by the rule $\hat{\omega}(g,h)\hat{\phi}(x) = (\omega(g,h)\phi)\,\hat{ }$ ).

If $\xi = \det(B) > 0$, we write $\xi = \mathfrak d_{\xi} \mathfrak f_{\xi}^2$ with $\mathfrak f_{\xi} \in \Q_{>0}$ and $\mathfrak d_{\xi} \in \mathbb N$ such that $-\mathfrak d_{\xi}$ is the discriminant of the quadratic field $\Q(\sqrt{-\xi})$. Then we have (compare with \cite[Lemma 5.14]{PaldVP})
\begin{equation}\label{WthetaB:product}
\mathcal W_{\theta(\breve{\mathbf h},\phi_{\breve{\mathbf h}}),B} = \begin{cases}
2^{-\nu(N_f)}c(\mathfrak d_{\xi})\mathfrak f_{\xi}^{k-1/2}\zeta_{\Q}(2)^{-1}\prod_v \mathcal W_{B,v} & \text{if } \xi > 0, \\
0 & \text{if } \xi \leq 0,
\end{cases}
\end{equation}
where the local functions $\mathcal W_{B,v}$ are defined as the integrals
\begin{equation}\label{WBvh}
\mathcal W_{B,v}(h) = \int_{U(\Q_v)\backslash \SL_2(\Q_v)} \hat{\omega}_v(g,h)\hat{\phi}_{\breve{\mathbf h},v}(\beta;0,1)W_{v,\xi}(g)dg \times \begin{cases} 
\mathrm{vol}(\SL_2(\Z_p))^{-1} & \text{if } v = p, \\
\mathrm{vol}(\mathrm{SO}_2)^{-1} & \text{if } v = \infty.
\end{cases}
\end{equation}
Here, for each place $v$ the function $W_{v,\xi}$ is a suitably normalized local Whittaker function associated with $\breve{\mathbf h}$. Namely, at the archimedean place $v = \infty$ we consider 
\begin{equation}\label{Winfty-chi}
 W_{\infty,\xi} = \tilde V_+^m W_{\mathbf h_{\infty},\xi},
\end{equation}
where $W_{\mathbf h_{\infty},\xi}$ is the Whittaker function of $\widetilde{\mathrm{SO}}(2)$-type $k+1/2$ defined by 
\[
W_{\mathbf h_{\infty},\xi}(u(x)t(a)\tilde k_{\theta}) = e^{2\pi\sqrt{-1}\xi x}a^{k+1/2}e^{-2\pi\xi a^2}e^{\sqrt{-1}(k+1/2)\theta}, \qquad x \in \R, a \in \R^{\times}_{>0}, \theta \in \R/4\pi\Z,
\]
where for $\theta \in \R/4\pi\Z$ the elements $k_{\theta}\in \mathrm{SO}(2)$, $\tilde k_{\theta} \in \widetilde{\mathrm{SO}}(2)$ are defined by
\[
k_{\theta} = \begin{pmatrix} \cos\theta & \sin\theta \\ -\sin\theta & \cos\theta\end{pmatrix}, \qquad  \tilde k_{\theta} = \begin{cases}
[k_{\theta},1] & \text{if }  -\pi < \theta \leq \pi,\\
[k_{\theta},-1] & \text{if } \pi < \theta \leq 3\pi.
\end{cases}
\]
Observe that $W_{\mathbf h_{\infty},\xi}(1) = e^{-2\pi \xi}$. And if $v = p$ is a finite prime, then $W_{p,\xi}$ is the non-zero multiple of the local Whittaker function $W_{\breve{\mathbf h}_p,\xi}$ determined by requiring that $W_{p,\xi}(1) = \Psi_p(\xi;\alpha_p)$. That is to say, $W_{p,\xi} := W_{\breve{\mathbf h}_p,\xi}(1)^{-1}\Psi_p(\xi;\alpha_p) \cdot W_{\breve{\mathbf h}_p,\xi}$. In Appendix \ref{whitakkerfunctions} below we recall the definition of the local Whittaker functions $W_{\breve{\mathbf h}_p,\xi}$ and collect some special values of them that will be used in this section.
 
We will determine $\mathcal W_{\theta(\breve{\mathbf h},\phi_{\breve{\mathbf h}}),B}$ by computing via \eqref{WBvh} the local values $\mathcal W_{B,p}(1)$ at all finite places, and the values $\mathcal W_{B,\infty}(h_{\infty})$ at special elements $h_{\infty} \in \mathrm{GSp}_2(\R)$ as in \eqref{Whinfty}. We start dealing with the case of finite places. At rational primes $p \nmid M_g$, the computation of $\mathcal W_{B,p}(1)$ was already carried out in \cite[Section 5]{PaldVP}. With the same notation as before, if $\xi \neq 0$ and $\mu_p = [\SL_2(\Z):\Gamma_0(p)]$, then (check Equations (35), (36), and (37) in loc. cit.):
\begin{equation}\label{WBp1:PaldVP}
\mathcal W_{B,p}(1) = \begin{cases}
 \mathbf 1_{\Z_p}(b_1,b_2,b_3) \sum_{n=0}^{\mathrm{min}(\mathrm{val}_p(b_i))} p^{\frac{n}{2}}\Psi_p(p^{-2n}\xi;\alpha_p) & \text{if } p \nmid 2N_f, \\
 \mathbf 1_{\Z_2}(b_1,b_2,b_3) 2^{\frac{-7}{2}}\sum_{n=0}^{\mathrm{min}(\mathrm{val}_p(b_i))} 2^{\frac{n}{2}}\Psi_2(2^{-2n+2}\xi;\alpha_2) & \text{if } p = 2, \\
 \mathbf 1_{\Z_p}(b_1,b_2,b_3)\mu_p^{-1} \Psi_p(\xi;\alpha_p)& \text{if } p \mid N_g.
\end{cases}
\end{equation}
Thus we assume from now on that $p$ is a prime dividing $M_g = N_f/N_g$, and first study the change of polarization.

\begin{lemma}
Let $p$ be a prime dividing $M_g$, and let $x=(x_1,x_2,x_3) \in V_3(\Q_p)$, and $y = (y_1,y_2) \in \Q_p^2$. Then
\[
\hat{\phi}_{\breve{\mathbf h},p}(x;y) = \phi_{p,1}(x) \cdot \phi_{p,2}(y)
\]
where the functions $\phi_{p,1} \in \mathcal S(V_3(\Q_p))$ and $\phi_{p,2} \in \mathcal S(\Q_p^2)$ are given by
\[
\phi_{p,1}(x) = \mathbf 1_{p\Z_p}(x_1)\mathbf 1_{\Z_p}(x_2)\mathbf 1_{\Z_p}(x_3), \qquad \phi_{p,2}(y) = \mathbf 1_{p\Z_p}(y_1)\mathbf 1_{\Z_p}(y_2).
\]
\end{lemma}
\begin{proof}
This follows straightforward from the definition of the partial Fourier transform,
\[
\hat{\phi}_{\breve{\mathbf h},p}(x;y) = \int_{\Q_p} \phi_{\breve{\mathbf h},p}(z;x;y_1)\psi_p(-y_2z) dz.
\]
\end{proof}

With this change of polarization, one has 
\[
\hat{\omega}_p(g,h)\hat{\phi}_{\breve{\mathbf h},p}(\beta;0,1) = \omega_p(g,h)\phi_{p,1}(\beta) \cdot \phi_{p,2}((0,1)g),
\]
and using equation \eqref{WBvh} we can now compute $\mathcal W_{B,p}(1)$. A first key observation is to use the $\Gamma_0(p)$-invariance of $\breve{\mathbf h}_p$ and of $\hat{\phi}_{\breve{\mathbf h},p}$, together with the fact that $\SL_2(\Q_p)=U(\Q_p)T(\Q_p)\SL_2(\Z_p)$, to rewrite $\mathcal W_{B,p}(1)$ as
\begin{equation}\label{WBp1-sumIBr}
\mathcal W_{B,p}(1) = \mu_p^{-1} \sum_{r\in R_p} \int_{\Q_p^{\times}} \hat{\omega}_p(t(a)r,1)\hat{\phi}_{\breve{\mathbf h},p}(\beta;0,1) W_{p,\xi}(t(a)r)|a|_p^{-2} d^{\times}a = \mu_p^{-1} \sum_{r\in R_p} I(B,r),
\end{equation}
where $R_p$ is a set of representatives for $\SL_2(\Z_p)/\Gamma_0(p)$. We can choose the set $R_p$ to consist of the elements 
\[
r_b = \begin{pmatrix} 1 & 0 \\ b & 1 \end{pmatrix} \text{ with } b \in \{0,1,\dots,p-1\}, \, \text{and } s = \begin{pmatrix} 0 & 1 \\ -1 & 0 \end{pmatrix},
\]
and so we are reduced to compute the values $I(B,1)$, $I(B,s)$, and $I(B,r_b)$ for $b=1,\dots,p-1$. To compute these values, first we point out that the local Whittaker functions $W_{p,\xi}$ satisfy 
\[
W_{p,\xi}(t(a)g) = \chi_{\psi}(a^{-1})\chi_{\delta}(a^{-1})|a|_p^{1/2} W_{p,a^2\xi}(g), \qquad a \in \Q_p^{\times}, \, g \in \SL_2(\Q_p),
\]
and recall also that $\chi_{\psi}(x) = (-1,x)_p\gamma(x,\psi)$. In particular, using this one easily finds that
\[
I(B,r) = \sum_{n\in\Z} p^{n/2}\chi_{\delta}(p^n)\chi_{\psi}(p^n)\int_{\Z_p^{\times}} (a,p^n)_p \omega_p(t(p^na)r,1)\phi_{p,1}(\beta) \phi_{p,2}((0,1)t(p^na)r) W_{p,p^{2n}a^2\xi}(r) d^{\times}a.
\]

We fix $B$ and split the discussion according to whether $r = 1$, $r = s$, or $r = r_b$ for some $b \in \Z_p^{\times}$. To compute the terms $\omega_p(t(p^na)r,1)\phi_{p,1}(\beta)$, we will need to apply the rules for the Weil representation, relative to $V_3$ and $\psi = \psi_p^{-D}$. It is easy to check that $\gamma(\psi,V_3) = 1$ and $\chi_{\psi,V_3}(x) = (-2,x)_p \chi_{\psi}(x)^3$ for $x \in \Q_p^{\times}$ (cf. Section \ref{sec:Weilreps}). In the discussion below, we put $\nu_i = \mathrm{val}_p(b_i)$.

{\bf Case $r=1$.} Start observing that for $n \in \Z$ and $a \in \Z_p^{\times}$ we have $
(0,1)t(p^na) = (0,p^{-n}a^{-1})$, and hence $\phi_{p,2}((0,1)t(p^na)) = 1$ if $n\leq 0$, and vanishes otherwise. Therefore, we obtain 
\begin{align*}
I(B,1) & = \sum_{n\leq 0} p^{n/2}\chi_{\delta}(p^n)\chi_{\psi}(p^n) \int_{\Z_p^{\times}} (a,p^n)_p\omega_p(t(p^na),1)\phi_{p,1}(\beta) W_{p,p^{2n}a^2\xi}(1) d^{\times}a.
\end{align*}
By applying the rules of the Weil representation, one finds
\[
\omega_p(t(p^na),1)\phi_{p,1}(\beta) = (-2,p^n)_p (a,p^n)_p\chi_{\psi}(p^n)^3 p^{-3n/2} \mathbf 1_{\Z_p}(p^nb_1)\mathbf 1_{\Z_p}(p^nb_2)\mathbf 1_{p\Z_p}(p^nb_3).
\]
Therefore, using that $W_{p,p^{2n}a^2\xi}(1) = W_{p,p^{2n}\xi}(1)$ for all $a \in \Z_p^{\times}$ and that $\mathrm{vol}(\Z_p^{\times},d^{\times}a)=1$ (and also that $\chi_{\psi}(p^n)^4=1$),
\begin{align*}
I(B,1) & = \sum_{n\leq 0} p^{-n}\chi_{\delta}(p^n) (-2,p^n)_p \mathbf 1_{\Z_p}(p^nb_1)\mathbf 1_{\Z_p}(p^nb_2)\mathbf 1_{p\Z_p}(p^nb_3) W_{p,p^{2n}\xi}(1) = \\
& = \sum_{n=0}^{\mathrm{min}(\nu_1,\nu_2,\nu_3-1)} p^n \chi_{\delta}(p^n)(-2,p^n)_pW_{p,p^{-2n}\xi}(1).
\end{align*}

{\bf Case $r=s$.} Similarly as before, we now have 
$(0,1)t(p^na)s = (-p^{-n}a^{-1},0)$, thus $\phi_{p,2}((0,1)t(p^na)s) = 1$ if $n\leq -1$, and vanishes otherwise. Therefore, 
\[
I(B,s) = \sum_{n\leq -1} p^{n/2}\chi_{\delta}(p^n)\chi_{\psi}(p^n)\int_{\Z_p^{\times}} (a,p^n)_p \omega_p(t(p^na)s,1)\phi_{p,1}(\beta) W_{p,p^{2n}a^2\xi}(s) d^{\times}a.
\]
Now applying the rules of the Weil representation yields
\[
\omega_p(t(p^na)s,1)\phi_{p,1}(\beta) = (-2,p^n)_p(a,p^n)_p\chi_{\psi}(p^n)^3p^{-3n/2-1}\mathbf 1_{p^{-1}\Z_p}(p^nb_1)\mathbf 1_{\Z_p}(p^nb_2)\mathbf 1_{\Z_p}(p^nb_3).
\]
From this we conclude, using again that $W_{p,p^{2n}a^2\xi}(s)=W_{p,p^{2n}\xi}(s)$ for all $a \in \Z_p^{\times}$, that
\begin{align*}
    I(B,s) & = \sum_{n=1}^{\mathrm{min}(\nu_1+1,\nu_2,\nu_3)}p^{n-1}\chi_{\delta}(p^n)(-2,p^n)_pW_{p,p^{-2n}\xi}(s).
\end{align*}

{\bf Case $r=r_b$, $b \in \Z_p^{\times}$.} We have now $(0,1)t(p^na)r_b = (p^{-n}a^{-1}b,p^{-n}a^{-1})$, and hence $\phi_{p,2}((0,1)t(p^na)r_b) = 1$ if $n\leq -1$, and vanishes otherwise. Thus again we can rewrite 
\[
I(B,r_b) = \sum_{n\leq -1} p^{n/2}\chi_{\delta}(p^n)\chi_{\psi}(p^n)\int_{\Z_p^{\times}} (a,p^n)_p \omega_p(t(p^na)r_b,1)\phi_{p,1}(\beta) W_{p,p^{2n}a^2\xi}(r_b) d^{\times}a.
\]
We can now use that $W_{p,p^{2n}a^2\xi}(r_b) = \psi_p(b^{-1}p^{2n}a^2\xi) W_{p,p^{2n}a^2\xi}(s)$ to rewrite this as
\[
I(B,r_b) = \sum_{n\leq -1} p^{n/2}\chi_{\delta}(p^n)\chi_{\psi}(p^n)\int_{\Z_p^{\times}} (a,p^n)_p\psi_p(b^{-1}p^{2n}a^2\xi) \omega_p(t(p^na)r_b,1)\phi_{p,1}(\beta) W_{p,p^{2n}a^2\xi}(s) d^{\times}a.
\]
To deal with the term $\omega_p(t(p^na)r_b,1)\phi_{p,1}(\beta)$, we first notice that 
\[
r_b = u(b^{-1})s\begin{pmatrix}-b & -1 \\ 0 & -b^{-1}\end{pmatrix}.
\]
The rightmost element belongs to $\Gamma_0(p)$, and hence leaves invariant the function $\phi_{p,1}$. We must therefore compute $\omega_p(t(p^na)u(b^{-1})s)\phi_{p,1}(\beta)$. By applying the rules of the Weil representation, we have
\[
\omega_p(t(p^na)u(b^{-1})s)\phi_{p,1}(\beta) = (-2,p^n)_p(a,p^n)_p\chi_{\psi}(p^n)^3p^{-3n/2-1}\overline{\psi_p(b^{-1}a^2p^{2n}\xi)}\mathbf 1_{p^{-1}\Z_p}(p^nb_1)\mathbf 1_{\Z_p}(p^nb_2)\mathbf 1_{\Z_p}(p^nb_3).
\]
From this, it follows that
\begin{align*}
    I(B,r_b) & = \sum_{n=1}^{\mathrm{min}(\nu_1+1,\nu_2,\nu_3)} p^{n-1}\chi_{\delta}(p^n)(-2,p^n)_pW_{p,p^{-2n}\xi}(s),
\end{align*}
and hence $I(B,r_b) = I(B,s)$, independently on $b$.

Putting all the above discussion together, defining $m(B):=\mathrm{min}(\nu_1+1,\nu_2,\nu_3)$ and $n(B):=\mathrm{min}(\nu_1,\nu_2,\nu_3-1)$ we can rewrite \eqref{WBp1-sumIBr} as
\begin{equation}\label{WBp1=nB+mB}
\mathcal W_{B,p}(1) = \mu_p^{-1} \left(\sum_{n=0}^{n(B)} p^n\chi_{\delta}(p^n)(-2,p^n)_pW_{p,p^{-2n}\xi}(1) + \sum_{n=1}^{m(B)} p^n \chi_{\delta}(p^n)(-2,p^n)_p W_{p,p^{-2n}\xi}(s)\right).
\end{equation}

\begin{proposition}\label{prop:WBp1}
Let $B = \begin{pmatrix} b_1 & b_2/2 \\ b_2/2 & b_3 \end{pmatrix} \in \mathrm{Sym}_2(\Q)$ be a two-by-two symmetric matrix. If $b_1\not\in \Z_p$, or $b_2 \not \in \Z_p$, or $b_3 \not \in p\Z_p$, then $\mathcal W_{B,p}(1) = 0$. Otherwise, let $\xi = \det(B)$ and define integers
\[
m(B) := \mathrm{min}(\nu_1+1,\nu_2,\nu_3), \quad n(B) := \mathrm{min}(\nu_1,\nu_2,\nu_3-1),
\]
where $\nu_i = \mathrm{val}_p(b_i)$. Then $m(B) \geq n(B) \geq 0$, and 
\[
\mathcal W_{B,p}(1) = \mathcal E_p(B) \mu_p^{-1}\Psi_p(\xi;\alpha_p),
\]
where $\mu_p = [\SL_2(\Z):\Gamma_0(p)]$ and
\[
\mathcal E_p(B) = 1+(1-p^{-1})\sum_{n=1}^{n(B)} p^{2n}\chi_{-2\delta}(p)^n +(n(B)-m(B)) p^{2n(B)+1}\chi_{-2\delta}(p)^{n(B)+1}.
\]
\end{proposition}
\begin{proof}
It is clear from \eqref{WBp1=nB+mB} that $\mathcal W_{B,p}(1)=0$ if either $b_1 \not\in \Z_p$, $b_2 \not\in \Z_p$, or $b_3 \not \in p\Z_p$. Let us thus assume that $b_1 \in \Z_p$, $b_2 \in \Z_p$, and $b_3 \in p\Z_p$, and notice that then we have $m(B), n(B) \geq 0$. It is also clear from their definition that $m(B) \geq n(B)$, and one can easily check that $2n(B)\leq \mathrm{val}_p(\xi)$. In addition, when $m(B) > n(B)$ one necessarily has $m(B) = n(B)+1$. In particular, observe that the factor $n(B)-m(B)$ in the definition of $\mathcal E_p(B)$ is either $0$ or $-1$, according to whether $m(B)=n(B)$ or $m(B)>n(B)$, respectively.

From our computations in Appendix \ref{whitakkerfunctions}, for $n = 1, \dots, m(B)$ we have $W_{p,p^{-2n}\xi}(s) = p W_{p,p^{-2(n-1)}\xi}(s)$. Using this and reindexing the second sum in \eqref{WBp1=nB+mB}, we get
\[
\mathcal W_{B,p}(1) = \left(\sum_{n=0}^{n(B)} p^n\chi_{-2\delta}(p)^n W_{p,p^{-2n}\xi}(1) + p^2\chi_{-2\delta}(p)\sum_{n=0}^{m(B)-1} p^{n} \chi_{-2\delta}(p)^nW_{p,p^{-2n}\xi}(s)\right)\mu_p^{-1}.
\]
Now, for $n = 0, \dots, m(B)-1$ we also have $W_{p,p^{-2n}\xi}(s) = -p^{-1}W_{p,p^{-2n}\xi}(1)$, and since $W_{p,p^{-2n}\xi}(1) = \Psi_p(p^{-2n}\xi;\alpha_p) = p^n \Psi_p(\xi;\alpha_p)$, we deduce that
\[
 \mathcal W_{B,p}(1) = \left(\sum_{n=0}^{n(B)} p^{2n}\chi_{-2\delta}(p)^n - p\chi_{-2\delta}(p)\sum_{n=0}^{m(B)-1} p^{2n} \chi_{-2\delta}(p)^n\right)\mu_p^{-1} \Psi_p(\xi;\alpha).
\]
If $m(B) = n(B)$, this reduces to 
\[
 \mathcal W_{B,p}(1) = \left(1+(1-p^{-1})\sum_{n=1}^{n(B)} p^{2n}\chi_{-2\delta}(p)^n\right)\mu_p^{-1}\Psi_p(\xi;\alpha_p),
\]
while if $m(B)>n(B)$, then $m(B) -1 = n(B)$ and the above reads
\[
 \mathcal W_{B,p}(1) = \left(1+(1-p^{-1})\sum_{n=1}^{n(B)} p^{2n}\chi_{-2\delta}(p)^n - p^{2n(B)+1}\chi_{-2\delta}(p)^{n(B)+1}\right)\mu_p^{-1}\Psi_p(\xi;\alpha_p).
\]
\end{proof}

We must emphasize that the quantities $\mathcal E_p(B)$ in the proposition are non-zero rational numbers, and that they depend only on $p$ and $B$, as the notation suggests.

Now we deal with the computation of $\mathcal W_{B,\infty}(h_{\infty})$, for an arbitrary $h_{\infty} = n(X)m(A,1)$ as in \eqref{Whinfty}, with $X \in \mathrm{Sym}_2(\R)$ and $A \in \GL_2^+(\R)$. We recall that by definition
\[
\mathcal W_{B,\infty}(h_{\infty}) = \mathrm{vol}(\mathrm{SO}_2(\R))^{-1}\int_{U(\R)\backslash\SL_2(\R)} \hat{\omega}(g,h)\hat{\phi}_{\breve{\mathbf h},\infty}(\beta;0,1)W_{\infty,\xi}(g) dg,
\]
where $W_{\infty,\xi}$ is the Whittaker function of $\widetilde{\mathrm{SO}}(2)$-type $\ell+1/2$ defined in \eqref{Winfty-chi}. An inductive argument shows that this satisfies
\[
W_{\infty,\xi}(t(a)) = a^{k+1/2}e^{-2\pi\xi a^2}\sum_{j=0}^m (-4\pi)^{j-m}a^{2j}\frac{\Gamma(k+1/2+m)}{\Gamma(k+1/2+j)} \binom{m}{j} \qquad \text{for } a \in \R_{>0}.
\]

\begin{proposition}\label{prop:WBinfty}
With the above notation, one has
\[
\mathcal W_{B,\infty}(n(X)m(A,1)) = \begin{cases}
2^{\ell+1}\det(Y)^{(\ell+1)/2}C(B,Y)e^{2\pi\sqrt{-1}\mathrm{Tr}(BZ)} & \text{if } B > 0,\\
0 & \text{otherwise,}
\end{cases}
\]
where $Y = A {}^tA^{-1}$, $Z = X+\sqrt{-1}Y$, and $C(B,Y)$ is defined as in \eqref{const_diff}.
\end{proposition}
\begin{proof}
This is \cite[Lemma 5.6]{Chen}. 
\end{proof}

We can now finish the computation of $\mathcal W_{\theta(\breve{\mathbf h},\phi_{\breve{\mathbf h}}),B}(h_{\infty})$, where $h_{\infty}=n(X)m(A,1) \in \mathrm{GSp}_2(\R)$ is as in \eqref{Whinfty}. So fix $B \in \mathrm{Sym}_2(\Q)$ be as usual, with entries given by $b_1$, $b_2/2$ and $b_3$, and set $\xi = \det(B)$. If $\xi \leq 0$, the above proposition implies that $\mathcal W_{\theta(\breve{\mathbf h},\phi_{\breve{\mathbf h}}),B}(h_{\infty}) = 0$, so let us assume that $\xi > 0$ and write $\xi = \mathfrak d_{\xi}\mathfrak f_{\xi}^2$ with conventions as above. To simplify the notation in the computation, for each prime $p$ dividing $M_g = N_f/N_g$ we let $\mathcal E_p(B)$ be as in Proposition \ref{prop:WBp1}, and define
\[
\mathcal E(B) := \prod_{p\mid M_g} \mathcal E_p(B).
\]
With this, it follows from \eqref{WthetaB:product}, \eqref{WBp1:PaldVP} and Proposition \ref{prop:WBp1} that $\mathcal W_{\theta(\breve{\mathbf h},\phi_{\breve{\mathbf h}}),B}(h_{\infty}) = 0$ if either $b_1\not\in \Z$, $b_2\not\in \Z$, or $b_3\not\in M_g\Z$. And assuming that $b_1, b_2 \in \Z$ and $b_3 \in M_g\Z$, combining \eqref{WthetaB:product}, \eqref{WBp1:PaldVP}, Proposition \ref{prop:WBp1}, and Proposition \ref{prop:WBinfty} we find
\[
    \mathcal W_{\theta(\breve{\mathbf h},\phi_{\breve{\mathbf h}}),B}(h_{\infty}) = 2^{-\nu(N_f)-7/2}\mu_{N_f}^{-1}\zeta_{\Q}(2)^{-1} c(\mathfrak d_{\xi})\mathfrak f_{\xi}^{k-1/2} \mathcal E(B)\mathcal W_{B,\infty}(h_{\infty})\hspace{-0.1cm} \prod_{p\nmid N_f} \hspace{-0.1cm} \sum_{n=0}^{\mathrm{min}(\nu_i)} p^{\frac{n}{2}}\Psi_p\left(\frac{4\xi}{p^{2n}};\alpha_p\right)\hspace{-0.1cm}\prod_{p\mid N_f} \Psi_p(\xi;\alpha_p),
\]
where we have used that for an odd prime $p \nmid N_f$ one has $\Psi_p(p^{-2n}\xi;\alpha_p) = \Psi_p(4p^{-2n}\xi;\alpha_p)$. We also have $\mathfrak d_{4\xi}=\mathfrak d_{\xi}$ and $\mathfrak f_{\xi}^{k-1/2} = 2^{-k+1/2}\mathfrak f_{4\xi}^{k-1/2}$, hence we can rewrite the above expression as
\[
    2^{-k-3}\mu_{N_f}^{-1}\zeta_{\Q}(2)^{-1}  \mathcal E(B)\mathcal W_{B,\infty}(h_{\infty})
    \sum_{\substack{d\mid(b_1,b_2,b_3),\\(d,N_f)=1}} 2^{-\nu(N_f)}c(\mathfrak d_{4\xi})\mathfrak f_{4\xi}^{k-1/2} d^{1/2}
    \prod_{p\nmid N_f} \Psi_p\left(\frac{4\xi}{d^2};\alpha_p\right)
    \prod_{p\mid N_f} \Psi_p(\xi;\alpha_p).
\]
Now, for each integer $d$ with $(d,N_f)=1$ and each $p \mid N_f$ we have $\Psi_p(\xi;\alpha_p) = \Psi_p(4d^{-2}\xi;\alpha_p)$. We also have $\mathfrak d_{4\xi} = \mathfrak d_{4\xi/d^2}$ and $\mathfrak f_{4\xi}^{k-1/2} = d^{k-1/2}\mathfrak f_{4\xi/d^2}$, and hence using Proposition \ref{prop:WBinfty} and equation \eqref{cxi:Psip} one deduces that
\[
 \mathcal W_{\theta(\breve{\mathbf h},\phi_{\breve{\mathbf h}}),B}(h_{\infty}) =  2^{\ell-k-2} \mu_{N_f}^{-1} \zeta_{\Q}(2)^{-1} \mathcal E(B) \det(Y)^{(\ell+1)/2}C(B,Y)e^{2\pi\sqrt{-1}\mathrm{Tr}(BZ)} \sum_{\substack{d\mid(b_1,b_2,b_3),\\(d,N_f)=1}} d^k c\left(\frac{4\xi}{d^2}\right).
\] 
But now the last sum is exactly the $B$-th Fourier coefficient of $F = \mathrm{SK}(h)$. Comparing with \eqref{BFourier-DF} we obtain
\[
\mathcal W_{\theta(\breve{\mathbf h},\phi_{\breve{\mathbf h}}),B}(h_{\infty}) = 
\begin{cases}
2^{\ell-k-2} \mu_{N_f}^{-1} \zeta_{\Q}(2)^{-1} \mathcal E(B) \mathcal W_{\tilde D_+^m\mathbf F,B}(h_{\infty}) & \text{if } b_3 \in M_g\Z,\\
0 & \text{otherwise,}
\end{cases}
\]
and hence we have proved the following statement:

\begin{proposition}\label{prop:Ftheta}
With the above notation, 
\[
\theta(\breve{\mathbf h},\phi_{\breve{\mathbf h}}) = 2^{2m-2}\mu_{N_f}^{-1}\zeta_{\Q}(2)^{-1} \breve{\mathbf F},
\]
where $\breve{\mathbf F}$ is the adelization of the nearly holomorphic Siegel form $\breve F \in S_{\ell+1}^{nh}(\Gamma_0^{(2)}(N_f))$ whose Fourier coefficients are given by
\[
A_{\breve F}(B) = \begin{cases}
\mathcal E(B) A_{\Delta_{k+1}^m F}(B) & \text{if } b_3 \in M_g\Z,\\
0 & \text{otherwise.}
\end{cases}
\]
\end{proposition}

By recalling that $A_{\Delta_{k+1}^mF}(B) = C(B,Y)A_F(B)$, with $C(B,Y)$ as in \eqref{const_diff}, we see that
\[
A_{\breve F}(B) = \begin{cases}
\mathcal E(B) C(B,Y)A_{F}(B) & \text{if } b_3 \in M_g\Z,\\
0 & \text{otherwise.}
\end{cases}
\]

\subsection{Conclusion of the proof of Theorem \ref{thm:centralvalue}}\label{sec:proof}

We can now finish the proof of Theorem \ref{thm:centralvalue}. Suppose first that $\Lambda(f\otimes\mathrm{Ad}(g),k)\neq 0$, so that $\mathcal Q$ is non-vanishing by Corollary \ref{cor:nonvanishing}. Then we know from  \eqref{centralvalueformula-Q} that 
\begin{equation}\label{cv-mainidentity}
\Lambda(f \otimes \mathrm{Ad}(g), k) = \frac{4 \Lambda(1,\pi,\mathrm{ad})\Lambda(1,\tau,\mathrm{ad})}{\langle \breve{\mathbf h}, \breve{\mathbf h} \rangle\langle \breve{\mathbf g}, \breve{\mathbf g}\rangle\langle \breve{\pmb{\phi}}, \breve{\pmb{\phi}} \rangle} \left( \prod_v \mathcal I_v^{\sharp}(\breve{\mathbf h}, \breve{\mathbf g}, \breve{\pmb{\phi}})^{-1}\right) \mathcal Q(\breve{\mathbf h}, \breve{\mathbf g}, \breve{\pmb{\phi}}),
\end{equation}
where $\breve{\mathbf h} \otimes\breve{\mathbf g}\otimes\breve{\pmb{\phi}} \in \tilde{\pi}\otimes\tau\otimes\omega$ is our test-vector as chosen in Section \ref{sec:testvector}. Now we can compute all the terms on the right hand side. First of all, by Propositions \ref{prop:Ip-PaldVP}, \ref{prop:Ip-Mg}, \ref{prop:Ireal}, we have
\[
\prod_v \mathcal I_v^{\sharp}(\breve{\mathbf h}, \breve{\mathbf g}, \breve{\pmb{\phi}})^{-1} = \pi^{-2m}2^{-2m}C_{\infty}(k,\ell)2^{-\nu(M_g)}N_g\prod_{p\mid M_g}(p+1) = \pi^{-2m}2^{-2m-\nu(M_g)}C_{\infty}(k,\ell)N_g \mu_{M_g}.
\]
Secondly, we have $\langle \breve{\pmb{\phi}},\breve{\pmb{\phi}}\rangle = 2^{-1}$, $\langle \breve{\mathbf g},\breve{\mathbf g}\rangle = \zeta_{\Q}(2)^{-1}\langle g,g\rangle$ and (cf. \cite[page 22]{Waldspurger80})
\begin{align*}
\langle \breve{\mathbf h},\breve{\mathbf h}\rangle & = (4\pi)^{-2m} \frac{\Gamma(k+m+1/2)\Gamma(m+1)}{\Gamma(k+1/2)}\langle \mathbf h, \mathbf h\rangle =  2^{-1}\zeta_{\Q}(2)^{-1} (4\pi)^{-2m} \frac{\Gamma(k+m+1/2)\Gamma(m+1)}{\Gamma(k+1/2)} \langle h, h \rangle \\
& = 2^{-1-2m}\zeta_{\Q}(2)^{-1} (4\pi)^{-2m}\frac{(\ell+k-1)!(k-1)!m!}{(k+m-1)!(2k-1)!} \langle h,h\rangle.
\end{align*}
Besides, by applying \cite[Theorem 5.15]{Hida-MFGalCoh}, \cite[\S 3.2.1]{Watson}, we find
\[
  \Lambda(1,\pi,\mathrm{ad}) = 2^{2k}N_f^{-1}\mu_{N_f}\langle f,f\rangle, \quad \Lambda(1,\tau,\mathrm{ad}) = 2^{\ell+1}N_g^{-1}\mu_{N_g}\langle g,g\rangle,
\] 
and altogether the first term on the right hand side of \eqref{cv-mainidentity} reads
\begin{align*}
 \frac{4\Lambda(1,\pi,\mathrm{ad})\Lambda(1,\tau,\mathrm{ad})}{\langle \breve{\mathbf h}, \breve{\mathbf h}\rangle\langle \breve{\mathbf g},\breve{\mathbf g}\rangle \langle \breve{\pmb{\phi}}, \breve{\pmb{\phi}}\rangle} & = 
 \pi^{2m}\frac{2^{8m+3k+5}\zeta_{\Q}(2)^2\mu_{N_f}\mu_{N_g}}{N_fN_g} \frac{(k+m-1)!(2k-1)!}{(\ell+k-1)!(k-1)!m!} \frac{\langle f,f\rangle}{\langle h,h\rangle}.
\end{align*}

Finally, it remains to compute the value of the global $\SL_2$-period evaluated on our test vector, $\mathcal Q(\breve{\mathbf h},\breve{\mathbf g},\breve{\pmb{\phi}})$. By Proposition \ref{prop:globalG} we have 
\[
\theta(\breve{\mathbf G},\phi_{\breve{\mathbf g}}) = C_1^{-1} \breve{\mathbf g}, \quad C_1 = 2^{1-\ell}M_g\mu_{N_g}\zeta_{\Q}(2)^2 \langle g,g\rangle^{-1},
\]
and hence \cite[Theorem 5.3]{Qiu} implies that $\mathcal Q(\breve{\mathbf h},\breve{\mathbf g},\breve{\pmb{\phi}}) = C_1^2\mathcal P(\theta(\breve{\mathbf h},\phi_{\breve{\mathbf h}}), \breve{\mathbf G})$, where 
\[
\mathcal P: \Pi \otimes \Pi \otimes \Upsilon \otimes \Upsilon \, \longrightarrow \, \C 
\]
is the $\mathrm{SO}(V_4)$-period defined by associating any choice of decomposable vectors $\mathbf F_1, \mathbf F_2 \in \Pi$, $\mathbf G_1, \mathbf G_2 \in \Upsilon$ the product of integrals
\[
\mathcal P(\mathbf F_1, \mathbf F_2, \mathbf G_1, \mathbf G_2) := \left(\int_{[\mathrm{SO}(V_4)]} \mathbf F_1(h)\overline{\mathbf G_1(h)}dh\right)\left(\int_{[\mathrm{SO}(V_4)]} \overline{\mathbf F_2(h)}\mathbf G_2(h)dh\right),
\]
and we abbreviate $\mathcal P(\theta(\breve{\mathbf h},\phi_{\breve{\mathbf h}}), \breve{\mathbf G})= \mathcal P(\theta(\breve{\mathbf h},\phi_{\breve{\mathbf h}}),\theta(\breve{\mathbf h},\phi_{\breve{\mathbf h}}), \breve{\mathbf G},\breve{\mathbf G})$. 
But from Proposition \ref{prop:Ftheta} we know that $\theta(\breve{\mathbf h},\phi_{\breve{\mathbf h}}) = C_2 \breve{\mathbf F}$ with $C_2 = 2^{2m-2}\mu_{N_f}^ {-1}\zeta_{\Q}(2)^{-1}$, hence $\mathcal Q(\breve{\mathbf h},\breve{\mathbf g},\breve{\pmb{\phi}}) = C_1^2C_2^2\mathcal P(\breve{\mathbf F}, \breve{\mathbf G})$. Now, $\breve{\mathbf F}$ is the adelization of $\breve{F}$ and $\breve{\mathbf G}_{|\GL_2\times \GL_2} = \mathbf g \otimes \mathbf V_{M_g}\mathbf g$, which is the adelization of $g \times V_{M_g}g$, hence $\mathcal P(\breve{\mathbf F}, \breve{\mathbf G}) = C_3^2|\langle \breve F_{|\mathcal H \times \mathcal H}, g \times V_{M_g} g\rangle|^2$ with $C_3 = 2^{-1}\zeta_{\Q}(2)^{-2}$ (cf. \cite[Section 9]{IchinoIkeda-periods}). Altogether,
\[
\mathcal Q(\breve{\mathbf h},\breve{\mathbf g},\breve{\pmb{\phi}}) = (C_1C_2C_3)^2 |\langle \breve F_{|\mathcal H \times \mathcal H}, g \times V_M g\rangle|^2 =
2^{4m-2\ell-4} \zeta_{\Q}(2)^{-2} M_g^2 \mu_{N_f}^{-2}\mu_{N_g}^2
\frac{|\langle \breve F_{|\mathcal H \times \mathcal H}, g \times V_{M_g} g\rangle|^2}{\langle g,g\rangle^2}. 
\]
Combining all the terms, we obtain that 
\[
\Lambda(f\otimes \mathrm{Ad}(g),k) = 2^{6m+k+1-\nu(M_g)} \frac{M_g^2\mu_{N_g}^3\mu_{M_g}}{N_f\mu_{N_f}} \frac{(k+m-1)!(2k-1)!}{(\ell+k-1)!(k-1)!m!} C_{\infty}(k,\ell) \cdot \frac{\langle f,f\rangle}{\langle h,h\rangle}\frac{|\langle \breve F_{|\mathcal H \times \mathcal H}, g \times V_{M_g} g\rangle|^2}{\langle g,g\rangle^2}. 
\]
By using the definition of $C_{\infty}(k,\ell)$, we see \[
\frac{(k+m-1)!(2k-1)!}{(\ell+k-1)!(k-1)!m!} C_{\infty}(k,\ell) = \frac{(2m)!}{m!}\frac{(k+m-1)!}{(\ell-1)!} \sum_{\substack{0\leq s \leq 2m,\\ s \text{ even}}} \prod_{\substack{0\leq j \leq s-2, \\ j \text{ even}}} \frac{(2m-j)(2m-j-1)}{(j+2)(2k+j+1)} = C_{\infty}(f,g),
\]
and so the claimed formula in Theorem \ref{thm:centralvalue} follows by noticing that $\mu_{N_f}=\mu_{N_g}\mu_{M_g}$.

If $\Lambda(f\otimes\mathrm{Ad}(g),k) = 0$, then Corollary \ref{cor:nonvanishing} tells us that the functional $\mathcal Q$ is identically zero, and hence the central formula stated in Theorem \ref{thm:centralvalue} holds trivially because all the local periods continue to be non-zero whereas $\mathcal Q(\breve{\mathbf h},\breve{\mathbf g},\breve{\pmb{\phi}}) = 0$, and hence $\langle \breve F_{|\mathcal H\times\mathcal H}, g\times V_{M_g}g\rangle = 0$. \hfill{\ensuremath{\square}}

\vspace{0.2cm}

One can use the explicit formula in Theorem \ref{thm:sc-intro} to prove Deligne's algebraicity conjecture for $\Lambda(f\otimes\mathrm{Ad}(g),k)$:

\begin{corollary}
Let $f \in S_{2k}(N_f)$ and $g \in S_{\ell+1}(N_g)$ be as in Theorem \ref{thm:centralvalue}. If $\sigma \in \mathrm{Aut}(\C)$, then 
\[
\left(\frac{\Lambda(f\otimes\mathrm{Ad}(g),k)}{\langle g,g\rangle^2 c^+(f)}\right)^{\sigma} = \frac{\Lambda(f^{\sigma}\otimes\mathrm{Ad}(g^{\sigma}),k)}{\langle g^{\sigma},g^{\sigma}\rangle^2 c^+(f^{\sigma})},
\]
where $c^+(f)$ is the `plus' period associated with $f$ as in \cite{Shimura77}. In particular, if $\Q(f,g)$ denotes the number field generated by the Fourier coefficients of $f$ and $g$, then 
\[
\Lambda(f\otimes\mathrm{Ad}(g),k)^{\mathrm{alg}} := \frac{\Lambda(f\otimes\mathrm{Ad}(g),k)}{\langle g,g\rangle^2 c^+(f)} \in \Q(f,g).
\]
\end{corollary}

The corollary can be proved along the same lines as \cite[Corollary 6.5]{PaldVP} or \cite[Corollary 8.2]{Chen}, using Kohnen's formula \cite{Kohnen-FC} relating Fourier coefficients of $h$ and central values of twisted $L$-series for $f$ (see also \cite[Theorem A]{ChenCheng} for a different approach). We leave the details for the reader.

\section{Application to subconvexity}\label{sec:subconvexity}

This section is devoted to derive a partial result towards the subconvexity problem stated in \eqref{scproblem} in the Introduction, as a direct consequence of the computation of local $\SL_2$-periods in Section \ref{sec:localperiods} (some of those computations already carried out in \cite{PaldVP}).

As a piece of motivation, let us recall that the subconvexity problems for the families of automorphic $L$-functions 
\begin{equation}\label{sc:pitau}
L(\pi\otimes\tau, s), \qquad \pi \text{ on } \GL_2 \text{ fixed, } \tau \text{ on } \GL_2 \text{ varying,}
\end{equation}
\begin{equation}\label{sc:piadtau}
L(\pi\otimes\mathrm{ad}(\tau), s), \qquad \pi \text{ on } \PGL_2 \text{ fixed, } \tau \text{ on } \GL_2 \text{ varying,}
\end{equation}
are closely related to fundamental arithmetic equidistribution questions. Indeed, the subconvexity problem in \eqref{sc:pitau} is related for instance to the distribution of integral points on spheres, representations of integers by ternary quadratic forms, Heegner points and closed geodesics on modular surfaces, etc. (see, e.g., \cite{IwaniecSarnak-Perspectives}, \cite{MichelVenkatesh-ICM}, \cite{Michel-IAS}). Besides, the subconvexity problem in \eqref{sc:piadtau} has applications towards the limiting mass distribution of automorphic forms, also referred to as the `arithmetic quantum unique ergodicity' (see \cite{SarnakAQC}, \cite{IwaniecSarnak-Perspectives}, \cite{HolowinskySoundarajan}, \cite{NelsonPitaleSaha}, \cite{Sarnak-RP}). This motivated and pushed the efforts to tackle these problems by many authors. The major achievement of these efforts was the solution for the subconvexity problem in \eqref{sc:pitau} given by Michel--Venkatesh in \cite{MichelVenkatesh}, relying crucially on Michel's observation that \eqref{sc:pitau} can be reduced to the subconvexity problem for $L(\pi \otimes \chi, s)$ with $\pi$ on $\GL_2$ fixed and $\chi$ on $\GL_1$ varying (cf. \cite{Michel}).

Concerning the subconvexity problem in \eqref{sc:piadtau}, much less progress has been done until very recently (except for the case where $\pi$ is dihedral, see \cite{Sarnak}). We focus our attention here in the work of Nelson \cite{Nelson}, who reduces the subconvexity problem in \eqref{sc:piadtau}, under important local assumptions, to the subconvexity problem for
\begin{equation}\label{sc:tautauchi}
L(\tau \otimes \tau^{\vee} \otimes \chi, s), \qquad \chi \text{ on } \GL_1 \text{ fixed, } \tau \text{ on } \GL_2 \text{ varying.}
\end{equation}
Thanks to the factorization 
\[
L(\tau \otimes \tau^{\vee} \otimes \chi, s) = L(\chi,s)L(\mathrm{ad}(\tau)\otimes\chi,s),
\]
one can further reduce the subconvexity problem in \eqref{sc:tautauchi} to that for 
\[
L(\mathrm{ad}(\tau)\otimes\chi,s) \qquad \chi \text{ on } \GL_1 \text{ fixed, } \tau \text{ on } \GL_2 \text{ varying.}
\]
Recent work of Munshi \cite{Munshi} on this latter problem, which can be seen as a specialization of \eqref{sc:piadtau} upon restricting $\pi$ to an Eisenstein series, motivates Hypothesis (H) below in Nelson's study of \eqref{sc:piadtau}.

As commented in the Introduction, our contribution to the subconvexity problem in \eqref{sc:piadtau}  relies on providing the bounds for local $\SL_2$-periods required in the main result of Nelson \cite{Nelson}. Let us fix an odd integer $\ell \geq 1$, and let $q$ traverse an infinite sequence $\mathfrak Q$ of (odd) primes. For each prime $q \in \mathfrak Q$, choose an automorphic representation $\tau$ of $\GL_2(\A)$ such that
\begin{center}
    the local component $\tau_q$ is a twist of the special representation,
\end{center}
and let $\mathcal G$ be the infinite family of all such representations $\tau$, when varying $q$ in $\mathfrak Q$. We may also refer to elements in $\mathcal G$ as pairs $(q,\tau)$, in order to keep track of the distinguished prime $q$ of each automorphic representation $\tau$ in the family.

We consider the following hypothesis on the family $\mathcal G$, namely the existence of a subconvex bound for $L(\tau\otimes\tau^{\vee}\otimes\chi,1/2)$ with polynomial dependence upon the character $\chi$:

\vspace{0.2cm}

\noindent {\bf Hypothesis (H):} there exist an absolute constant $\delta_0 = \delta_0(\mathcal G) > 0$ such that for all $\tau \in \mathcal G$ and all unitary characters $\chi$
 of $\A^{\times}/\Q^{\times}$ one has 
 \[
 L(\tau\otimes\tau^{\vee}\otimes\chi,1/2) \ll C(\tau\otimes\tau^{\vee}\otimes\chi)^{1/4-\delta_0}C(\chi)^{O(1)}.
 \]
 
\vspace{0.2cm}
 
 Here, we use the usual `big O' and Vinogradov notation, so that the above hypothesis is equivalent to the existence of absolute constants $c_0, A_0 \geq 0$ and $\delta_0 > 0$ (depending only on the family $\mathcal G$) such that
 \[
  |L(\tau\otimes\tau^{\vee}\otimes\chi,1/2)| \leq c_0 C(\tau\otimes\tau^{\vee}\otimes\chi)^{1/4-\delta_0}C(\chi)^{A_0}
 \]
 for all $\tau \in \mathcal G$ and all unitary characters $\chi$ of $\A^{\times}/\Q^{\times}$.
 
 \begin{theorem}\label{thm:sc}
 Fix an odd integer $\ell \geq 1$. With the above notation, suppose that every $\tau \in \mathcal G$ is the automorphic representation associated with some newform $g \in S_{\ell+1}(N_g)$ of odd squarefree level $N_g$ (and trivial nebentypus), and that $\mathcal G$ satisfies Hypothesis (H). Then, there exists an absolute constant $\delta = \delta(\mathcal G)$ with the following property: if $\pi = \pi(f)$ is an automorphic representation of $\PGL_2(\A)$ associated with a newform $f \in S_{2k}(N_f)$ of weight $2k$, with $1 \leq k \leq \ell$ odd, and odd squarefree level, then 
 \[
 L(\pi\otimes\mathrm{ad}(\tau),1/2) \ll C(\pi\otimes\mathrm{ad}(\tau))^{1/4-\delta}P^{O(1)},
 \]
 where $P = C(\pi)\cdot \prod_{p\neq q} C(\tau_p)$.
 \end{theorem}
\begin{proof}
The theorem follows by checking that the hypotheses of \cite[Theorem 2]{Nelson} hold. Indeed, it is enough to check that for every $\tau = \tau(g) \in \mathcal G$ and every $\pi = \pi(f)$ as in the statement, and every rational prime $p$, either $\pi_p$ is unramified or the conclusion of the conjecture in \cite[\S 2.15.1]{Nelson} is satisfied. If $p$ does not divide $N_f$, then $\pi_p$ is unramified and there is nothing to say. If $p \mid N_f$, then we must prove that (with notations as in Section \ref{sec:localperiods}) there are unit vectors $\varphi_1 \in \tilde{\pi}_p$, $\varphi_2 \in \tau$, $\varphi_3 \in \omega_p$ such that 
\begin{equation}\label{conditions-NelsonConj}
\alpha_p^{\sharp}(\varphi_1,\varphi_2,\varphi_3) (C(\pi_p)C(\tau_p))^{O(1)} \gg 1 \quad \mathrm{ and } \quad \mathcal S(\varphi_i) \ll  (C(\pi_p)C(\tau_p))^{O(1)} \,\, (i=1,2,3),
\end{equation}
where $\alpha_p^{\sharp}(\varphi_1,\varphi_2,\varphi_3)$ is the normalized local integral as defined in Section \ref{sec:SL2periods-thm}, $C(\pi_p)$ and $C(\tau_p)$ denote the analytic conductor of $\pi_p$ and $\tau_p$, respectively, and $\mathcal S(\varphi_i)$ are the Sobolev norms of $\varphi_i$ (on the corresponding representation, in each case), as defined\footnote{One actually defines a family of Sobolev norms $\mathcal S_d$, for each integer $d$, and then $\mathcal S$ denotes a Sobolev norm of the form $\mathcal S_d$ for some fixed large enough $d$ (the ``implied index'').} in \cite[Section 2]{MichelVenkatesh} (see also \cite[Sections 4.6, 5.3]{NelsonSpectral}). We let $\varphi_1$, $\varphi_2$, $\varphi_3$ be the $p$-th components $\breve{\mathbf h}_p$, $\breve{\mathbf g}_p$, $\breve{\pmb{\phi}}_p$ of our test vector chosen in Section \ref{sec:testvector}, normalizing them so that each of the $\varphi_i$ has norm $1$. Since each of the $\varphi_i$ is fixed by $\Gamma_0(p) \subseteq \SL_2(\Z_p)$, it is well--known that the Sobolev norm of $\varphi_i$ satisfies $\mathcal S(\varphi_i) = ||\varphi_i||p^{O(1)} = p^{O(1)}$. Having this into account, we divide the discussion in two cases.

\begin{itemize}
    \item[a)] If $p \mid N_g$, then $C(\pi_p) = C(\tau_p) = p$. Besides, from \cite[Proposition 7.14]{PaldVP} we have 
    \[
    \alpha_p^{\sharp}(\varphi_1,\varphi_2,\varphi_3) = \frac{p-w_p}{p+w_p}\zeta_p(2)^{-1} = \frac{(p-w_p)^2}{p^2},
    \]
    and therefore we clearly see that both conditions in \eqref{conditions-NelsonConj} hold.
    
    \item[b)] If $p \nmid N_g$, we have $C(\pi_p)=p$ and $C(\tau_p) = 1$. In this case, we may invoke instead Proposition \ref{prop:case2-localintegral}, which tells us that
    \[
    \alpha_p^{\sharp}(\varphi_1,\varphi_2,\varphi_3) = \frac{2(p-1)^2(p-\xi)(p\xi-1)}{p^2(p+1)(p+\xi)(p\xi+1)},
    \]
    where $\xi = \chi(p)^2$ with $\chi: \Q_p^{\times} \to \C^{\times}$ the unramified character such that $\tau_p = \pi(\chi,\chi^{-1})$. Again, it follows that both conditions in \eqref{conditions-NelsonConj} are satisfied. 
\end{itemize}
\end{proof}

Some final comments are in order:
 \begin{itemize}
     \item[i)] For a family $\mathcal G$ as in the Introduction, we immediately see that $P = C(\pi)$, and hence Theorem \ref{thm:sc-intro} is just a particular instance of the above statement.
     \item[ii)] If the quantity $P$ in the statement satisfies $\log P \gg \log q$, then the conclusion is worse than the convex bound. So the theorem becomes interesting only under the assumption that 
     \[
     C(\pi) \cdot \prod_{p\neq q} C(\tau_p) = q^{o(1)},
     \]
     where $o(1)$ is a quantity tending to $0$ as $q$ tends to $\infty$. One may hence assume this, which implies in particular that $\pi_q$ is unramified, and that $\tau$ is `essentially unramified away from $q$'.
     \item[iii)] As hinted in the introduction of \cite{Nelson}, modulo the Hypothesis (H) the above theorem should lead to strong quantitative forms of the arithmetic quantum unique ergodicity conjecture in the prime level aspect.
 \end{itemize}

\appendix

\section{Computation of local Whittaker functions at special elements}\label{whitakkerfunctions}

We collect here some special values of Whittaker functions attached to the local components of $\breve{\mathbf h}$, needed in Section \ref{sec:global-h}. If $p$ is a prime, and $\xi \in \Q$, we define
\[
W_{\breve{\mathbf h}_p,\xi}(g) = \int_{\Q_p} \breve{\mathbf h}_p(s^{-1}u(x)g)\overline{\psi_p(\xi x)}dx = \int_{\Q_p}\breve{\mathbf h}_p(s^{-1}u(x)g)\psi_p(-\xi x)dx, \quad g \in \SL_2(\Q_p),
\]
where $s = \left(\begin{smallmatrix}0 & 1 \\ -1 & 0 \end{smallmatrix}\right)$ and $\psi_p$ denotes the standard additive character of $\Q_p$. Assume that $p$ divides $N_f$. By using the definition of $\breve{\mathbf h}_p \in \tilde{\pi}_p$ given in Section  \ref{sec:testvector}, together with the transformation property spelled out in \eqref{transfproperty-hp}, one can prove the following statements. The details are left to the reader.

\begin{proposition}
With the above notation, 
\[
W_{\breve{\mathbf h}_p,\xi}(1) = \begin{cases} p^{-r}(1+(-p\xi,p)_p) & \text{if } \mathrm{val}_p(\xi)=2r, r \geq 0,\\
p^{-r-1}(p+1) & \text{if } \mathrm{val}_p(\xi) = 2r+1, r \geq 0, \\
0 & \text{if } \mathrm{val}_p(\xi) < 0.
\end{cases}
\]
\end{proposition}

\begin{proposition}
With the above notation, 
\[
W_{\breve{\mathbf h}_p,\xi}(s) = \begin{cases}
-p^{-r-1}(1+(-p\xi,p)_p) & \text{if } \mathrm{val}_p(\xi) = 2r, r \geq 0, \\
-p^{-r-2}(p+1) & \text{if } \mathrm{val}_p(\xi) = 2r+1, r \geq -1, \\
0 & \text{if } \mathrm{val}_p(\xi) < -1.
\end{cases}
\]
\end{proposition}

\begin{proposition}
With the above notation, if $b \in \Z_p^{\times}$ and $r_b = \left(\begin{smallmatrix}1 & 0 \\ b & 1\end{smallmatrix}\right)$, then
\[
W_{\breve{\mathbf h}_p,\xi}(r_b) = \begin{cases}
-\psi_p(b^{-1}\xi)p^{-r-1}(1+(-p\xi,p)_p) & \text{if } \mathrm{val}_p(\xi) = 2r, r \geq 0, \\
-\psi_p(b^{-1}\xi)p^{-r-2}(p+1) & \text{if } \mathrm{val}_p(\xi) = 2r+1, r \geq -1, \\
0 & \text{if } \mathrm{val}_p(\xi) < -1.
\end{cases}
\]
\end{proposition}

\bibliographystyle{alpha}

\end{document}